\documentclass[11pt]{article} % use larger type; default would be 10pt
\usepackage[utf8]{inputenc} % set input encoding (not needed with XeLaTeX)

\usepackage[margin=1.25in]{geometry} % to change the page dimensions
\usepackage{graphicx} % support the \includegraphics command and options
\usepackage{booktabs} % for much better looking tables
\usepackage{array} % for better arrays (eg matrices) in maths
\usepackage{paralist} % very flexible & customisable lists (eg. enumerate/itemize, etc.)
\usepackage{verbatim} % adds environment for commenting out blocks of text & for better verbatim
\usepackage{mathrsfs} 
\usepackage{amssymb}
\usepackage{amsthm}
\usepackage{amsmath,amsfonts,amssymb}
\usepackage{esint}
\usepackage{graphics}
\usepackage{enumerate}
\usepackage{mathtools}
\usepackage{xfrac}
\usepackage{bbm}
\usepackage{subcaption}
\usepackage{stmaryrd}

\let\oldsquare\square 

\usepackage{mathabx}

\usepackage{nicefrac}

\renewcommand{\square}{\oldsquare}

 \usepackage{tocloft}
 
\usepackage[usenames,dvipsnames]{xcolor}
\usepackage[colorlinks=true, pdfstartview=FitV, linkcolor=blue, citecolor=blue, urlcolor=blue]{hyperref}
\usepackage[normalem]{ulem}

\numberwithin{equation}{section}

\newtheorem{theorem}{Theorem}[section]

\newtheorem{corollary}[theorem]{Corollary}
\newtheorem{proposition}[theorem]{Proposition}
\newtheorem{lemma}[theorem]{Lemma}
\theoremstyle{definition}

\newtheorem{remark}[theorem]{Remark}

\let\originalleft\left
\let\originalright\right
\renewcommand{\left}{\mathopen{}\mathclose\bgroup\originalleft}
\renewcommand{\right}{\aftergroup\egroup\originalright}

\newcommand{\vertiii}{\vert\kern-0.3ex\vert\kern-0.25ex\vert}

\newcommand{\norm}[1]{\left\|#1\right\|}
\newcommand{\abs}[1]{\left|#1\right|}

\newcommand*{\N}{\ensuremath{\mathbb{N}}}
\newcommand*{\Z}{\ensuremath{\mathbb{Z}}}
\newcommand*{\R}{\ensuremath{\mathbb{R}}}
\newcommand*{\Zd}{\ensuremath{\mathbb{Z}^d}}
\newcommand*{\Rd}{\ensuremath{{\mathbb{R}^d}}}

\renewcommand*{\tilde}{\widetilde}

\newcommand{\ep}{\varepsilon}
\newcommand{\ve}{\varepsilon}
\newcommand{\vp}{\varphi}
\newcommand{\Om}{\Omega}
\newcommand{\La}{\Lambda}
\newcommand{\la}{\lambda}
\newcommand{\Ga}{\Gamma}
\newcommand{\al}{\alpha}
\newcommand{\ga}{\gamma}
\newcommand{\ka}{\kappa}

\DeclareSymbolFont{boldoperators}{OT1}{cmr}{bx}{n}
\SetSymbolFont{boldoperators}{bold}{OT1}{cmr}{bx}{n}
\renewcommand{\a}{\mathbf{A}}
\newcommand{\A}{\mathbf{A}}

\renewcommand{\O}{\mathcal{O}}
\newcommand{\X}{\mathcal{X}}

\newcommand{\Y}{\mathcal{Y}}

\newcommand{\ahom}{\bar{\a}}

\newcommand{\cu}{\square}
\newcommand{\F}{\mathcal{F}}
\renewcommand{\P}{\mathbb{P}}
\newcommand{\E}{\mathbb{E}}

\renewcommand{\O}{\mathcal{O}}

\newcommand{\indc}{1}

\DeclareMathOperator{\dist}{dist}

\DeclareMathOperator*{\osc}{osc}

\DeclareMathOperator{\tr}{tr}

\DeclareMathOperator{\supp}{supp}

\newcommand{\avsum}{\mathop{\mathpalette\avsuminner\relax}\displaylimits}

\makeatletter
\newcommand\avsuminner[2]{%
  {\sbox0{$\m@th#1\sum$}%
   \vphantom{\usebox0}%
   \ooalign{%
     \hidewidth
     \smash{\,\rule[.23em]{8.8pt}{1.1pt} \relax}%
     \hidewidth\cr
     $\m@th#1\sum$\cr
   }
  }
}
\makeatother

\makeatletter
\newcommand\avsuminnerr[2]{%
  {\sbox0{$\m@th#1\sum$}%
   \vphantom{\usebox0}%
   \ooalign{%
     \hidewidth
     \smash{\,\rule[.23em]{6pt}{0.7pt} \relax}%
     \hidewidth\cr
     $\m@th#1\sum$\cr
   }%
  }%
}
\makeatother

\def\Xint#1{\mathchoice
{\XXint\displaystyle\textstyle{#1}}%
{\XXint\textstyle\scriptstyle{#1}}%
{\XXint\scriptstyle\scriptscriptstyle{#1}}%
{\XXint\scriptscriptstyle\scriptscriptstyle{#1}}%
\!\int}
\def\XXint#1#2#3{{\setbox0=\hbox{$#1{#2#3}{\int}$}
\vcenter{\hbox{$#2#3$}}\kern-.5\wd0}}

\def\fint{\Xint-}

\makeatletter 
\newcommand{\negphantom}{\v@true\h@true\negph@nt} 
\newcommand{\neghphantom}{\v@false\h@true\negph@nt} 
\newcommand{\negph@nt}{\ifmmode\expandafter\mathpalette 
  \expandafter\mathnegph@nt\else\expandafter\makenegph@nt\fi} 
\newcommand{\makenegph@nt}[1]{% 
  \setbox\z@\hbox{\color@begingroup#1\color@endgroup}\finnegph@nt} 
\newcommand{\finnegph@nt}{% 
  \setbox\tw@\null 
  \ifv@ \ht\tw@\ht\z@\dp\tw@\dp\z@\fi \ifh@\wd\tw@-\wd\z@\fi\box\tw@} 
\newcommand{\mathnegph@nt}[2]{% 
  \setbox\z@\hbox{$\m@th #1{#2}$}\finnegph@nt} 
\makeatother

\usepackage{titlesec}

\newcommand{\addperiod}[1]{#1.}
\titleformat{\section}
   {\normalfont\Large}{\thesection.}{0.5em}{}
\titleformat*{\subsection}{\normalfont\large}%{\bfseries}
\titleformat{\subsubsection}[runin]
  {\bfseries}
  {\thesubsubsection.}
  {0.5em}
  {\addperiod}
\titleformat*{\subsubsection}{\bfseries}
\titleformat*{\paragraph}{\bfseries}
\titleformat*{\subparagraph}{\large\bfseries}

\title{\bf \Large Green function and invariant measure estimates for nondivergence form elliptic homogenization}

\author{Scott Armstrong
\thanks{Laboratoire Jacques-Louis Lions, Sorbonne Universit\'e 
and
Courant Institute of Mathematical Sciences, New York University.
{\footnotesize \href{mailto:scottnarmstrong@gmail.com}{scottnarmstrong@gmail.com}.}}\and 
Benjamin Fehrman
\thanks{Department of Mathematics, Louisiana State University.
{\footnotesize \href{mailto:fehrman@math.lsu.edu}{fehrman@math.lsu.edu}.}
}
\and
Jessica Lin
\thanks{Department of Mathematics and Statistics, McGill University.
{\footnotesize \href{mailto:jessica.lin@mcgill.ca}{jessica.lin@mcgill.ca}.}
}
}
\date{November 29, 2025}

\usepackage[nottoc,notlot,notlof]{tocbibind}

\begin{document}

\maketitle

\begin{abstract}
We prove quantitative estimates on the parabolic Green function and the stationary invariant measure in the context of stochastic homogenization of elliptic equations in nondivergence form. We consequently obtain a quenched, local CLT for the corresponding diffusion process and a quantitative ergodicity estimate for the environmental process.  Each of these results are characterized by deterministic (in terms of the environment) estimates which are valid above a random, ``minimal'' length scale, the stochastic moments of which we estimate sharply.
\end{abstract}

\setcounter{tocdepth}{1}  
\tableofcontents

\renewcommand{\thefootnote}{\fnsymbol{footnote}} 
\footnotetext{\emph{MSC 2020 subject codes:} 35B27, 60F17, 60K37}     
\renewcommand{\thefootnote}{\arabic{footnote}} 

\renewcommand{\thefootnote}{\fnsymbol{footnote}} 
\renewcommand{\thefootnote}{\arabic{footnote}}

\section{Introduction}
\subsection{Motivation and informal summary of the main results}
We consider the large-scale behavior of solutions of the nondivergence form elliptic equation 
\begin{equation}
\label{e.ellip}
- \tr(\a(x) D^2 u) = 0 \quad \mbox{in} \ U \subseteq \Rd,
\end{equation}
and its parabolic counterpart, namely 
\begin{equation}
\label{e.parab}
\partial_t u - \tr(\a(x) D^2 u) = 0
\quad \mbox{in} \ U_{T} \subseteq \R \times \Rd\,, 
\end{equation}
where $U_{T}:=(0, T]\times U$. Here~$\tr(M)$ denotes the trace of a~$d\times d$ matrix~$M$, the spatial gradient of a real-valued function~$u$ is denoted by~$Du$, and~$D^2u$ is the Hessian of~$u$. The coefficient field~$\a(x)$ is a~$\Zd$--stationary random field taking values in the symmetric matrices satisfying, for given constants~$0<\lambda\leq \Lambda <\infty$, the uniform ellipticity condition
\begin{equation}
\label{e.unifellip}
\lambda I_d \leq \a(x) \leq \Lambda I_d, \quad x\in\Rd\,,
\end{equation}
and, for some exponent~$\alpha_0\in (0,1]$ and constant~$K_0\in [1,\infty)$, the uniform H\"older continuity condition 
\begin{equation}
\label{e.holder}
\bigl[ \a \bigr]_{C^{0,\alpha_0}(\Rd)}:=
\sup_{x,y\in \Rd,\, x\neq y} \frac{|\A(x)-\A(y)|}{|x-y|^{\alpha_0}}\leq K_0. 
\end{equation}
We will suppose that the law of~$\a(\cdot)$ satisfies a strong quantitative decorrelation condition: namely a finite range of dependence. 

\smallskip

The equations~\eqref{e.ellip} and~\eqref{e.parab} arise in the study of diffusion processes in~$\Rd$ in a random environment. Consider a stochastic process~$\{ X_t\}$ evolving according to the SDE
\begin{equation}
\label{e.SDE}
dX_t = \sigma(X_t) \,dW_t
\end{equation}
where~$\{W_t\}_{t>0}$ is a Brownian motion and~$\sigma:\Rd \to \R^{d\times d}$ is a given H\"older continuous function satisfying the uniform nondegeneracy condition
\begin{equation*}
\lambda I_d \leq \frac12 \sigma \sigma^t \leq \Lambda I_d \quad \mbox{in}\ \Rd\,. 
\end{equation*}
If we set~$\a(x):=\frac12\sigma(x)\sigma^t(x)$, then the infinitesimal generator of the Markov process~$\{X_t\}$  is precisely the operator~$\varphi\mapsto \tr(\a(x) D^2\varphi)$, and the parabolic equation~\eqref{e.parab} is the backward Kolmogorov equation. If~$\sigma$ is itself a stationary random field, then~$\{ X_t\}$ is a ``diffusion in a random environment'' and the study of its large-scale behavior is essentially equivalent to that of the large-scale behavior of solutions of~\eqref{e.parab}, which lies in the realm of stochastic homogenization.

Note that the diffusion~\eqref{e.SDE} is \emph{not} reversible, in general. The infinitesimal generator of a reversible diffusion process is a divergence form elliptic operator. The homogenization theory for divergence form equations, though in some ways analogous, is necessarily completely separate from the nondivergence form case. 

\smallskip

The classical qualitative result of stochastic homogenization in the nondivergence form setting was proved more than forty years ago~\cite{PV,Y1}. It roughly states that, on an event of~$\P$--probability one, solutions of~\eqref{e.ellip} and~\eqref{e.parab} converge in the large-scale limit to solutions of the homogenized equations, in which the coefficient~$\a(x)$ is replaced by a constant, deterministic matrix~$\ahom$. This result can be equivalently formulated in terms of the corresponding Markov process~$\{ X_t\}$ as the statement that, on an event of~$\P$--probability one, the rescaled process~$X^\ep_t := \ep X_{t/\ep^2}$ converges in law as~$\ep \to 0$ to a Brownian motion with covariance matrix~$2\ahom$.

\smallskip

In this paper we are interested in quantitative homogenization estimates. That is, we are interested in the speed of convergence of homogenization. Of particular interest is upper bounds on the size of the (random) minimal length scale on which one begins to see the effects of homogenization. 

\smallskip

Quantitative homogenization can be proved by considering the unique (up to normalization) stationary \emph{invariant measure}~$\mu$, which can be written as~$d\mu (x) = m(x) dx$. The density function~$m$ belongs to~$L^1_{\mathrm{loc}}$ and solves the forward Kolmogorov equation (independent of the time variable), which formally is the ``doubly divergence form'' equation
\begin{equation}
\label{e.ddf}
-\sum_{i,j=1}^d 
\partial_{x_i} \partial_{x_j} 
\bigl( \a_{ij}(x) m(x) \bigr) = 0.
\end{equation}
Note that this invariant measure has no explicit form or representation formula. 
Under a qualitative ergodicity assumption, the effective matrix~$\ahom$ can be identified as the ensemble average of~$m\A$, and the ergodic theorem then implies that 
\begin{equation}
\label{e.harpoons}
m\bigl( \tfrac \cdot \ep \bigr)
\rightharpoonup 1
\quad \mbox{and} \quad 
m\bigl( \tfrac \cdot \ep \bigr)
\a\bigl( \tfrac \cdot \ep \bigr)
\rightharpoonup \ahom
\quad 
\mbox{as $\ep\to 0$, \ weakly in~$L^1_\mathrm{loc}$, \ $\P$--a.s.}
\end{equation}
This pair of weak limits implies homogenization by purely deterministic, PDE arguments. 
In probabilistic terms, one can use~\eqref{e.harpoons} to prove the quenched CLT for the process~$X_t^\ep$ directly by applying the martingale CLT and then observing that, thanks to~\eqref{e.harpoons}, the quadratic variation of this process is averaging to~$2 \ahom$. 

\smallskip

The main point of this paper is to obtain quantitative estimates on the ergodicity of the invariant measure~$m$, under the assumption that~$\a$ has a finite range of dependence, and to thereby obtain a quantitative rate of convergence for the weak limits in~\eqref{e.harpoons}. What we prove is that~$m$ decorrelates at least at an algebraic rate, and this translates into an algebraic rate of convergence for the weak limits in~\eqref{e.harpoons}. 
We do not obtain the optimal decay exponents for these estimates, as we are instead focused on optimal stochastic integrability bounds. We show that the probability of ``exceptional events'' in any space domain is at most $O(\exp(-{\text{volume}}))$ (see for example \cite[Remark 4.11]{AK} for a general explanation of this optimality). That is, we sharply bound the probability of large~$O(1)$-type deviations, but do not sharply estimate the size of the typical fluctuation. In Theorem \ref{t.algrate2}, we prove the first quantitative homogenization result for the Cauchy-Dirichlet problem which exhibits this exact optimal stochastic integrability. The study of optimal rates of convergence for~\eqref{e.harpoons} will be the subject of future work, which we expect to use the present work as a starting point for a renormalization-type argument. 

\smallskip

The proof of our main results are based on the close connection between the stationary invariant measure and the parabolic Green function~$P(t,x,y)$ for the Kolmogorov equations. As a function of~$(t,x)$, for each $y\in \R^{d}$, the Green function~$P(\cdot, \cdot, y)$ is the solution of
\begin{equation}\label{e.Pdef}
\begin{cases}
\partial_{t}P(\cdot, \cdot, y)-\tr(\A D^{2}P(\cdot, \cdot, y))=0&\text{in $(0, \infty)\times \R^{d}$},\\
P(0,\cdot,y)=\delta(\cdot-y)&\text{on $\R^{d}$}.
\end{cases}
\end{equation}
Since~\eqref{e.parab} is in nondivergence form, the equation does not preserve mass. 
However, since the equation is homogenizing on large scales to a constant-coefficient equation--which does preserve mass--and the Krylov-Safonov estimate guarantees that~$P(t,\cdot,y)$ must ``spread out'' for large times~$t$, we should expect that~$P(t,\cdot,y)$ has an asymptotic mass as~$t \to \infty$. If we denote this mass by~$m(y)$, then it turns out that~$d\mu(x)=m(x)\,dx$ is the unique stationary invariant measure. That is, we can characterize the invariant measure in terms of the limiting mass of the parabolic Green function. 
Indeed, one of our first main results (Theorem \ref{t.GF}) asserts that $P(t,\cdot, y)$ converges asymptotically as $t\to\infty$ to $m(y)\overline{P}(t,\cdot,y)$, where $\overline{P}$ is the parabolic Green function of the homogenized equation, 
\begin{equation}\label{e.Pbardef}
\begin{cases}
\partial_{t}\overline{P}(\cdot, \cdot-y)-\tr(\overline{\A} D^{2}\overline{P}(\cdot, \cdot-y))=0&\text{in $(0, \infty)\times \R^{d}$},\\
\overline{P}(0,\cdot,y)=\delta(\cdot-y)&\text{on $\R^{d}$}.
\end{cases}
\end{equation}

\smallskip

We expect that quantitative homogenization estimates for the parabolic Green function imply quantitative estimates on the invariant measure~$\mu$. 
Our strategy is to use the quantitative estimates proved in~\cite{asnondiv}, which provide algebraic rates of homogenization with optimal stochastic integrability (in terms of solutions of Dirichlet and Cauchy-Dirichlet problems), to obtain bounds for the speed of homogenization for the parabolic Green function. We then use the above identification to obtain quantitative estimates for the limits in~\eqref{e.harpoons}: see Theorem~\ref{t.mto1}. 
Consequently, we obtain a large-scale regularity estimate for solutions of~\eqref{e.ddf} which implies a quenched Liouville-type result for the stationary invariant measure (see Theorem~\ref{t.largescaleC01}). 

\smallskip

In our last result, we quantify the ergodicity for the so-called \emph{environmental process}, also known as the ``environment from the point of view of the particle.''
We let~$\{ X_t \}$ denote the Feller diffusion process whose generator is the operator~$L$, where $L u  = \tr (\a D^2 u)$, and such that~$X_0=0$. 
Note that~$\{ X_t \}$ is the solution of the SDE
\begin{equation}
\label{e.Feller}
dX_t = \sigma (X_t) \,dW_t
\end{equation}
where~$\sigma(x) : = \sqrt{2 \a(x) }$ and~$\{ W_t \}$ is a Brownian motion on~$\Rd$. We let~$\mathbf{P}$ denote the probability measure associated to~$\{X_t\}$.
We are interested in ergodic properties of the process~$\a(X_t)$.  It is well-known that 
\begin{equation}
\label{e.EPVP}
\frac1t \int_{0}^t  \a(X_s)\,ds \rightarrow \ahom, \quad \mbox{as} \ t \to \infty, \quad \mbox{$\P \otimes \mathbf{P}$--a.s.}
\end{equation}
We introduce a martingale argument which transforms quantitative estimates for the limits in~\eqref{e.harpoons} into a quantitative rate for the limit~\eqref{e.EPVP}. This results in the statement of Theorem~\ref{t.environmental}, below.

\subsection{Statement of the main results}

Throughout the paper, we fix a spatial dimension~$d \in\N$ with~$d\geq 2$, ellipticity constants~$\lambda,\Lambda\in(0,\infty)$ with~$\lambda\leq \Lambda$, a H\"older exponent~$\alpha_0\in (0,1]$ and~$K_0\in [1,\infty)$.
We let~$\mathcal{S}^d$ denote the set of~$d\times d$ symmetric matrices. We define the set~$\Omega$ by  
\begin{equation*}
\Om:=\left\{\text{$\A: \R^{d}\rightarrow \mathcal{S}^{d}$ satisfies \eqref{e.unifellip} and \eqref{e.holder}}\right\}.
\end{equation*}
For every Borel subset~$U\subseteq \R^{d}$, we define the~$\sigma$--algebra~$\mathcal{F}(U)$ by
\begin{equation*}
\mathcal{F}(U)
:=
\text{
the $\sigma$-algebra generated by~$\{ \A \mapsto \A(x) \,:\, x\in U\}$,
}
\end{equation*}
and we denote the largest of these by~$\mathcal{F} := \mathcal{F}(\R^{d})$. 

\smallskip

The law of the coefficient field~$\a(\cdot)$ is prescribed by a probability measure~$\P$ on~$(\Omega,\F)$. The assumptions on~$\P$ are twofold: it is stationary with respect to~$\Zd$ translations, and has a unit range of dependence. Precisely, if we denote the act of translation by~$y\in \R^{d}$ by the operator~$\tau_{y}: \Om\rightarrow \Om$, where 
\begin{equation*}
\tau_{y}\A:=\A(\cdot+y)\,,
\end{equation*}
and extend~$\tau_y$ to events~$E\in\F$ in the obvious way, then the~$\Zd$--stationarity assumption is 
\begin{equation}
\label{e.stationary}
\P = \P \circ \tau_z\,, \quad \forall z\in\Zd\,.
\end{equation}
The unit range of dependence assumption is 
\begin{equation}
\label{e.FRD}
\left\{ 
\begin{aligned}
& 
\mbox{ for all Borel subsets $U, V\subseteq \R^{d}$ with~$\dist(U,V) \geq 1$,}
\\ &
\mbox{ $\F(U)$ and~$\F(V)$ are~$\P$--independent.}
\end{aligned}
\right.
\end{equation}
The hypothesis \eqref{e.FRD} is a stronger, quantitative form of the classical qualitative  ergodicity assumption, which is that translation-invariant events must have probability either zero or one. We use the notation $\X=\O_{p}(C)$ to mean that 
\begin{equation*}
\E[\exp((C^{-1}\X_{+})^{p})]\leq 2.
\end{equation*}
We denote triadic cubes by
\begin{equation*}
\cu_{n}:=\left(-\frac{1}{2}3^{n}, \frac{1}{2}3^{n}\right)^{d}
\quad \mbox{and} \quad 
\cu_{n}(x):=3^{n}\Big\lfloor 3^{-n}x+\frac{1}{2}\Big\rfloor +\cu_{n}, \qquad x\in\Rd\,. 
\end{equation*}
The ball centered at~$x$ of radius~$r$ is denoted by $B_{r}(x)$ and the parabolic cylinder $Q_{r}(t,x):=(t-r^{2}, t]\times B_{r}(x)$. When $x$ or $(x,t)$ are the origin, we simply refer to these as $B_{r}$ and $Q_{r}$, respectively. We reserve the notation for $f\in L^{p}(U)$ with $|U|<\infty$, 
\begin{equation*}
\norm{f}_{\underline{L}^{p}(U)}:=\left(\fint_{U}|f|^{p}\right)^{\frac{1}{p}}=|U|^{-\frac{1}{p}}\norm{f}_{L^p(U)}\quad\text{and}\quad  (f)_{U}:=\fint_{U} f.
\end{equation*}
For a Radon measure $\nu$, we use the notation 
\begin{equation*}
| \nu |(U) := \| \nu\|_{\mathrm{TV}(U)},
\end{equation*}
where~$\| \cdot \|_{\mathrm{TV}(U)}$ is the total variation norm.

\smallskip

Our first main result is a quenched, local CLT for solutions of~\eqref{e.parab} with an explicit rate. We consider initial data~$v_0$ which lies below a Gaussian, and we show that the solution must converge to a multiple of the homogenized heat kernel in the large-time limit, with a quantitative rate. 
Like most of our results, the estimate itself is deterministic---there are no random quantities on the right side of~\eqref{e.CLT} or~\eqref{e.int}---but they are restricted to times~$t$ such that~$\sqrt{t}$ is larger than a certain minimal scale~$\Y$, which we bound sharply in~\eqref{e.minscaleclt}.  In the following statements, we write $\ahom$ for the homogenized matrix, $\overline{P}$ for the parabolic Green function associated to $\ahom$, and $m$ for the spatially stationary invariant measure constructed in \cite{PV} (though, we emphasize that our results provide a new, quenched construction of these objects that are independent of the results in \cite{PV}).

\begin{theorem}[Quantitative local CLT]
\label{t.realthm} 
Let $p\in (0,d)$. There exist~$\gamma(d,\lambda,\Lambda, p)\in(0,1]$, $C(d,\lambda,\Lambda)\in [1,\infty)$, and a random variable~$\Y$ satisfying 
\begin{equation}\label{e.minscaleclt}
\Y = \O_p(C),
\end{equation}
such that, for every~$R\geq \Y$ and solution~$v$ of the initial-value problem 
\begin{equation*}
\begin{cases}
\partial_{t}v-\tr(\A D^{2}v)=0&\text{in $(0, \infty)\times \R^{d}$,}\\
v(0,x)=v_{0}&\text{on $\R^{d}$,}
\end{cases}
\end{equation*}
with initial data satisfying, for $M\geq 0$, the bound
\begin{equation}\label{e.intgauss}
|v_{0}(x)|\leq MR^{-d}\exp\left(-\frac{|x|^{2}}{R^2}\right)
\,,
\end{equation}
there exists a random variable~$c[v_0]$ satisfying, for every~$t\geq R^2$,
\begin{equation}
\label{e.CLT}
\left| v(t,x)-c[v_{0}]\overline{P}(t,x)\right|
\leq 
CM\biggl (\frac{t}{R^2}\biggr )^{\!\!-\gamma}t^{-\frac{d}{2}}\exp\left(-\frac{|x|^{2}}{Ct}\right).
\end{equation}
If we assume moreover that $v_{0}\in C^{0, \sigma}(B_{R})$ for some $\sigma \in (0,1]$, then there exists $\ga=\ga(\la, \La, d, \sigma)\in (0,1)$ and $C(\la, \La, d, \sigma)\in [1, \infty)$ so that 
\begin{equation}\label{e.int}
\left| c[v_{0}]-\int_{\R^{d}} v_{0}(x)\, dx\right|
\leq 
CM\Bigl (1+M^{-1}R^{d+\sigma}[v_{0}]_{C^{0, \sigma}(B_{R})}\Bigr )
R^{-\gamma}.
\end{equation}
\end{theorem}

The constant $\gamma$ is identified explicitly in the proof of Theorem~\ref{t.realthm} and satisfies $\gamma\rightarrow 0$ as $p\rightarrow d$.  Furthermore, the previous result may be applied to the parabolic Green function itself, resulting in the following statement. 
\begin{theorem}[Parabolic Green function estimates]
\label{t.GF}
Let $p\in (0,d)$. There exist an exponent~$\gamma(d,\lambda,\Lambda, p)\in (0,1)$, a constant~$C(d,\lambda,\Lambda)\in [1,\infty)$, 
and a random variable~$\Y$ satisfying $\Y= \O_p(C)$,
and, for every~$y\in\cu_0$, a positive random variable~$m(y)$
such that, for every~$t \geq \Y^{2}$ and~$x\in\Rd$, 
\begin{equation}
\label{e.GF}
\left| P(t,x,y)-m(y) \overline{P}(t,x-y)\right|
\leq 
C m(y)\biggl (\frac{t}{\Y^2}\biggr )^{\!\!-\gamma}t^{-\frac{d}{2}}\exp\left(-\frac{|x-y|^{2}}{Ct}\right)\,.
\end{equation}
\end{theorem}

By stationarity, the statement of Theorem~\ref{t.GF} holds for any~$y \in \Rd$, with the estimate~\eqref{e.GF} valid with~$\tau_{[y]}\Y:=\mathcal{Y}(\tau_{[y]} \cdot)$ in place of~$\mathcal{Y}$, where~$[x]$ denotes the nearest element of~$\Zd$ to~$x$. As previously mentioned, the function $y\mapsto m(y)$ appearing in Theorem~\ref{t.GF} turns out to be the unique~$\Zd$--stationary invariant measure with ensemble mean equal to one. 
By an \emph{invariant measure}~$m$ in an open subset~$U\subseteq \Rd$, we mean a solution of the adjoint equation, which is formally written in coordinates as
\begin{equation}
\label{e.invmeas}
- \sum_{i,j=1}^d
\partial_{x_i}\partial_{x_j}
\bigl( \a_{ij} m \bigr)
= 0 \quad \mbox{in} \ U. 
\end{equation}
Of course, the equation is interpreted in the weak sense: precisely, a Radon measure~$\mu$ is an invariant measure in~$U$ if 
\begin{equation}
\label{e.invmeas.weak}
\int_{U} \tr \bigl( \a(x) D^2\varphi(x) \bigr) \, d\mu(x)
= 0, \quad \forall \varphi \in C^\infty_c(U)\,.
\end{equation}
The~$\Zd$--stationarity of~$m$ is inherited from the stationarity of~$\a$. The invariance of~$m$ is evident from the invariance of the functional~$c[\cdot]$ in Theorem~\ref{t.realthm} under the flow, and the fact that Theorem~\ref{t.GF} and linearity imply that
\begin{equation*}
c[v_0] = \int_{\Rd} v_0(x) m(x)\,dx
\,.
\end{equation*}
A proof of this is given below in \eqref{e.cformula}.

\smallskip

As we will discover from the proof of Theorem~\ref{t.GF}, the invariant measure satisfies the following pointwise upper and lower bounds: for some exponents~$q,\delta >0$ depending only on~$(d,\lambda,\Lambda)$, and a constant~$C$ depending additionally on~$(\al_{0}, K_{0})$ in addition to $(d, \la, \La)$,
\begin{equation}
\label{e.msand}
\inf_{B_{\Y}} m \geq C^{-1} \Y^{-q} 
\quad 
\mbox{and} 
\quad 
\sup_{B_{\Y}} 
m
\leq 
C\Y^{d-1-\delta}
\,.
\end{equation}
We emphasize that solutions of the adjoint equation~\eqref{e.invmeas} do not satisfy a scale-invariant Harnack inequality, nor do they possess scale-invariant~$L^\infty$ upper bounds or lower bounds. The best integrability estimate for a solution~$m$ of~\eqref{e.invmeas} which is valid on arbitrarily large scales and for every coefficient field~$\a(x)$ is with respect to~$L^{\frac{d}{d-1}+\delta}$: see Proposition~\ref{p.Guido} and Remark~\ref{r.nonono} for the counterexample. It is \emph{only on small scales}---scales on which the coefficients are quantitatively continuous---that we can appeal to better, pointwise estimates. 

\smallskip

In the proof of~\eqref{e.msand} (which appears in Section \ref{ss.m.exis.uniq}),  we start with the deterministic estimate on the scale~$r=\Y$ by~\eqref{e.int}, which says essentially that~$\| m \|_{\underline{L}^1(B_\Y)} \leq 2$. We can then use the deterministic, scale-invariant~$L^{d/(d-1)}$ estimate to obtain~$\| m \|_{\underline{L}^{\frac d{d-1}+\delta} (B_{\Y/2})} \leq C$. Applying H\"older's inequality, we then obtain, after giving up a volume factor, a bound of the type
\begin{equation*}
\sup_{x\in B_{\Y/4}} \| m \|_{\underline{L}^1(B_1(x))} \leq C \Y^{d-1-\delta}.
\end{equation*}
Now that we have an anchoring estimate on the unit scale, we are free to use the quantitative continuity of the coefficients assumed in~\eqref{e.holder} to upgrade to pointwise estimates, and we arrive at~$\| m \|_{L^\infty(B_{\sfrac12}(x))} \leq C \Y^{d-1-\delta}$ for each~$x \in B_\Y$. The lower bound is argued similarly. 

\smallskip

As a corollary of Theorem \ref{t.GF}, we obtain the following Nash-Aronson type estimate on large scales. We note that, unlike in the divergence form setting, such estimates are not true, in general, in the nondivergence form case. This estimate is therefore necessarily probabilistic, owing its validity to the fact that the equation is homogenizing on large scales. 

\begin{corollary}\label{c.NA}
Let $p\in (0,d)$. There exists a constant~$c(d, \la, \La)\in (0, \infty)$ and a constant $C(d,\lambda,\Lambda)\in [1,\infty)$, and a random variable~$\Y$ satisfying $\Y= \O_p(C)$ such that, for every~$y\in\Rd$, for $m$ as defined in Theorem \ref{t.GF}, and~$t \geq \tau_{[y]}\Y^{2}$,
\begin{equation}
\label{e.PGF.bound}
cm(y)t^{-\frac{d}{2}}\exp\left(-\frac{|x-y|^{2}}{ct}\right)\leq P(t,x,y)\leq Cm(y)t^{-\frac{d}{2}}\exp\left(-\frac{|x-y|^{2}}{Ct}\right).
\end{equation}
\end{corollary}

We remark that the bounds~\eqref{e.GF} and~\eqref{e.PGF.bound} can be integrated in time to obtain a similar bound on the elliptic Green function.
Note that the bounds~\eqref{e.msand} may be inserted into~\eqref{e.PGF.bound}, if desired, but the latter is actually much more precise as stated. 

\smallskip

The statement of Theorem~\ref{t.GF} also implies quantitative information about the rate of convergence of the weak limits in~\eqref{e.harpoons}.

\begin{theorem}
\label{t.mto1}
Let $p\in (0,d)$. There exist~$\gamma(d,\lambda,\Lambda, \al_{0}, K_0, p)\in (0,1)$, $C(d,\lambda,\Lambda)\in [1,\infty)$, and a random variable~$\Y$ with $\Y= \O_p(C),$ such that for every $R\geq \Y$, for $m$ as defined in Theorem \ref{t.GF}, 
\begin{equation}\label{e.m.conv}
\biggl| \fint_{R\cu_{0}} m(x) \,dx -1 \biggr|
+
\biggl| \fint_{R\cu_{0}} m(x) \a(x) \,dx - \ahom \biggr|
\leq CR^{-\ga}. 
\end{equation}
Consequently, for every~$\ep\in (0,1]$ with $\ep^{-1} \geq \Y$, there exists $\ga'(d, \la, \La, \al_{0}, K_{0}, p)\in (0,1)$ such that
\begin{equation}
\label{e.weaklim.quant}
\bigl\| m\bigl( \tfrac \cdot \ep \bigr) - 1 
\bigr\|_{W^{-\gamma',1}(\cu_0)}
+
\bigl\| (m\a)\bigl( \tfrac \cdot\ep\bigr) - \ahom 
\bigr\|_{W^{-\gamma',1}(\cu_0)}
\leq 
C\ep^\gamma\,.
\end{equation}

\end{theorem}

The Green function bounds in Theorem~\ref{t.GF} can be used to directly prove bounds on~$(m-1)$ in Theorem~\ref{t.mto1}. 
The bounds on~$m\a$ are proved by combining Green function estimates with bounds on the stationary approximate correctors (introduced below in~\eqref{e.dvd.def}) which are implicit in~\cite{asnondiv}.

\smallskip

Our next main result is a large-scale~$C^{0,1}$-type estimate for invariant measures. The estimate says that an arbitrary solution of the ``doubly divergence form'' adjoint equation in a finite ball can be approximated by a constant multiple of the unique stationary invariant measure in a smaller ball, in the same way as a Lipschitz function can be approximated by a constant. This is why we call it a~``$C^{0,1}$-type estimate.'' The estimate is only valid on balls above a certain random scale, which is why we call it a ``large-scale'' regularity estimate.

\begin{theorem}[Large-scale~$C^{0,1}$-regularity of invariant measures]
\label{t.largescaleC01}
Let~$p\in(0,d)$. 
There exists~$\delta_0 ( p,d,\lambda,\Lambda) \in(0,\frac 12]$ and, for every~$\delta\in (0, \delta_0]$, constants~$\beta( \delta,p,d,\lambda,\Lambda) \in (0,\frac 12]$ and~$C(\delta, p,d,\lambda,\Lambda)<\infty$ 
and a random variable~$\X$ 
satisfying~$\X = \O_p(C)$ such that, 
for every~$R\geq \X$ and invariant measure~$\nu$ in~$B_{R}$, 
there exists~$k\in\R$ such that, 
for every~$r\in [\X ,R]$, we have the estimate
\begin{equation}
\label{e.largescaleC01}
\frac{| \nu - k\mu|(B_r)}{|B_r|} 
\leq
\frac{Cr}{R} 
\biggl  (
\frac{
| \nu \ast \eta_{R^{1-\delta}} | (B_{5R})
}{|B_{5R}|} 
+
CR^{-\beta}
\frac{| \nu | (B_{R})}{|B_R|}
\biggr  )
\,, 
\end{equation}
where $\eta_{R}=R^{-d}\eta(\sfrac{\cdot}{R})$ is the standard mollifier. In particular, for~$r\in [\X ,R]$, 
\begin{equation}
\label{e.largescaleC01.TV}
\frac{| \nu - k \mu| (B_{r})}{|B_r|}
\leq
\frac{Cr}{R} 
\frac{ |  \nu |(B_{5R} )}{|B_{5R}|}
\,.
\end{equation}
\end{theorem}

Theorem~\ref{t.largescaleC01} can be viewed as a quantitative Liouville-type result. By taking~$R\to \infty$ in~\eqref{e.largescaleC01.TV}, we see that, for any coefficient field for which~$\X < \infty$, the measure~$d\mu=m(x)\,dx$ is the unique solution of the invariant measure equation with volume-normalized total variation in~$B_R$ growing like~$o(R)$. 
This is a much stronger, quenched version of uniqueness for the invariant measure when compared, for example, to the annealed uniqueness obtained by Papanicolaou and Varadhan \cite{PV}.

\smallskip

The proof of Theorem~\ref{t.largescaleC01}  
is based on an iteration argument originating in~\cite{AS}, which introduced the large-scale regularity in stochastic homogenization in the context of divergence form equations. The idea is to borrow the improved regularity of the homogenized equation. This requires the approximation of a general invariant measure by an~$\ahom$-harmonic function, which relies on Theorem~\ref{t.GF}. What is new  here is that this approximation must be in a weak sense---we compare \emph{mollifications} of the invariant measure to harmonic functions---and then upgrade using a weak-to-strong norm estimate for invariant measures (see Lemma~\ref{l.harmapprox} and Lemma~\ref{l.weak.to.strong} respectively). 

\smallskip

Theorem~\ref{t.largescaleC01} says that any invariant measure on a finite cube can be approximated, near the center of the cube, by a multiple of~$m(x)\,dx$. But this can be turned around:~$m$ can be approximated by an invariant measure defined on a finite cube. This idea can be used to localize~$m(x)$, by showing that it can be approximated by a sequence of fields having finite ranges of dependence, and thereby obtaining quantitative ergodicity for~$m(x)$. 
Note that the idea that ``large-scale regularity gives decorrelation'' lies at the heart of the theory of stochastic homogenization for divergence form equations~\cite{AKMbook}. 

\smallskip

Previously, Deuschel and Guo~\cite{DG} proved a uniform H\"older type estimate for solutions of~\eqref{e.ddf}. By dualizing the Krylov-Safonov estimate, they obtained a~$C^{0,\delta}$-type estimate for the \emph{ratio} of a solution with another positive solution. 

\smallskip

We turn to the statement of our final result, in which we give a quantitative bound of the ergodicity of the environmental process.
As described above, we let~$\{ X_t \}$ denote the Feller diffusion process whose generator is~$L$, where $L u  = \tr (\a D^2 u)$, which satisfies~\eqref{e.Feller}, and we denote the quenched law of~$\{ X_t \}$ by~$\mathbf{P}$.

\begin{theorem}
\label{t.environmental}
Let~$\xi \in (0,1)$ and~$p\in (0,d\xi)$. 
There exist~$\gamma=\ga(d,\lambda,\Lambda, \al_{0}, K_0, p)\in (0,1)$ and $C(d,\lambda,\Lambda, \al_{0}, K_0, p,\xi)\in(0,\infty)$, and a random variable~$\Y$ (depending on the environment but not the trajectories) with~$\Y= \O_p(C)$
such that, for every $T\geq \Y^{2}$, and $\eta\geq CT^{-\gamma}$, 
\begin{equation}
\label{e.lim.EFVP}
\mathbf{P} \Biggl[ \, 
\biggl| \frac{1}{T}\int_0^T \!\!\!
\a(X_s) \,ds - \ahom \biggr|
\geq \eta 
\Biggr]
\leq
C \exp\left(-\frac{\eta^2 T^{1-\xi}}{C}\right)
\,.
\end{equation}
\end{theorem}

Theorem~\ref{t.environmental} result quantifies the speed of convergence of the~$\mathbf{P}$-almost sure limit~\eqref{e.EPVP}.
Indeed, taking~$\eta = CT^{-\gamma}$ in the conclusion yields
\begin{equation*}
\mathbf{P} \biggl[ \, 
\biggl| \frac{1}{T}\int_0^T\!\!\!
\a(X_s) \,ds - \ahom \biggr|
\geq CT^{-\gamma} 
\biggr]
\leq
C\exp\left(-\frac{T^{1-\xi-2\gamma}}{C}\right).
\end{equation*}
Note that the estimate~\eqref{e.lim.EFVP} is \emph{quenched} in the sense that it depends on the environment~$\a(\cdot)$ only through the random variable~$\Y$, which specifies the smallest scale at which the estimate is valid. 

\smallskip 

We obtain Theorem~\ref{t.environmental} in Section~\ref{s.EPVP} as a consequence of the estimates in Theorem~\ref{t.mto1} via a martingale argument.
A similar statement to that of Theorem~\ref{t.environmental} was proved previously in the discrete case~\cite{GPT}, which is basically equivalent to an estimate like~\eqref{e.lim.EFVP} except with a right-hand side which is~$O(T^{-\delta})$ for a small~$\delta>0$ instead of the exponential bound. 

\subsection{Relation to previous works in the literature}  

Stochastic homogenization for elliptic equations in nondivergence form was first studied by Papanicolaou and Varadhan \cite{PV} and Yurinski\u{\i} \cite{Y1}.
In \cite{PV,Y1}, it was observed, using the fact that the underlying diffusion is a martingale, that homogenization reduces to proving an ergodic theorem for the so-called \emph{environment from the point of view of the particle.} It is therefore sufficient to construct a stationary invariant measure for this process.

\smallskip

The first quantitative homogenization results were proved by Yurinski{\u\i}~\cite{Y3,Y2}, who obtained algebraic convergence rates for linear equations in dimensions five and larger, and sub-algebraic (logarithmic-type) estimates in dimensions less than five. 
Several decades later, Caffarelli and Souganidis \cite{CafSou2010} obtained sub-algebraic homogenization estimates in all dimensions for linear and nonlinear equations. An algebraic error estimate was finally obtained in the work of Armstrong and Smart~\cite{asnondiv} and subsequently extended to the parabolic case by Lin and Smart~\cite{linsmart}. 
The arguments in this paper are strongly based on~\cite{asnondiv}.

\smallskip

The estimates of~\cite{asnondiv} imply certain \emph{large-scale regularity} estimates ($C^{1,1}$-estimates above a random scale) which, under special assumptions, can be used to obtain \emph{optimal} homogenization estimates. Under strong assumptions on the environment which guarantee concentration inequalities, this was demonstrated in prior work by Armstrong and Lin~\cite{AL}, who proved optimal bounds on the stationary approximate correctors. Estimates of a similar kind were subsequently proved in a set of very recent papers of Guo and Tran~\cite{guotran, guotran2}. Their results are obtained for a discrete lattice under iid assumptions, but can certainly be extended to the continuum setting or correlated conductances satisfying a spectral gap inequality or similar. In particular, they proved versions of Theorem~\ref{t.environmental} and the first part of Theorem \ref{t.mto1}, which are sharper in the scaling of~$T$ (respectively $R$), but qualitative in stochastic integrability.

The application of concentration inequalities to nondivergence form equations is analogous to the divergence form setting (see, for example, \cite{GNO3,AK}): a small algebraic rate of homogenization implies a large-scale regularity theory.  Assuming the environment yields a concentration inequality (such as the Efron-Stein or spectral gap inequality), it is possible to obtain estimates which are optimal in the scaling of the error but typically not stochastic integrability without further arguments (see, for example, \cite{GloNeuOtt20}).  An alternative approach based on \emph{renormalization arguments} has been developed for equations in divergence form \cite{AK}.  The estimates in this paper are the first step to a renormalization theory for equations in nondivergence form, which would be expected to yield optimal estimates.

\smallskip

In probabilistic language, the homogenization of equations~\eqref{e.ellip} and~\eqref{e.parab} is equivalent to an invariance principle for the diffusion satisfying~\eqref{e.SDE}. There is a similar correspondence between invariance principles for so-called ``balanced'' random walks in random environments on the lattice and discrete, finite-difference versions of the PDEs~\eqref{e.ellip} and~\eqref{e.parab}. The homogenization of the latter can be studied by methods completely analogous to the continuum case, essentially because homogenization is about large scales and is indifferent to the precise microstructure. 
Thus, in analogy with~\cite{PV}, Lawler~\cite{Law1982} established the invariance principle for balanced random walks in random environment under a uniform ellipticity assumption, and Guo, Peterson and Tran~\cite{GPT} adapted the arguments of~\cite{asnondiv} to the lattice, proving an algebraic convergence rate. They also consequently
obtained quantitative estimates which are similar to, but weaker than, our Theorems~\ref{t.GF} and~\ref{t.environmental}. Berger and Deuschel~\cite{BerDeu2014} and Berger, Cohen, Deuschel, and Guo~\cite{harnackrwre} obtained qualitative homogenization results for the case of degenerate diffusions satisfying the condition of \textit{genuine $d$-dimensionality} \cite[Definition~2]{BerDeu2014}. We expect the methods in this paper to be robust to these more general assumptions. 

\subsection{Outline of the Paper.} The rest of the paper is organized as follows. In Section \ref{s.loc}, we collect some deterministic estimates for parabolic nondivergence form equations. Section \ref{s.hom3} contains some preliminary quantitative homogenization results for the parabolic Cauchy-Dirichlet problem and Cauchy problem. In Section \ref{s.PGF}, we prove Theorem \ref{t.realthm} via an iterative, inductive argument, and we obtain Theorem \ref{t.GF} as a simple consequence. Section \ref{s.invmeas} contains all arguments pertaining to the quenched invariant measure, notably the proofs of Theorem \ref{t.mto1} and Theorem \ref{t.largescaleC01}. Finally, Section \ref{s.EPVP} contains the proof of Theorem \ref{t.environmental}, the quantitative ergodicity of the environmental process. We provide an Appendix at the end of the paper which contains some deterministic estimates on the elliptic adjoint equation which are used in Section \ref{s.invmeas}.

\subsection*{Acknowledgements}  
S.A. was supported from NSF grants DMS-1954357, DMS-2000200, DMS-2350340, from the Simons Programme at IHES during a sabbatical year, and from the European Research Council (ERC) under the European Union's Horizon Europe research and innovation programme, grant agreement number 101200828.
B.F.~acknowledges financial support from the National Science Foundation DMS-Probability Standard Grant 2348650, the Simons Foundation Travel Grant MPS-TSM-00007753, and the Louisiana Board of Regents RCS Grant 20130014386. J.L.~acknowledges financial support from the National Sciences and Engineering Research Council of Canada (NSERC), the Fonds de Recherche du Quebec-Nature et Technologies, and the Canada Research Chairs Program CRC-2018-00154 and CRC-2023-00081.

\section{Some deterministic estimates for nondivergence form equations}\label{s.loc}

We collect several deterministic estimates which will be useful throughout the paper. We establish pointwise, deterministic bounds for the solution of the equation
\begin{equation}\label{le_0}
\begin{cases}
\partial_t v-\tr(\A D^{2}v)=0&\text{in $(0, \infty)\times \R^{d}$,}\\
v(0,x)=v_{0}(x)&\text{on $\R^{d}$,}
\end{cases}
\end{equation}
for initial data $v_0\in L^{\infty}(\R^{d})$. The bounds in this section make no use of the stationarity or unit range of dependence assumptions, and they are valid for every fixed~$\a \in\Omega$. 

\smallskip

We begin by recalling the classical interior Krylov-Safonov estimate. Let $v$ solve
\begin{equation*}
\partial_{t}v-\tr(\A D^{2}v)=0\quad \text{in $Q_{r}(t_{0}, 0)$}. 
\end{equation*}
There exists $C=C(\la, \La, d)>0$ and $\sigma=\sigma(\la, \La, d)\in (0, 1)$ so that 
\begin{equation}\label{e.vanks}
r^{\sigma}\left[v\right]_{\mathcal{C}^{0,\sigma}(Q_{r/2}(t_{0}, 0))}\leq C\norm{v}_{L^{\infty}(Q_{r}(t_{0}, 0))},
\end{equation}
where $\mathcal{C}^{0, \sigma}$ is the standard parabolic H\"older semi-norm, 
\begin{equation*}
[v]_{\mathcal{C}^{0,\sigma}(Q_{r/2}(t_{0}, 0))}:=\sup_{(t,x)\neq(s,y)\in Q_{r/2}(t_{0}, 0)} \frac{|v(t,x)-v(s,y)|}{(|x-y|+|t-s|^{\frac{1}{2}})^{\sigma}}. 
\end{equation*}
We next state a short corollary of this estimate, which will be helpful throughout the rest of the paper. 

\begin{lemma}\label{l.ksref}
Let $v$ denote the solution of  
\begin{equation*}
\partial_{t}v-\tr(\A D^{2}v)=0\quad\text{in $(0, \infty)\times \R^{d}$}.
\end{equation*}
There exists $C=C(\la, \La, d)\in [1, \infty)$ and $\sigma=\sigma(\la, \La, d)\in (0,1)$ so that for every $r_{1}\geq r$, for every $t_{0}\geq r^{2}$, 
\begin{equation*}
r_{1}^{\sigma}\left[v(t_{0}, \cdot)\right]_{C^{0,\sigma}(B_{r_{1}})}\leq C\left(\frac{r_{1}}{r}\right)^{\sigma}\norm{v}_{L^{\infty}([t_{0}-r^{2}, t_{0}]\times B_{r_{1}})}. 
\end{equation*}
\end{lemma}

\begin{proof}
Without loss of generality, we let $t_{0}=r^{2}$, as the argument for any $t_{0}\geq r^{2}$ can be achieved by a time shift. We begin by noting that 
\begin{equation*}
\partial_{t}v-\tr(\A D^{2}v)=0\quad\text{in $Q_{r}(r^{2}, 0)$},
\end{equation*}
so by the Krylov-Safonov estimate \eqref{e.vanks}, there exists $C=C(\la, \La, d)\in [1,\infty)$ and $\sigma=\sigma(\la, \La, d)\in (0,1)$ so that 
\begin{equation*}
r^{\sigma}[v(r^{2}, \cdot)]_{C^{0, \sigma}(B_{r/2})}\leq C\norm{v}_{L^{\infty}(Q_{r}(r^{2}, 0))}\leq C\norm{v}_{L^{\infty}((0,r^{2}]\times B_{r_{1}})}.
\end{equation*}
This holds true for any $Q_{r}(r^{2},x)$, and hence, 
\begin{align*}
r^{\sigma}&[v(r^{2}, \cdot)]_{C^{0, \sigma}(B_{r_1})}\\
&=r^{\sigma}\sup_{x\neq y\in B_{r_{1}}} \frac{|v(r^{2},x)-v(r^{2},y)|}{|x-y|^{\sigma}}\\
&\leq r^{\sigma}\sup_{x\neq y\in B_{r_{1}}, |x-y|\leq r/2} \frac{|v(r^{2},x)-v(r^{2},y)|}{|x-y|^{\sigma}}+r^{\sigma}\sup_{x\neq y\in B_{r_{1}}, |x-y|>r/2} \frac{|v(r^{2},x)-v(r^{2},y)|}{|x-y|^{\sigma}}\\
&\leq C\norm{v}_{L^{\infty}((0,r^{2}]\times B_{r_{1}})}+2^{\sigma+1}\norm{v(r^{2}, \cdot)}_{L^{\infty}(B_{r_{1}})}.
\end{align*}
Multiplying by $(r_{1}r^{-1})^{\sigma}$ implies the result. 
\end{proof}

We next discuss estimates on the parabolic Green function, which we denote by~$P=P(t,x,y)$. We recall the following result which can be found in Friedman \cite{Fri2008}, establishing the existence and uniqueness of~$P$ and giving Gaussian-type estimates (these are useful in practice only for small times~$t$ on the order of the length scale on which the coefficients vary).  

\begin{proposition}\label{le_prop_1}\cite[Theorem~11, Chapter 1, Section VI]{Fri2008}
There exists a unique Parabolic Green function $P: (0,\infty)\times\R^d\times\R^d\rightarrow\R$ solving \eqref{e.Pdef}.  Furthermore, for every $T\in[0,\infty)$, there exist constants $C=C(\la, \La, d, \alpha_{0}, K_{0}, T)\in [1, \infty)$ and $B=B(\la, \La, d, \alpha_{0}, K_{0}, T)\in [1, \infty)$ such that, for every $x,y\in\R^d$ and $t\in(0,T]$,
\[0<P(t,x,y)\leq Ct^{-\frac{d}{2}}\exp\left(-\frac{\abs{x-y}^2}{Bt}\right).\]
\end{proposition}

We next observe that we can construct a family of sub/supersolutions to \eqref{le_0}:
\begin{lemma}\label{le_lem_1}  Consider the equation
\begin{equation}\label{le_lem_eq}
\partial_{t}v-\tr(\A D^{2}v)=0\quad\text{in $(0, \infty)\times \R^{d}$},
\end{equation}
and for every $\ka,\theta\in(0,\infty)$ let $f_{\ka, \theta}$ be the function defined by
\[f_{\ka,\theta}(t,x) = t^{-\ka}\exp\left(-\frac{\abs{x}^2}{\theta  t}\right).\]
Then $f_{\ka,\theta}$ is a supersolution of \eqref{le_lem_eq} for every $\theta\in[4\Lambda, \infty)$ and $\ka\in[0,2\lambda d\theta^{-1}]$, and $f_{\ka,\theta}$ is a subsolution of \eqref{le_lem_eq} for every $\theta\in(0, 4\la]$ and $\ka\in[2\Lambda d\theta^{-1},\infty)$.  \end{lemma}

\begin{proof}  
By the uniform ellipticity assumption~\eqref{e.unifellip},
\begin{align*}
 \partial_t f_{\ka,\theta}- \tr(\A D^2f_{\ka,\theta})) &=f_{\ka,\theta}\left(\frac{\abs{x}^2}{\theta t^2}+ \frac{2\tr(\A)-\ka\theta}{\theta t}-\frac{4\langle \A x,x\rangle}{\theta^2 t^2}\right)\\
& \geq f_{\ka,\theta}\left(\frac{\abs{x}^2}{\theta t^2}+ \frac{2\lambda d-\ka\theta}{\theta t}-\frac{4\Lambda\abs{x}^2}{\theta^2 t^2}\right).
\end{align*}
Therefore, after choosing $\theta\in[4\Lambda, \infty)$ and $\ka\in[0,2\lambda d\theta^{-1}]$, we obtain that $f_{\ka,\theta}$ is a supersolution of \eqref{le_lem_eq}.  The analogous argument establishes that $f_{\ka,\theta}$ is a subsolution of \eqref{le_lem_eq} for every $\theta\in(0, 4\la]$ and $\ka\in[2\Lambda d\theta^{-1},\infty)$.
\end{proof}

We may now use this family of sub/supersolutions to obtain sub-optimal bounds on the parabolic Green function for large times $t\geq 1$. 

\begin{proposition}\label{le_prop_2}  There exist constants $C,c, B,b$, $\underline{\ka}$, and $\overline{\ka}\in [1,\infty)$, all depending on $\la, \La,d, \al_{0}, K_{0}$ such that for $P$ as in Proposition \ref{le_prop_1}, every $x,y\in\R^d$ and $t\in[1,\infty)$,
\begin{equation}\label{le_prop_2_eq}ct^{-\underline{\ka}}\exp\left(-\frac{\abs{x-y}^2}{bt}\right)\leq P(t,x,y)\leq Ct^{-\bar{\ka}}\exp\left(-\frac{\abs{x-y}^2}{Bt}\right).\end{equation}
\end{proposition}

\begin{proof}  Let $C,B\in(0,\infty)$ be the constants from Proposition~\ref{le_prop_1} corresponding to $T=1$:  we have that, for every $x,y\in\R^d$ and $t\in(0,1]$,
\begin{equation}\label{le_2}0\leq P(t,x,y)\leq Ct^{-\frac{d}{2}}\exp\left(-\frac{\abs{x-y}^2}{Bt}\right).\end{equation}
We will now extend this estimate to a global in time estimate using the supersolution constructed in Lemma~\ref{le_lem_1}, at the potential cost of decreasing the constants.  Let $C=c_*$, $B= B\vee 4\Lambda $, and $\bar{\ka} = 2\lambda dB^{-1}$. From this choice, we have that, for every $x,y\in\R^d$,
\[Cf_{\bar{\kappa},B}(1,x-y)\geq C\exp\left(-\frac{\abs{x-y}^2}{B}\right).\]
The upper bound in \eqref{le_prop_2_eq} is then a consequence of Lemma~\ref{le_lem_1} and the comparison principle.  It remains to establish the lower bound.

It follows from Friedman \cite[Corollary 1 of Theorem 5', Appendix I]{Fri2008} that there exists $c,b\in(0,\infty)$ such that, for every $x,y\in\R^d$,
\[P(1,x,y)\geq c\exp\left(-\frac{\abs{x-y}^2}{b}\right).\]
Let $b = b\wedge 4\lambda$ and $\underline{\ka} = 2\Lambda db^{-1}$.  We then have that, for every $x,y\in\R^d$,
\[ cf_{\underline{\ka},b}(1,x-y)\leq c\exp\left(-\frac{\abs{x-y}^2}{b}\right).\]
The proof of the lower bound in \eqref{le_prop_2_eq} is then a consequence of Lemma~\ref{le_lem_1} and the comparison principle.  \end{proof}

\section{Some quantitative homogenization estimates}\label{s.hom3}

In this section, we collect several homogenization results which will be used throughout the sequel. Most of these results can be viewed as corollaries of \cite{asnondiv}. 

We first collect some definitions and results from \cite{asnondiv}. For $\A\in \Om$, $U\subseteq \R^{d}$, $M\in \mathcal{S}^{d}$, and $\ell\in \R$, we recall the quantities
\begin{equation*}  
S(U, \A(M+\cdot), \ell)=\left\{u\in C(\overline{U}): \text{$-\tr(\A(M+D^{2}u))\geq \ell$ in $U$}\right\},
\end{equation*}
and 
\begin{equation*}
\mu(U, \A(M+\cdot), \ell):=\frac{1}{|U|}\sup\left\{|\partial \Ga_{u}(U)|: u\in S(U, \A(M+\cdot), \ell)\right\}, 
\end{equation*}
where $\Ga_{u}$ is the convex envelope of $u$, $\partial \Ga_{u}$ is the subdifferential, and $|U|$ denotes the Lebesgue measure of $U$. We similarly define 
\begin{align*}
\mu_{*}(U, \A(M+\cdot), \ell):&=\frac{1}{|U|}\sup\left\{|\partial \Ga_{-u}(U)|: \text{$-\tr(\A(M+D^{2}u))\leq \ell$ in $U$}\right\}\\
&=\frac{1}{|U|}\sup\left\{|\partial \Ga_{-u}(U)|: \tr(\A(-M+D^{2}(-u)))\leq \ell\right\}\\
&=\frac{1}{|U|}\sup\left\{|\partial \Ga_{u}(U)|: u\in S(U,(\A(-M+\cdot), -\ell) \right\}.
\end{align*}
The quantity $\mu$ is invariant with respect to affine shifts, and it measures the maximal curvature of the graph of the convex envelope of a supersolution of the equation $-\tr(A(M+D^2u))=\ell$.  The quantity $\mu_{*}$ measures the curvature of subsolutions (see \cite[Section~2]{asnondiv} for a more complete discussion).

We recall that the statement $X\leq \O_{1}(A)$ means
\begin{equation*}
\P[A^{-1}X>\la]\leq \exp(-\la),
\end{equation*}
which is equivalent to 
\begin{equation*}
\P[X>\la]\leq \exp(-A^{-1}\la).
\end{equation*}
Observe that $A\mapsto \O_{1}(A)$ is a decreasing function, and further, we have that if $X_{i}=\O_{1}(A_{i})$, then
\begin{equation*}
\sum_{i=1}^{\infty} X_{i}\leq \O_{1}\left(\sum_{i=1}^{\infty}A_{i}\right).
\end{equation*}
(See for example \cite[Lemma A.4]{AKM} for a reference.)

We first recall one of the main results from \cite{asnondiv}: 
\begin{theorem}\cite[Theorem 2.9]{asnondiv}\label{t.mainas}
There exists a deterministic matrix $\ahom$, constant $C=C(\la, \La, d)\in [1,\infty)$ and $\tau\in (0,1)$ such that for all $1\leq i,j\leq d$ and every $m\in \N$, 
\begin{equation*}
\E[\mu(\cu_{m}, \A(e_{i}\otimes e_{j}+\cdot), -\ahom_{ij})^{2}+\mu^{*}(\cu_{m}, \A(e_{i}\otimes e_{j}+\cdot), -\ahom_{ij})^{2}]\leq C\tau^{m}.
\end{equation*}
\end{theorem}
We now aim to control 
\begin{equation*}
\mathcal{E}(m):=\max_{1\leq i,j\leq d} \max\left\{\mu(\cu_{m}, \A(e_{i}\otimes e_{j}+\cdot), -\ahom_{ij}), \mu^{*}(\cu_{m}, \A(e_{i}\otimes e_{j}+\cdot), -\ahom_{ij})\right\}
\end{equation*}
\begin{lemma}\label{l.fr}
There exists constants $C=C(\la, \La, d)\in [1, \infty)$ and $\al=\al(\la, \La, d)\in (0,1)$, such that for every $m\in \mathbb{N}$, and $l<m$, 
\begin{align*}
\mathcal{E}(m)\leq C3^{-l\al}+\O_{1}(C3^{-d(m-l)})
\end{align*}
\end{lemma}

\begin{proof}
For ease of notation, we let $\mu(\cu_{m})=\mu(\cu_{m}, \A(e_{i}\otimes e_{j}+\cdot), -\ahom_{ij})$, and similarly, $\mu^{*}(\cu_{m})=\mu^{*}(\cu_{m}, \A(e_{i}\otimes e_{j}+\cdot), -\ahom_{ij})$.

By Theorem \ref{t.mainas}, we deduce that $\E[\mu(\cu_{m})]\leq C\tau^{m/2}$. 
Choosing $\al:=-\log_{3}(\sqrt{\tau})\in (0,1)$, we deduce that $\E[\mu(\cu_{m})]\leq C3^{-m\al}$.

To prove the lemma, we follow a (now) standard argument used in the theory of quantitative homogenization (see for example \cite[Lemma A.7]{AKMbook}). For completeness, we provide details specific to this setting. First, using the subadditivity of $\mu(\cu_{m})$ and the finite range of dependence assumption, we have that for any $l<m$, and for any $t>0$, 
\begin{equation}\label{e.cor210}
\log \E[\exp(3^{d(m-l-1)}t\mu(\cu_{m}))]\leq 3^{d(m-l-1)}\log \E[\exp(t\mu(\cu_{l}))]. 
\end{equation}
This exact estimate appears in \cite[Proof of Corollary 2.10]{asnondiv}. 
As in that argument, owing to the uniform ellipticity assumption, there exists a constant $M_{0}=M_{0}(\la, \La, d)$ such that $\P[\sup_{m}|\mu(\cu_{m})|\leq (2M_{0})^{d}]=1$, and hence, letting $t=(2M_{0})^{-d}$, we have
\begin{align*}
\log \E[\exp(3^{d(m-l-1)}(2M_{0})^{-d}\mu(\cu_{m}))]&\leq 3^{d(m-l-1)} \log\E[\exp((2M_{0})^{-d}\mu(\cu_{l}))]\\
&\leq 2\cdot 3^{d(m-l-1)} (2M_{0})^{-d}\E[\mu(\cu_{l})], 
\end{align*}
where in the final inequality, we used the elementary inequalities, 
\begin{equation*}
\begin{cases}
\exp(s)\leq 1+2s&\text{for all $s\in [0,1]$,}\\
\log(1+s)\leq s&\text{for all $s\geq 0$}.
\end{cases}
\end{equation*}
By Chebyshev's inequality, this implies that for every $\la>0$, 
\begin{align*}
\P[\mu(\cu_{m})>\la]\leq \exp\left(-3^{d(m-l-1)}(2M_{0})^{-d}\la+ 2\cdot 3^{d(m-l-1)} (2M_{0})^{-d}\E[\mu(\cu_{l})]\right).
\end{align*}
Letting now $\la\mapsto 2\E[\mu(\cu_{l})]+\la$, since $\mu(\cu_{l})\geq 0$, this implies 
\begin{align*}
\P[\mu(\cu_{m})-2\E[\mu(\cu_{l})]>\la]\leq \exp(-3^{d(m-l-1)}(2M_{0})^{-d}\la).
\end{align*}
This implies that for $C=C(M_{0}, d)=C(\la, \La, d)$, 
\begin{equation*}
\mu(\cu_{m})-2\E[\mu(\cu_{l})]\leq \O_{1}(C3^{-d(m-l)}),
\end{equation*}
and hence by the observation at the start of the proof, 
\begin{equation*}
\mu(\cu_{m})\leq C3^{-l\al}+\O_{1}(C3^{-d(m-l)}).
\end{equation*}
An identical argument can be made for $\mu^*(\cu_m)$, and this yields the desired estimate for $\mathcal{E}(m)$. 
\end{proof}

We now show how this estimate can be upgraded to argue that beyond a certain scale $\Y$, for which $\Y=\O_{d}(C)$, $\mathcal{E}(m)$ in fact decays algebraically. 
\begin{corollary}\label{c.minscale}
There exists constants $\theta=\theta(\la, \La, d)\in (0,1)$ and $C=C(\la, \La, d)\in [1,\infty)$, and a random variable $\Y=\O_{d}(C)$, such that for every $m\in \N$ with $3^{m}\geq \Y$, 
\begin{equation}
\mathcal{E}(m)\leq \left(\frac{3^{m}}{\Y}\right)^{-\theta}. 
\end{equation}
\end{corollary}

\begin{proof}
This proof roughly follows \cite[Lemma 4.13]{AK}. However, we present a simplified argument, because we are not aiming for the optimal exponent $\theta$. 

We start from the statement of Lemma \ref{l.fr}. For any $\rho\in (0,1)$, let $l:=\lfloor \rho m\rfloor$. Then $\rho m -1\leq l\leq \rho m$, and thus
\begin{equation}\label{e.intstep}
\mathcal{E}(m)\leq C3^{-\rho m\al}+\O_{1}(C3^{-d(1-\rho)m}).
\end{equation}

Let $\theta:=\frac{\al d}{2(\al+d)}$, which implies $\frac{(\al+d)\theta}{\al}<\frac{d}{2}.$ Define
\begin{equation*}
\Y:=\sup \left\{3^{N}: \text{$N\in \mathbb{N}$ and $\exists j\in \mathbb{N}\cap [N, \infty), \mathcal{E}(m)>3^{-\theta(j-N)}$}\right\}.
\end{equation*}
Notice that for all $3^{m}\geq \Y$, we have the desired bound. 

We now claim that $\Y=\O_{d}(C)$. Let $N\geq N_{0}$, where 
\begin{equation*}
N_{0}:=\lceil \theta^{-1}\tilde{C}\rceil,
\end{equation*}
where $\tilde{C}:=\log_{3}(2C_{\eqref{e.intstep}})$. For each $m\in \mathbb{N}\cap [N, \infty)$, we let 
\begin{equation*}
\rho_{m}=\frac{\theta(m-N)+\tilde{C}}{\al m}, 
\end{equation*}
for which it is clear, by definition of $N_{0}$, that $\rho_{m}\leq \frac{\theta}{\al}<1$. Moreover, for each $m\in \mathbb{N}\cap [N, \infty)$, $C_{\eqref{e.intstep}}3^{-\rho_{m} \al m}\leq \frac{1}{2}3^{-\theta(m-N)}$, and so by \eqref{e.intstep}, 
\begin{equation*}
\mathcal{E}(m)\leq \frac{1}{2}3^{-\theta(m-N)}+\O_{1}(C_{\eqref{e.intstep}}3^{-d(1-\rho_{m})m}).
\end{equation*}
This implies
\begin{equation*}
\P\big[\mathcal{E}(m)>3^{-\theta(m-N)}\big]\leq \exp\left(-\frac{1}{2}C^{-1}_{\eqref{e.intstep}}3^{-\theta(m-N)+d(1-\rho_{m})m}\right). 
\end{equation*}
In other words, 
\begin{equation*}
\indc_{\left\{\mathcal{E}(m)>3^{-\theta(m-N)}\right\}}\leq \O_{1}(C3^{\theta(m-N)-d(1-\rho_{m})m}).
\end{equation*}
By summing over $m\geq N$, we have 
\begin{equation*}
\indc_{\left\{\Y\geq 3^{N}\right\}}\leq \sum_{m=N}^{\infty}\indc_{\left\{\mathcal{E}(m)>3^{-\theta(m-N)}\right\}}\leq \O_{1}\bigg(C\sum_{m=N}^{\infty}3^{\theta(m-N)-d(1-\rho_{m})m}\bigg)
\end{equation*}
the proof is complete if we can show that for $C>1$ ,
\begin{equation}\label{e.summm}
\sum_{m=N}^{\infty}3^{\theta(m-N)-d(1-\rho_{m})m}\leq C3^{-Nd}. 
\end{equation}
Indeed, if \eqref{e.summm} holds, then
\begin{align*}
\E\left[\exp(\Y^{d})\right]&\leq \exp(3^{N_{0}d})+\int_{3^{N_{0}}}^{\infty}\exp(s^{d})\P[\Y>s]\, ds\\
&\leq \exp(3^{N_{0}d})+C\sum_{N=N_{0}}^{\infty}\exp(3^{Nd})\exp(-C3^{Nd})<\infty. 
\end{align*}
We now prove \eqref{e.summm}. From the expression for $\rho_{m}$, we note that 
\begin{align*}
\theta(m-N)-d(1-\rho_{m})m&=\theta(m-N)-d\left[\frac{\al m-\theta(m-N)+\tilde{C}}{\al}\right]\\
&=m\left(\frac{(\al+d)\theta}{\al}-d\right)-\frac{(\al+d)}{\al}\theta N-d\tilde{C}{\al},
\end{align*}
where by choice of $\theta$, the coefficient on the term $m$ is strictly negative. Hence
\begin{equation*}
\sum_{m=N}^{\infty}3^{\theta(m-N)-d(1-\rho_{m})m}< C\sum_{m=N}^{\infty}3^{m\left(\frac{(\al+d)\theta}{\al}-d\right)-\frac{(\al+d)}{\al}\theta N}\leq C3^{-Nd}, 
\end{equation*}
where we can take $C>1$. 

\end{proof}

\begin{remark}\label{r.giveup}
We can think of $\Y$ as a ``minimal scale'', because for $3^{m}>\Y$, the right hand side of the above expression decays algebraically. It is seen from Corollary \ref{c.minscale} that one can obtain ``optimal stochastic integrability'' of the minimal scale. We also consider the minimal scale which is near optimal: for every $p\in (0,d)$, there exists a random variable $\mathcal{Y}_{p}=\O_{p}(C)$ and $\theta=\theta(\la, \La, d, p)\in (0,1)$ such that for all $m\in \mathbb{N}$ with $3^{m}\geq \mathcal{Y}_{p}$, 
\begin{equation*}
\mathcal{E}(m)\leq 3^{-m\theta}. 
\end{equation*}
Indeed, the above estimate is immediate from Corollary \ref{c.minscale} by letting $\mathcal{Y}_{p}=\mathcal{Y}^{\frac{d}{p}}=\mathcal{Y}^{1+\frac{d-p}{p}}$. The near optimal minimal scale will allow us a bit more flexibility in future arguments. 
\end{remark}

Finally, we recall that by \cite[Lemma 2.1]{asnondiv}, there exists~$C=C(d, \la, \La)\in [1, \infty)$ such that, for every supersolution~$u\in S(\cu_{m}, \A, \ell)$, 
\begin{equation*}
\inf_{\partial\cu_{m}} u\leq \inf_{\cu_{m}} u+C3^{2m}\mu(\cu_{m}, \A,\ell)^{\frac{1}{d}}, 
\end{equation*}
and we have a similar estimate for $\mu_{*}(\cu_{m}, \A,\ell)$ from below. By Corollary \ref{c.minscale}, this implies for every $m\in \N$ with $3^{m}\geq \Y$,
\begin{equation}\label{e.muabp}
\inf_{\partial\cu_{m}} u\leq \inf_{\cu_{m}} u+C3^{2m}\left(\frac{3^{m}}{\Y}\right)^{-\frac{\theta}{d}}.
\end{equation}

It is useful to work with a slightly augmented minimal scale, which has the property that the minimal scale $\Y$ (or $\Y_{p}$) is sufficiently small an entire mesoscopic subgrid.  We formalize this in the following remark.

\begin{remark}
\label{r.Ystep1}
If~$\Y_{p}$ is the minimal scale given in Remark \ref{r.giveup}, then for every $k\in\N$, we introduce the enlarged minimal scale
\begin{equation*}
\tilde{\Y}_{(k)p} 
:=
\sup_{m\in \Z}
\biggl\{ 3^m \,:\, 
\sup_{z \in \Zd \cap \cu_{km}} 
\tau_z \Y_{p} > 3^{m} 
\biggr\}.
\end{equation*}
It is easy to check by a union bound that~$\tilde{\Y}_{(k)p}=\O_{d}(C')$ for a new constant $C'>C$.  
Therefore, we freely replace~$\Y_{p}$ by~$\tilde{\Y}_{(k)p}$ in any argument. 

\end{remark}
\subsection{Estimates for the Cauchy-Dirichlet problem}

We prove a homogenization error estimate for the Cauchy-Dirichlet problem for~\eqref{e.parab}. Let $U_{T}:=(0, T]\times U$ for $U\subseteq \R^{d}$. Let $u^{\ve}$ solve
\begin{equation}\label{e.parCD}
\begin{cases}
\partial_{t}u^{\ve}-\tr\left(\A^{\ve}D^{2} u^{\ve}\right)=0&\text{in $U_{T}$,}\\
u^{\ve}=g&\text{on $\partial_{p}U_{T}$}, 
\end{cases}
\end{equation}
where $\A^{\ve}=\A(\cdot/\ve)$, and $\bar{u}$ solve
\begin{equation}\label{e.parhom}
\begin{cases}
\partial_{t}\bar{u}-\tr\left(\ahom D^{2} \bar{u}\right)=0&\text{in $U_{T}$,}\\
\bar{u}=g&\text{on $\partial_{p}U_{T}$}, 
\end{cases}
\end{equation}
where $g\in \mathcal{C}^{0,1}(\partial_{p}U_{T})$ is a fixed Lipschitz function. The following statement is an estimate for the homogenization error of the Cauchy-Dirichlet problem, which exhibits an algebraic-type decay. 

\begin{theorem}[Error Estimate for the Cauchy-Dirichlet Problem]\label{t.algrate2}
There exists $\beta(\la, \La, d) \in (0,1)$, $C(\la, \La, d, U_{T})\in [1,\infty)$, and a random variable~$\Y$ satisfying 
$\Y = \O_d(C)$
such that, for every~$g\in \mathcal{C}^{0,1}(\partial_pU_T)$ and~$\ep \in (0,\Y^{-1} ]$, if~$u^\ep$ and~$\overline{u}$ are the solutions of~\eqref{e.parCD} and~\eqref{e.parhom}, respectively, 
then 
\begin{equation}\label{e.t1eq}
\norm{u^{\ve}-\bar{u}}_{L^{\infty}(U_{T})}\leq C\norm{g}_{\mathcal{C}^{0,1}(\partial_{p}U_{T})} (\Y\ve)^{\beta}.
\end{equation}
For any $p\in (0,d)$, for $\Y_{p}$ as in Remark \ref{r.giveup}, for $\ve\in (0, \Y_{p}^{-1}]$, there exists $\beta=\beta(\la, \La, d, p)\in (0,1)$ and $C=C(\la, \La, d)\in [1, \infty)$ such that 
\begin{equation}\label{e.tgiveup}
\norm{u^{\ve}-\bar{u}}_{L^{\infty}(U_{T})}\leq C\norm{g}_{\mathcal{C}^{0,1}(\partial_{p}U_{T})} \ve^{\beta}.
\end{equation}
\end{theorem}

For each $m\in \R$, we introduce the \emph{approximate Dirichlet correctors} defined by $\phi^{m}_{ij}(\cdot)$ for $1\leq i, j\leq d$, solving
\begin{equation*}
\begin{cases}
-\tr(\a(e_{i}\otimes e_{j}+D^{2}\phi^{m}_{ij}))=-\ahom_{ij}&\text{in $\cu_{m}$},\\
\phi^{m}_{ij}=0&\text{on $\partial \cu_{m}$}.
\end{cases}
\end{equation*}
For any matrix $M\in \mathbb{S}^{d\times d}$, we may also define $\phi^{m}(\cdot; M)$ by 
\begin{equation}\label{e.Mcorrect}
\begin{cases}
-\tr(\a(M+D^{2}\phi^{m}(\cdot; M))=-\tr(\ahom M)&\text{in $\cu_{m}$},\\
\phi^{m}(\cdot; M)=0&\text{on $\partial \cu_{m}$}.
\end{cases}
\end{equation}
Observe that by linearity, 
\begin{equation*}
\phi^{m}(x; M)=\sum_{i,j=1}^{d}M_{ij}\phi^{m}_{ij}(x).
\end{equation*}

We next show that the approximate Dirichlet correctors always control the homogenization error. 
\begin{lemma}\label{l.toobig}
Let~$U\subseteq \cu_0$ be a bounded Lipschitz domain. 
Let $u, v$ satisfy 
\begin{equation}\label{e.2bigeq}
\begin{cases}
u_{t}-\tr (\ahom D^{2}u)=0=v_{t}-\tr (\A D^{2}v) &\text{in $U_{T}$}, \\
u=v=g&\text{on $\partial_{p}U_{T}$,}
\end{cases}
\end{equation}
with $\norm{g}_{\mathcal{C}^{0,1}(\partial_{p}U_{T})}\leq M_{0}.$ There exists $\sigma(\la, \La, d)\in (0,1)$ and $C(\la, \La, d)\in [1,\infty)$ such that
\begin{equation*}
\| u -v\|_{L^\infty(U_T)}
\leq 
CM_{0}
\Bigl( 
\sup_{1\leq i,j \leq d}
\| \phi_{ij}^0 \|_{L^\infty(\cu_0)}  
\Bigr)^{\sigma/11}.
\end{equation*}
\end{lemma}

\begin{proof}
Throughout the proof, we will allow $c, C$ to vary line-by-line, depending only on $(\la, \La, d)$. We may, without loss of generality, assume that $M_{0}=1$, and hence $\norm{u}_{L^{\infty}(U_{T})}\leq 1$. Let $\al$ be given by
\begin{equation}\label{e.toobig}
\al :=\max_{\overline{U_T}}(u-v), 
\end{equation}
We will assume that $\al>0$, since if $\al<0$, then we can argue analogously by reversing the roles of $u$ and $v$. 

Without loss of generality, we may assume, that for any $K>0$, we can replace $v$ by $\tilde{v}$ solving 
\begin{equation}\label{e.stricteq}
\begin{cases}
\tilde{v}_{t}-\tr (\A D^{2}\tilde{v})=K\al &\text{in $U_{T}$}, \\
\tilde{v}=v&\text{on $\partial_{p}U_{T}$,}
\end{cases}
\end{equation}
since by the parabolic Alexandrov-Bakelman-Pucci estimate, there exists $C=C(\la, \La, d)\in [1,\infty)$ so that 
\begin{equation*}
|\tilde{v}-v|\leq CK\al, 
\end{equation*} 
for which we still have 
\begin{equation*}
\max(u-\tilde{v})\geq \max(u-v)-|v-\tilde{v}| \geq (1-CK)\al>0,
\end{equation*}
for $K(\la, \La, d)$ sufficiently small. We henceforth set $v$ to be $\tilde{v}$ solving \eqref{e.stricteq}.

\smallskip 

The Krylov-Safonov estimates and the Lipschitz continuity of $u, v$ on the parabolic boundary imply
\begin{equation*}
\norm{u}_{\mathcal{C}^{0, \sigma}(\overline{U_{T}})}+\norm{v}_{\mathcal{C}^{0, \sigma}(\overline{U_{T}})}\leq C. 
\end{equation*}
Since $u=v$ on $\partial_{p}U_{T}$,
\begin{equation}\label{e.notbdy}
|u(t,x)-v(t,x)|\leq C(d_{p}[(t,x), \partial_{p}U_{T}])^{\sigma}, 
\end{equation}
for the parabolic distance $d_{p}[(t,x), (s,y)]:=|x-y|+|t-s|^{1/2}$. 

\smallskip

Choose $(t_{0}, x_{0})\in U_{T}$ so that $\max_{\overline{U_{T}}} (u-v)=u(t_{0}, x_{0})-v(t_{0}, x_{0})=\al$. 
By \eqref{e.notbdy}, we must have that
\begin{equation*}
d_{p}[(t_{0}, x_{0}), \partial_{p}U_{T}]
\geq 
c \alpha^{1/ \sigma}
\,.
\end{equation*}
Therefore, interior estimates for caloric functions yield, for every~$k\in\N$ and $r \leq r_0:=c\alpha^{1/ \sigma}$, with $c$ to be chosen, 
\begin{equation}\label{e.caloric}
\| D^k u \|_{L^\infty(Q_r(t_0,x_0))}
\leq 
\frac{C^kk!}{r^k} \| u \|_{L^\infty(U_T)}\leq \frac{C^kk!}{r^k} 
\,.
\end{equation}
Let $p$ denote the second-order Taylor polynomial (in the parabolic sense) for $u$ centered at~$(t_0,x_0)$.
Then
\begin{equation}
\label{e.polysol}
\partial_{t}p-\tr\bigl(\ahom D^{2}p\bigr)=0\,,
\end{equation}
and hence, for every~$r\leq r_0$, by \eqref{e.caloric},
\begin{equation*}
\| u - p \|_{L^\infty(Q_{r}(t_0,x_0))}
\leq
C\alpha^{-3/\sigma} r^3
\,.
\end{equation*}
Let $q(t,x):=p(t,x)-c_{0}\al|x-x_{0}|^{2}$ for $c_{0}$ to be chosen, and then define $Q$ by 
\begin{equation*}
Q(t,x):= 
q(t,x)+
\sum_{i,j=1}^d
(\partial_{x_i}\partial_{x_j}
q)\phi_{ij}^0(x) 
\,.
\end{equation*} 
Since $q$ is a 2nd-order polynomial, this implies, by \eqref{e.Mcorrect} and \eqref{e.polysol}, that 
\begin{align*}
\lefteqn{
\partial_{t}Q-\tr(\A D^{2}Q)=\partial_{t}p-\tr\left(\A\left(D^{2}p-2c_{0}\al I_d+\sum_{i,j=1}^d
(\partial_{x_i}\partial_{x_j}
q)D^{2}\phi_{ij}^{0}\right)\right)
} \qquad &
\\ &
=\partial_{t}p-\tr\left(\A\left(D^{2}p-2c_{0}\al I_d+\sum_{i,j=1}^d
(\partial_{x_i}\partial_{x_j}
p-2c_{0}\al I_d)D^{2}\phi_{ij}^{0}\right)\right)\\
&= \partial_{t}p-\tr\left(\ahom\left(D^{2}p-2c_{0}\al I_d\right)\right)\\
&\leq 2\La dc_{0}\al\leq K\al,
\end{align*}
for $c_{0}(\la,\La, d)$ sufficiently small, chosen in terms of $K=K(\la, \La, d)$.

Observe moreover that
\begin{align*}
(v-q)(t_{0}, x_{0})=(v-u)(t_{0}, x_{0}) &
\leq (v-u)(t,x)\\
& \leq (v-p)(t,x)+\norm{u-p}_{L^{\infty}(Q_{r}(t_{0}, x_{0}))}
\\ & =(v-q)(t,x)-c_{0}\al|x-x_{0}|^{2}+C\al^{-3/\sigma}r^{3}.
\end{align*}
Thus, by \eqref{e.caloric},
\begin{align*}
(v-Q)(t_{0}, x_{0}) & \leq (v-Q)(t,x)-c_{0}\al|x-x_{0}|^{2}
\\ & \qquad +C\al^{-3/\sigma}r^{3}+C\sup_{1\leq i,j\leq d}\left(\frac{1}{r_{0}^{2}}+c_{0}\al\right)\norm{\phi^{0}_{ij}}_{L^{\infty}(\cu_{0})}.
\end{align*}
This implies that for all $(t,x)\in \partial_{p}Q_{r}(t_{0}, x_{0})$, by definition of $r_{0}=c\alpha^{1/ \sigma}$, 
\begin{equation*}
(v-Q)(t_{0}, x_{0})\leq (v-Q)(t,x)-c_{0}\al r^{2}+C\al^{-3/\sigma}r^{3}+C\sup_{1\leq i,j\leq d}\al^{-2/\sigma}\norm{\phi^{0}_{ij}}_{L^{\infty}(\cu_{0})}.
\end{equation*}
If 
\begin{equation*}
C\al^{-3/\sigma}r^{3}+C\sup_{1\leq i,j\leq d}\al^{-2/\sigma}\norm{\phi^{0}_{ij}}_{L^{\infty}(\cu_{0})}< c_{0}\al r^{2}, 
\end{equation*}
then $v-Q$ achieves a strict local minimum in $Q_{r}(t_{0},x _{0})$, and this contradicts the minimum principle. Therefore, we must instead have 
\begin{equation*}
C\al^{-3/\sigma}r^{3}+C\sup_{1\leq i,j\leq d}\al^{-2/\sigma}\norm{\phi^{0}_{ij}}_{L^{\infty}(\cu_{0})}\geq c_{0}\al r^{2},
\end{equation*}
and hence, for $r \leq c\alpha^{1+3/\sigma}$,
\begin{equation*}
\sup_{1\leq i,j\leq d}C\al^{-2/\sigma}\norm{\phi^{0}_{ij}}_{L^{\infty}(\cu_{0})}\geq c_{0}\al r^{2}-C\al^{-3/\sigma}r^{3}\geq \frac{c_{0}}{2}\al r^{2}, 
\end{equation*} 
where $c$ is chosen in term of $C$ and $c_{0}$ in the final inequality. Setting $r := c\alpha^{1+3/\sigma}$, we obtain
\begin{equation*}
\sup_{1\leq i,j\leq d}\norm{\phi^{0}_{ij}}_{L^{\infty}(\cu_{0})}\geq \frac{c_{0}}{2}\al^{1+2/\sigma+2+6/\sigma}=c\al^{3+8/\sigma} \geq c\al^{11/\sigma}
\,. 
\end{equation*}
This completes the proof. 
\end{proof}

Equipped with this result, we are now ready to prove Theorem \ref{t.algrate2}.
\begin{proof}[Proof of Theorem \ref{t.algrate2}]
By \eqref{e.muabp} and the definition of the approximate corrector, we have
\begin{equation}\label{e.approxmu2}
\sup_{\cu_{m}(x)} 3^{-2m}| \phi^{m}_{ij}|\leq C\left(\frac{3^{m}}{\Y}\right)^{-\frac{\theta}{d}}.
\end{equation}

Fix $\norm{g}_{\mathcal{C}^{0,1}(\partial_{p}U_{T})}= M_{0}.$
Assume, without loss of generality, that $U_{T}\subseteq Q_{1}$ and $M_{0}=1$. We will show that there exists $\beta=\beta(\la, \La, d)\in (0,1)$ and $C=C(\la, \La, d)\in [1,\infty)$ so that for every $\ve\in (0, \Y^{-1}]$, 
\begin{equation}\label{e.t1eq1}
\max_{\overline{U_{T}}}\left(\bar{u}-u^{\ve}\right)\leq C(\ve\Y)^{\beta}. 
\end{equation}
An analogous argument, reversing the roles of $u^{\ve}$ and $\bar{u}$, yields the opposite inequality. 

Fix $\ve\in (0,1)$, and let $\sigma=\sigma(\la, \La, d)\in (0,1)$ as in Lemma~\ref{l.toobig}. We now apply Lemma~\ref{l.toobig} with $\A^{\ve}(\cdot):=\A(\cdot/\ve)$. This yields that 
\begin{equation*}
\| u^{\ve} -\bar{u}\|_{L^\infty(U_T)}
\leq 
C
\Bigl( 
\sup_{1\leq i,j \leq d}
\|\phi_{ij}\|_{L^\infty(\cu_0)}  
\Bigr)^{\sigma/11}, 
\end{equation*}
where $\phi_{ij}$ solves 
\begin{equation*}
\begin{cases}
-\tr(\a^{\ve}(e_{i}\otimes e_{j}+D^{2}\phi_{ij}))=-\ahom_{ij}&\text{in $\cu_{0}$},\\
\phi_{ij}=0&\text{on $\partial \cu_{0}$}.
\end{cases}
\end{equation*}
Letting $\ve=3^{-\frac{m}{2}}$, and writing this in terms of $\phi^{m}=3^{-m}\phi_{ij}(3^{\tfrac{m}{2}}\cdot)$, we have 
\begin{equation*}
\| u^{\ve} -\bar{u}\|_{L^\infty(U_T)}
\leq 
C
\Bigl( 
\sup_{1\leq i,j \leq d}
\| 3^{-m}\phi_{ij}^{\frac{m}{2}}\|_{L^\infty(\cu_\frac{m}{2})}  
\Bigr)^{\sigma/11}
\end{equation*}
By \eqref{e.approxmu2}, this implies  
\begin{equation*}
\| u^{\ve} -\bar{u}\|_{L^\infty(U_T)}
\leq 
 C\left(\frac{3^{\sfrac{m}{2}}}{\Y}\right)^{-\frac{\theta \sigma}{11 d}}.
 \end{equation*}
Finally, by choice of $\ve$, this implies 
\begin{equation*}
\max_{\overline{U_{T}}}\left(\bar{u}-u^{\ve}\right)\leq C(\Y\ve)^{\beta}, 
\end{equation*}
up to redefining $\beta=\beta(\la, \La,d)$, which completes the proof of \eqref{e.t1eq}. The final claim is immediate from the definition of $\Y_{p}$ in Remark \ref{r.giveup}.
\end{proof}

For quantitative homogenization results, it is not necessary and moreover inconvenient to represent the scale separation by means of the artificial parameter~$\ep>0$. We therefore restate Theorem \ref{t.algrate2} without this scaling, so that the microscopic scale is of order one and the macroscopic scale is large,  as follows. Observe that if $\tilde{u}^{\ve}$ is a solution of 
\begin{equation*}
\begin{cases}
\partial_{t}\tilde{u}_{\ve}-\tr\left(\A D^{2} \tilde{u}_{\ve}\right)=0&\text{in $Q_{1/\ve}$,}\\
\tilde{u}_{\ve}(0,x)=g(x)&\text{on $\partial_{p}Q_{1/\ve}$}, 
\end{cases}
\end{equation*}
and $\bar{u}_{\ve}$ is a solution of 
\begin{equation*}
\begin{cases}
\partial_{t}\bar{u}_{\ve}-\tr\left(\ahom D^{2} \bar{u}_{\ve}\right)=0&\text{in $Q_{1/\ve}$,}\\
\bar{u}_{\ve}(0,x)=g(x)&\text{on $\partial_{p}Q_{1/\ve}$}, 
\end{cases}
\end{equation*}
then there exists constants $\beta=\beta(\la, \La, d)\in (0,1)$, $C=C(\la, \La, d)\in [1, \infty)$, and a random variable $\Y$ with $\Y=\O_{d}(C)$ such that for every $\ve\in (0, \Y^{-1}]$, 
\begin{equation*}
\begin{aligned}
\norm{\ve^{2}(\tilde{u}_{\ve}-\bar{u}_{\ve})}_{L^{\infty}(Q_{1/\ve})}&\leq C\norm{\ve^{2}g}_{C^{0,1}(Q_{1/\ve})}(\Y\ve)^{\beta}\\
&=C(\ve^{2}\norm{g}_{L^{\infty}(Q_{1/\ve})}+\ve^{2}\ve^{-1}[g]_{C^{0,1}(Q_{1/\ve})})(\Y\ve)^{\beta}.
\end{aligned}
\end{equation*}
Therefore, we have 
\begin{equation}\label{e.unscaledalg}
\norm{(\tilde{u}_{\ve}-\bar{u}_{\ve})}_{L^{\infty}(Q_{1/\ve})}\leq C(\norm{g}_{L^{\infty}(Q_{1/\ve})}+\ve^{-1}[g]_{C^{0,1}(Q_{1/\ve})})(\Y\ve)^{\beta}.
\end{equation}
Similarly, for $p\in (0,d)$, for $\ve\in (0, \Y_{p}^{-1}]$, 
\begin{equation}\label{e.unscaledalg2}
\norm{(\tilde{u}_{\ve}-\bar{u}_{\ve})}_{L^{\infty}(Q_{1/\ve})}\leq C(\norm{g}_{L^{\infty}(Q_{1/\ve})}+\ve^{-1}[g]_{C^{0,1}(Q_{1/\ve})})\ve^{\beta}.
\end{equation}

We next show that we may allow initial data H\"older continuous with any exponent $\sigma \in (0,1]$. We state our result for a specific domain which will be useful in the following section, however this argument can be adapted to an arbitrary Cauchy-Dirichlet problem. 
\begin{lemma}\label{l.refinealg}
Let $\sigma \in (0,1]$, $p\in (0,d)$, and $\beta=\beta(\la,\La, d)\in (0,1]$. Let $\Y_{p}$ be as in Remark \ref{r.giveup}. There exists $C=C(\la, \La, d, p)\in [1, \infty)$ such that for every $\theta\in\left(0, \frac{\beta}{2(1-\beta)}\right)$, for any $r\geq \Y_{p}$, for $u$ the solution of
\begin{equation*}
\begin{cases}
\partial_{t}u-\tr(\A D^{2}u)=0&\text{in $Q_{r^{1+\theta}}$},\\
u(t,x)=0&\text{on $(0, r^{2+2\theta}]\times \partial B_{r^{1+\theta}}$},\\
u(0,x)=u_{0}(x)&\text{on $B_{r^{1+\theta}}$},
\end{cases}
\end{equation*}
and $\bar{u}$ the solution of
\begin{equation*}
\begin{cases}
\partial_{t}\bar{u}-\tr(\ahom D^{2}\bar{u})=0&\text{in $Q_{r^{1+\theta}}$},\\
\bar{u}(t,x)=0&\text{on $(0, r^{2+2\theta}]\times \partial B_{r^{1+\theta}}$},\\
\bar{u}(0,x)=u_{0}(x)&\text{on $B_{r^{1+\theta}}$}, 
\end{cases}
\end{equation*}
with $\supp(u_{0})\subseteq B_{r}$ and $u_{0}\in C^{0,\sigma}(B_{r})$, we have that
\begin{equation*}
\begin{aligned}
&\norm{u-\bar{u}}_{L^{\infty}(Q_{r^{1+\theta}})}\leq C\left(\norm{u_{0}}_{L^{\infty}(B_{r})}+r^{\sigma}\left[u_{0}\right]_{C^{0,\sigma}(B_{r})}\right)r^{-\frac{\beta\sigma}{2}}.
\end{aligned}
\end{equation*}
\end{lemma}
\begin{proof}
We consider the mollification given by \begin{equation*}
\tilde{u}_{0}(x):=(u_{0}\ast \eta_{r^{1-\al}})(x),
\end{equation*}
for some $\al\in (0,1)$ to be chosen, and $\eta_{r^{1-\al}}(x):=r^{(\al-1)d}\eta\left(r^{\al-1}x\right)$, where $\eta$ is the standard mollifier. Extend $u_{0}$ and $\tilde{u}_{0}$ by 0 to all of $\R^{d}$ so that $\tilde{u}_{0}\in C^{0,1}(\R^{d})$ and $supp(\tilde{u}_{0})\subseteq B_{r+r^{1-\al}}=: B_{\tilde{r}}.$ Let $u_{\al}$ solve
\begin{equation*}
\begin{cases}
\partial_{t}u_\al-\tr(\A D^{2}u_{\al})=0&\text{in $Q_{\tilde{r}^{1+\theta}}$},\\
u(t,x)=0&\text{on $(0, \tilde{r}^{2+2\theta}]\times \partial B_{\tilde{r}^{1+\theta}}$},\\
u(0,x)=\tilde{u}_{0}(x)&\text{on $B_{\tilde{r}^{1+\theta}}$},
\end{cases}
\end{equation*}
and $\bar{u}_{\al}$ solve
\begin{equation*}
\begin{cases}
\partial_{t}\bar{u}_{\al}-\tr(\ahom D^{2}\bar{u}_\al)=0&\text{in $Q_{\tilde{r}^{1+\theta}}$},\\\bar{u}_{\al}(t,x)=0&\text{on $(0, \tilde{r}^{2+2\theta}]\times \partial B_{\tilde{r}^{1+\theta}}$},\\
\bar{u}_{\al}(0,x)=\tilde{u}_{0}(x)&\text{on $B_{\tilde{r}^{1+\theta}}$}.
\end{cases}
\end{equation*}
Then for $\Y_{p}$ as in Remark \ref{r.giveup} and Theorem \ref{t.algrate2}, specifically \eqref{e.unscaledalg2} with $\ve:=\tilde{r}^{-(1+\theta)}$, yields that for $\tilde{r}^{(1+\theta)}> r\geq \Y_{p}$, 
\begin{align}
\label{e.u0algrate}
\norm{u_\al-\bar{u}_\al}_{L^\infty(Q_{\tilde{r}^{1+\theta}})}
 & \leq
C\left(\norm{\tilde{u}_{0}}_{L^{\infty}(B_{\tilde{r}^{1+\theta}})}+\tilde{r}^{(1+\theta)}\left[\tilde{u}_{0}\right]_{C^{0,1}(B_{\tilde{r}^{1+\theta}})}\right)\tilde{r}^{-(1+\theta)\beta}
\\ \notag &
=
C\left(\norm{\tilde{u}_{0}}_{L^{\infty}(B_{\tilde{r}})}+\tilde{r}^{(1+\theta)}\left[\tilde{u}_{0}\right]_{C^{0,1}(B_{\tilde{r}})}\right)\tilde{r}^{-(1+\theta)\beta},
\end{align}
where we used that $\supp(\tilde{u}_{0})\subseteq B_{\tilde{r}}$. We also have that 
\begin{align*}
\norm{u_{0}-\tilde{u}_{0}}_{L^{\infty}(\R^{d})}&\leq \sup_{x\in \R^{d}}\int_{\R^{d}}\left|u_{0}(x)-u_{0}(x-y)\right|\eta_{r^{1-\al}}(y)\, dy\\
&=\sup_{x\in \R^{d}} \int_{B_{r^{1-\al}}} |u_{0}(x)-u_{0}(x-y)|\eta_{r^{1-\al}}(y)\, dy\\
&\leq \left[u_{0}\right]_{C^{0,\sigma}(B_{r})}r^{\sigma(1-\al)}\\
&=r^{-\sigma \al}r^{\sigma}\left[u_{0}\right]_{C^{0,\sigma}(B_{r})}.
\end{align*}
By the maximum principle, this implies
\begin{equation}\label{e.tildevcomp}
\norm{u_{\al}-u}_{L^{\infty}(Q_{r^{1+\theta}})}+\norm{\bar{u}_{\al}-\bar{u}}_{L^{\infty}(Q_{r^{1+\theta}})}\leq 2r^{-\sigma \al}r^{\sigma}\left[u_{0}\right]_{C^{0,\sigma}(B_{r})}.
\end{equation}
Moreover, since $D\eta$ has mean 0,  
\begin{align*}
|D\tilde{u}_{0}(x)|=\left|\int_{\R^{d}}u_{0}(y)D \eta_{r^{1-\al}}(x-y)\, dy\right|
&=\left|\int_{B_{r^{1-\al}}(x)}(u_{0}(y)-u_{0}(x))D \eta_{r^{1-\al}}(x-y)\, dy\right|\\
&\leq r^{(1-\al)\sigma}\left[u_{0}\right]_{C^{0,\sigma}(B_{r})}\int |D \eta_{r^{1-\al}}(x-y)|\, dy\\
&\leq Cr^{(1-\al)(\sigma-1)}\left[u_{0}\right]_{C^{0,\sigma}(B_{r})}\\
&=Cr^{\al-1-\al\sigma}r^{\sigma}\left[u_{0}\right]_{C^{0,\sigma}(B_{r})}, 
\end{align*}
which implies 
\begin{equation}\label{e.c0est}
\left[\tilde{u}_{0}\right]_{C^{0,1}(B_{\tilde{r}})}\leq Cr^{\al(1-\sigma)-1}r^{\sigma}\left[u_{0}\right]_{C^{0,\sigma}(B_{r})}.
\end{equation}
We now combine \eqref{e.u0algrate}, \eqref{e.tildevcomp}, and \eqref{e.c0est} to conclude that 
\begin{align*}
\norm{u-\bar{u}}_{L^{\infty}(Q_{r^{1+\theta}})}
& 
\leq 2r^{-\sigma \al}r^{\sigma}\left[u_{0}\right]_{C^{0,\sigma}(B_{r})}\\
&\qquad +C\left(\norm{u_{0}}_{L^{\infty}(B_{r})}+\tilde{r}^{(1+\theta)}Cr^{\al(1-\sigma)-1}r^{\sigma}\left[u_{0}\right]_{C^{0,\sigma}(B_{r})}\right)\tilde{r}^{-(1+\theta)\beta}.
\end{align*}
Using that $1<\Y_{p}\leq r\leq \tilde{r}\leq 2r$, we may simplify this to argue that 
\begin{align*}
\lefteqn{
\norm{u-\bar{u}}_{L^{\infty}(Q_{r^{1+\theta}})}
} \quad & 
\\ & 
\leq 
2r^{-\sigma \al}r^{\sigma}\left[u_{0}\right]_{C^{0,\sigma}(B_{r})} +C\left(\norm{u_{0}}_{L^{\infty}(B_{r})}+r^{(1+\theta)}Cr^{\al(1-\sigma)-1}r^{\sigma}\left[u_{0}\right]_{C^{0,\sigma}(B_{r})}\right)r^{-(1+\theta)\beta}\\
&=2r^{-\sigma \al}r^{\sigma}\left[u_{0}\right]_{C^{0,\sigma}(B_{r})} +C\left(\norm{u_{0}}_{L^{\infty}(B_{r})}+Cr^{\theta+\al(1-\sigma)}r^{\sigma}\left[u_{0}\right]_{C^{0,\sigma}(B_{r})}\right)r^{-(1+\theta)\beta}.
\end{align*}
Now choose~$\al:=(1+\theta)\beta-\theta$. Note that~$\al\in (0,1)$ since~$\theta<\frac{\beta}{2(1-\beta)}$. Thus
\begin{equation*}
-\sigma\al=\theta+\al(1-\sigma)-(1+\theta)\beta, 
\end{equation*}
and hence, we may rewrite 
\begin{align*}
\norm{u-\bar{u}}_{L^\infty(Q_{r^{1+\theta}})}\leq 
C\left(\norm{u_{0}}_{L^{\infty}(B_{r})}+r^{\theta+\al(1-\sigma)}r^{\sigma}\left[u_{0}\right]_{C^{0,\sigma}(B_{r})}\right)r^{-(1+\theta)\beta}. 
\end{align*}
Since $\theta\in \left(0, \frac{\beta}{2(1-\beta)}\right)$, 
\begin{equation*}
\theta+\al(1-\sigma)-(1+\theta)\beta=-\al\sigma=-\beta\sigma+\sigma\theta(1-\beta)\leq -\frac{\beta\sigma}{2}
\end{equation*}
and this yields 
\begin{align*}
\norm{u-\bar{u}}_{L^\infty(Q_{r^{1+\theta}})}
&\leq C\left(\norm{u_{0}}_{L^{\infty}(B_{r})}+r^{\sigma}\left[u_{0}\right]_{C^{0,\sigma}(B_{r})}\right)r^{-\frac{\beta\sigma}{2}}
\end{align*} 
This completes the proof. 
\end{proof}

\subsection{Estimates on stationary approximate correctors} 

For every $\delta>0$, let us define the stationary approximate correctors 
\begin{equation}
\label{e.dvd.def}
\delta^2w^{\delta}_{ij} - \tr(\a(e_{i}\otimes e_{j}+ D^2w^{\delta}_{ij})) =0\;\;\textrm{in}\;\;\R^d, 
\end{equation}
for $1\leq i,j\leq d$.

\begin{proposition}
\label{p.delta.vee.delta}
There exists $\beta(\la, \La, d) \in (0,1]$, $C(\la, \La, d)\in [1,\infty)$, and a random variable~$\Y$ satisfying 
$\Y = \O_{d}(C)$
such that, for every~$\delta \in (0,\Y^{-1} ]$, 
\begin{equation}
\label{e.delta.vee.delta}
\sup_{x\in B_{\delta^{-100}}}|\ahom_{ij}-\delta^{2}w^{\delta}_{ij}(x)|\leq C(\Y^{-1}\delta)^{\beta}
\,.
\end{equation}

\end{proposition}
\begin{proof}
We recall that, by the maximum principle and the Krylov-Safonov estimate,
\begin{equation}\label{e.wholder}
\| \delta^{2}w_{ij}^{\delta} \|_{L^\infty(\Rd)}
+
\delta^{-\sigma}\left[\delta^{2}w_{ij}^{\delta}\right]_{C^{0, \sigma}(B_{2/\delta})}\leq C.
\end{equation}
Let us recall the classical argument (which can be found for example in~\cite{AL}). Since the stationary approximate correctors $\|\delta^{2}w_{ij}^{\delta}\|_{L^{\infty}(\R^{d})}\leq C$ for $C=C(\la, \La, d)$, letting $v_{ij}^{\delta}(x):=\frac{1}{2}x\cdot (e_{i}\otimes e_{j})x+w_{ij}^{\delta}(x)$, we observe that 
\begin{equation*}
\delta^{2}\osc_{B_{2/\delta}}v^{\delta}_{ij}\leq \delta^{2}\osc_{B_{2/\delta}}\frac{1}{2}x\cdot (e_{i}\otimes e_{j})x+\delta^{2}\osc_{B_{2/\delta}}w^{\delta}_{ij}\leq C,
\end{equation*}
and hence by the Krylov-Safonov estimates, since 
\begin{equation*}
-\tr(\A D^{2}v_{ij}^{\delta})=-\delta^{2}w^{\delta}_{ij}\quad\text{in $\R^{d}$},
\end{equation*}
we have
\begin{equation*}
\delta^{-\sigma}[\delta^{2}v^{\delta}_{ij}]_{C^{0, \sigma}(B_{2/\delta})}\leq C. 
\end{equation*}
By another triangle inequality, we may conclude \eqref{e.wholder}. 

\smallskip

We now suppose that there exists $\al>0$, $m\in \mathbb{N}$, and $x_{0}\in \cu_{m-1}$ such that  
\begin{equation*}
\inf_{x\in \cu_{m-1}}\left(\delta^{2}w^{\delta}_{ij}(x)- \ahom_{ij}\right)=\inf_{x\in \cu_{m}}\left( \delta^{2}w^{\delta}_{ij}(x)- \ahom_{ij}\right)=\delta^{2}w^{\delta}_{ij}(x_{0})- \ahom_{ij}=-\al. 
\end{equation*}
If no such $\al>0$, $m$ exists, then $\delta^{2}w^{\delta}_{ij}(x)- \ahom_{ij}\geq 0$. Since $\ahom_{ij}$ is a constant, it follows that $\inf_{x\in \cu_{m}} \delta^{2}w^{\delta}_{ij}=\delta^{2}w^{\delta}_{ij}(x_{0})$. 
By \eqref{e.dvd.def} and \eqref{e.wholder}, this implies that on $\cu_{m}$, for $3^{m}\leq C^{-1}\left(\frac{\al}{2}\right)^{1/\sigma}\frac{1}{\delta}\wedge \frac{2}{\delta}$, 
\begin{equation*}
-\tr(\a(e_{i}\otimes e_{j}+D^{2}w^{\delta}_{ij}))=-\delta^{2}w^{\delta}_{ij}\geq -\delta^{2}w^{\delta}_{ij}(x_{0})-C\delta^{\sigma}3^{m\sigma}=  -\ahom_{ij}-C\delta^{\sigma}3^{m\sigma}+\al\geq-\ahom_{ij}+\frac{\al}{2}. 
\end{equation*}
Now we consider the function $\tilde{v}_{ij}^{\delta}(x):=w^{\delta}_{ij}(x)-\frac{\al}{2}\left(\frac{1}{2}3^{2m}-\frac{1}{2}|x-x_{0}|^{2}\right)-w^{\delta}_{ij}(x_{0})$. Then $\tilde{v}^{\delta}_{ij}$ solves
\begin{equation*}
-\tr(\a (e_{i}\otimes e_{j}+D^{2}\tilde{v}_{ij}^{\delta}))\geq -\ahom_{ij}\quad\text{in $\cu_{m}$},
\end{equation*}
with $\tilde{v}_{ij}^{\delta}\geq 0$ on $\partial \cu_{m}$. By \eqref{e.muabp}, this implies that 
\begin{equation*}
(\sup_{\cu_{m}} -\tilde{v}_{ij}^{\delta})^{d}\leq C3^{2md}\left(\frac{3^{m}}{\Y}\right)^{-\theta},
\end{equation*}
and hence comparing to $-\tilde{v}^{\delta}_{ij}(x_0)$, we have
\begin{equation*}
\al^{d}3^{2md}\leq C3^{2md}\left(\frac{3^{m}}{\Y}\right)^{-\theta}.
\end{equation*}
In particular, if $\delta=c3^{-(100)^{-1}m}$, then, for all $\delta^{-1}\geq \Y$, by definition of $\al$, 
\begin{equation*}
\sup_{x\in B_{\delta^{-100}}}(\ahom_{ij}-\delta^{2}w^{\delta}_{ij}(x))\leq C(\Y^{-1}\delta)^{\frac{\theta}{d}}. 
\end{equation*}
An analogous argument can be made to obtain that 
\begin{equation*}
\sup_{x\in B_{\delta^{-100}}}|\ahom_{ij}-\delta^{2}w^{\delta}_{ij}(x)|\leq C(\Y^{-1}\delta)^{\frac{\theta}{d}}. 
\end{equation*}
By the same argument as in Remark \ref{r.Ystep1} (a union bound), we conclude that for some $\Y=\O_{d}(C)$, for any $\delta\in (0,\mathcal{Y}]$, \eqref{e.delta.vee.delta} holds. 
\end{proof}

\subsection{Estimates for the Cauchy problem}

Our next result in this section allows us to compare solutions of the heterogeneous Cauchy problem with solutions of the homogenized Cauchy problem, using Lemma~\ref{l.refinealg} and deterministic estimates from Section \ref{s.loc}. 
\begin{proposition}\label{prop_homogenize}  Let $\sigma\in (0,1)$, $v_0\in C^{0, \sigma}(B_r)$, and let $v$ and $\bar{v}$ solve
\begin{equation*}
\begin{cases}
\partial_{t}v-\tr(\A D^{2}v)=0&\text{in $(0, \infty)\times \R^{d}$},\\
v(0,\cdot)=v_{0}&\text{on $\R^{d}$},
\end{cases}
\end{equation*}
and 
\begin{equation*}\
\begin{cases}
\partial_{t}\bar{v}-\tr(\ahom D^{2}\bar{v})=0&\text{in $(0, \infty)\times \R^{d}$},\\
\bar{v}(0,\cdot)=v_{0}&\text{on $\R^{d}$}.
\end{cases}
\end{equation*}
Fix $p\in (0,d)$ and let $\beta(\la, \La, d)\in (0,1]$ and $\Y_{p}$ be as in Theorem \ref{t.algrate2}. There exist $C=C(\la, \La, d)>0$ and $\Sigma=\Sigma(\la, \La, d)>0$ such that, for any $\theta\in \left(0, \frac{\beta}{2-\beta}\right)$, and for any $r\geq \Y_{p}\vee\Sigma$,
\begin{align*}
\norm{v-\overline{v}}_{L^{\infty}((0,r^{2+\theta}]\times \R^d)}\leq C\left(\norm{v_0}_{L^\infty(B_{r})}+r^{\sigma}[v_0]_{C^{0,\sigma}(B_{r})}\right)r^{-\frac{\beta\sigma}{2}}+2\sup_{\abs{x}\geq 2^{-1}r}\abs{v_0(x)}.
\end{align*}
\end{proposition}
\begin{proof}  We introduce a smooth bump function
\begin{equation}\label{e.rhodef}
\rho_{r}(x):=\begin{cases}1&\text{if $|x|\leq \frac{1}{2}r$},\\
2\left(1-\frac{|x|}{r}\right)&\text{if $\frac{1}{2}r<|x|\leq r$,}\\
0&\text{if $|x|>r$.}
\end{cases}
\end{equation}
By linearity of the equation, we define $v=v_e+v_i$ into its exterior part
\begin{equation}\label{ph_2} \left\{ \begin{aligned}
& \partial_tv_e-\tr(\A D^2v_e)= 0 && \textrm{in}\;\;(0,\infty)\times\R^d,
\\ & v_e(0,\cdot)=(1-\rho_{r} )v_0 && \textrm{on}\;\;\R^d,
\end{aligned}\right.\end{equation}
and the interior part
\begin{equation}\label{ph_3}\left\{ \begin{aligned}
& \partial_tv_i-\tr(\A D^2v_i)= 0 && \textrm{in}\;\;(0,\infty)\times\R^d,
\\ & v_i(0, \cdot)=\rho_{r}v_0 && \textrm{on}\;\;\R^d,
\end{aligned}\right.\end{equation}
and we perform the analogous decomposition of $\bar{v}=\bar{v}_e+\bar{v}_i$.  As a consequence of the maximum principle and the definition of $\rho_{r}$, 
\begin{equation}\label{ph_04}\norm{v_{e}}_{L^{\infty}([0, \infty)\times \R^{d})}+\norm{\bar{v}_{e}}_{L^{\infty}([0, \infty)\times \R^{d})}\leq 2\sup_{\abs{x}\geq 2^{-1}r}\abs{v_0(x)}.\end{equation}
To estimate $v_{i}$ and $\bar{v}_{i}$, we first note that by the same computation as in Lemma~\ref{le_lem_1}, 
\begin{equation*}
|v_{i}(t,x)|+|\bar{v}_{i}(t,x)|\leq 2\norm{\rho_{r}v_{0}}_{L^{\infty}(B_{r})}\exp\left(-\frac{|x|^{2}-r^{2}}{4\La t}\right). 
\end{equation*}
Therefore, for every~$r\geq \Sigma=\Sigma(\la, \La, d)$ and~$t\in (0, r^{2+\theta}]$, 
\begin{align}\label{ph_4'}
\sup_{|x|\geq r^{1+\theta}} \bigl(  |v_{i}(t,x)|+|\bar{v}_{i}(t,x)| \bigr) &\leq 2\norm{\rho_{r}v_{0}}_{L^{\infty}(B_{r})}\exp\left(-\frac{r^{2+2\theta}-r^{2}}{4\La t}\right) \\ \notag 
&\leq C\norm{v_{0}}_{L^{\infty}(B_{r})}\exp\left(-\frac{r^{2\theta}-1}{4d\La}\right)\\ \notag 
&\leq \norm{v_{0}}_{L^{\infty}(B_{r})}r^{-\frac{\beta\sigma}{2}}, 
\end{align}
where $\Sigma$ is chosen sufficiently large to make the last inequality hold. 

Let $Q_{r}^{\theta}:=(0, r^{2+\theta}]\times B_{r^{1+\theta}}$. Let $u_i$ denote the solution of the Cauchy-Dirichlet problem
\begin{equation}\label{ph_4}\left\{ \begin{aligned}
& \partial_t u_i-\tr(\A D^2u_i)= 0 && \textrm{in}\;\;Q^{\theta}_{r},
\\ & u_i(t,x) = 0 && \textrm{on}\;\; (0,r^{2+\theta})\times \partial B_{r^{1+\theta}},
\\ & u_i(0,\cdot)=\rho_{r} v_0  && \textrm{on}\;\; B_{r^{1+\theta}}, 
\end{aligned}\right.\end{equation}
and let $\overline{u}_i$ denote the solution to the identical Cauchy-Dirichlet problem \eqref{ph_4} defined with the homogenized coefficient $\ahom$.
By \eqref{ph_4'}, the comparison principle yields
\begin{align}\label{ph_5} 
\norm{v_i-u_i}_{L^{\infty}(Q_{r}^{\theta})}+\norm{\overline{v}_i-\bar{u}_{i}}_{L^{\infty}(Q_{r}^{\theta})}&\leq 2\norm{v_{0}}_{L^{\infty}(B_{r})}r^{-\frac{\beta\sigma}{2}}.
\end{align}
We next apply Lemma~\ref{l.refinealg}, to conclude that since $r\geq \Y_{p}$, 
\begin{equation*}
\norm{u_{i}-\bar{u}_{i}}_{L^{\infty}(Q_{r^{1+\theta}})}\leq C\left(\norm{\rho_{r}v_{0}}_{L^{\infty}(B_{r})}+r^{\sigma}\left[\rho_{r}v_{0}\right]_{C^{0,\sigma}(B_{r})}\right)\Y^{\beta}r^{-\frac{\beta\sigma}{2}}.
\end{equation*}
We now estimate the H\"older norm of $\rho_{r}v_{0}$, using the definition of $\rho_{r}$ in \eqref{e.rhodef} to obtain
\begin{align*}
r^\sigma[\rho_{r}v_0]_{C^{0, \sigma}(B_{r})}&\leq r^\sigma\norm{\rho_{r}}_{L^{\infty}(B_r)}[v_0]_{C^{0, \sigma}(B_{r})}+r^{\sigma}[\rho_{r}]_{C^{0, \sigma}(B_{r})}\norm{v_0}_{L^{\infty}(B_r)}\\ \notag 
&\leq r^\sigma[v_0]_{C^{0, \sigma}(B_{r})}+2\norm{v_0}_{L^\infty(B_r)}.
\end{align*}
This implies 
\begin{equation}\label{ph_6}
\norm{u_{i}-\bar{u}_{i}}_{L^{\infty}(Q_{r^{1+\theta}})}\leq C\left(\norm{v_{0}}_{L^{\infty}(B_{r})}+r^{\sigma}\left[v_{0}\right]_{C^{0,\sigma}(B_{r})}\right)r^{-\frac{\beta\sigma}{2}}.
\end{equation}
Combining \eqref{ph_4'}, \eqref{ph_5}, \eqref{ph_6}, and the triangle inequality yields
\begin{align*}
\norm{v_{i}-\bar{v}_{i}}_{L^{\infty}([0, r^{2+\theta}]\times \R^{d})}&= \max\left(\norm{v_{i}-\bar{v}_{i}}_{L^{\infty}(Q_{r}^{\theta})}, \sup_{(0, r^{2+\theta}]\times (B_{r^{1+\theta}})^{c}} |v_{i}-\bar{v}_{i}|\right)\\ \notag 
&=C\left(\norm{v_{0}}_{L^{\infty}(B_{r})}+r^{\sigma}\left[v_{0}\right]_{C^{0,\sigma}(B_{r})}\right)r^{-\frac{\beta\sigma}{2}}.
\end{align*}
The result now follows from the previous display,~\eqref{ph_04} and the triangle inequality.
\end{proof}

\section{Convergence of the parabolic Green function}\label{s.PGF}

Before proving Theorem~\ref{t.realthm}, we first observe the following result for solutions of the homogenized Cauchy problem with initial data mean-zero and bounded above by a Gaussian. 
\begin{lemma}\label{l.mean0}
Suppose that $w$ is a solution of the initial-value problem
\begin{equation}
\begin{cases}
\partial_{t}w-\tr(\ahom D^{2}w)=0&\text{in $(0, \infty)\times \R^{d}$},\\
w(0,\cdot)=w_{0} &\text{on $\R^{d}$},
\end{cases}
\end{equation} 
with the initial condition satisfying, for some $s\in(0,\infty)$, 
\begin{equation}
\label{e.w0.cond}
\int_{\Rd} w_0(x)\,dx = 0
\qquad \mbox{and} \qquad 
\bigl|w_0(x) \bigr| 
\leq
M s^{-\frac d2} 
\exp \biggl( -\frac{ |x|^2}{s} \biggr)
\,.
\end{equation}
Then there exists $C=C(\La, d)\in[1,\infty)$ such that, for every $t\geq s$ and $x\in \R^d$,
\[\abs{w(t,x)}\leq CMs^\frac{1}{2} t^{-\frac{d+1}{2}}\exp\left(-\frac{\abs{x}^2}{16 \La t}\right).\]
\end{lemma}
\begin{proof}  
Up to an affine change of coordinates, we may assume without loss of generality that $\ahom$ is the identity matrix, and that its Green function $\bar{P}(t,x)$ is the heat kernel.  We first observe that, since $w_0$ has mean-zero, for every $(t,x)\in (0,\infty)\times \R^d$,
\begin{align*}w(t,x) & = \int_{\R^d} w_0(y) \overline{P}(t,x-y) \, dy = \int_{\R^d} w_0(y)\left(\overline{P}(t,x-y)-\overline{P}(t,x)\right)\, dy
\\ & =\frac{1}{(4\pi)^{\frac{d}{2}}}t^{-\frac{d+1}{2}}\int_{\R^d}\int_0^1 w_0(y)\frac{(x-ry)\cdot y}{t^{\frac{1}{2}}}\exp\left(-\frac{\abs{x-ry}^2}{4t}\right)\, dr \, dy.
\end{align*}
The assumptions on $w_0$ and the fact that there exists $c\in(0,\infty)$ such that $r\exp(-r^2)\leq c\exp(-2^{-1}r^2)$ imply that
\begin{align}\label{lmean_0}
\abs{w(t,x)} & \leq CMt^{-\frac{d+1}{2}}\int_{\R^d}\int_0^1 s^{-\frac{d}{2}}\abs{y}\frac{\abs{x-ry}}{t^{\frac{1}{2}}}\exp\left(-\frac{\abs{y}^2}{s}-\frac{\abs{x-ry}^2}{4t}\right)
\\ \nonumber & \leq CM t^{-\frac{d+1}{2}}\int_{\R^d}\int_0^1s^{-\frac{d}{2}}\abs{y}\exp\left(-\frac{\abs{y}^2}{s}-\frac{\abs{x-ry}^2}{8t}\right)
\\ \nonumber &  = CMt^{-\frac{d+1}{2}}\int_{\R^d}\int_0^1s^{-\frac{d}{2}}\abs{y}\exp\left(-\frac{\abs{y}^2}{s}-\frac{\abs{x}^2}{8t}+\frac{x\cdot ry}{4t}-\frac{r^2\abs{y}^2}{8t}\right).
\end{align}
After using Young's inequality to write
\[\abs{x\cdot ry}\leq\frac{\abs{x}^2}{4}+r^2\abs{y}^2,\]
we have 
\[\abs{w(t,x)} \leq CM t^{-\frac{d+1}{2}}\exp\left(-\frac{\abs{x}^2}{16t}\right)\int_{\R^d} s^{-\frac{d}{2}}\abs{y}\exp\left(-\frac{\abs{y}^2}{s}+\frac{\abs{y}^2}{8t}\right).\]
Therefore, for every $t\geq s$, for $c\in(0,\infty)$ depending on $d$,
\begin{align*}
\abs{w(t,x)} & \leq CM t^{-\frac{d+1}{2}}\exp\left(-\frac{\abs{x}^2}{16t}\right)\int_{\R^d} s^{-\frac{d}{2}}\abs{y}\exp\left(-\frac{7\abs{y}^2}{8s}\right)\\
&\leq CMs^\frac{1}{2} t^{-\frac{d+1}{2}}\exp\left(-\frac{\abs{x}^2}{16t}\right),
\end{align*}
which completes the proof.
\end{proof}

Equipped with this result, we are now ready to prove Theorem \ref{t.realthm}. In broad terms, the proof relies on an iterative, inductive argument in time intervals. In the ``spatial bulk'' (i.e. where $|x|\leq \sqrt{t\log t}$) we use the quantitative homogenization results for the Cauchy-Dirichlet problem to track how the heterogeneous, random solution is algebraically close to the homogenized one (with stretched exponentially good probability). In the tails, we use deterministic estimates on parabolic equations to bound the profiles of the random and homogenized solutions. Based on the uniformity of our estimates and the strong quantitative bounds we have available to us, we can iterate this argument in time to obtain a global in time result, which is Theorem \ref{t.realthm}.
\begin{proof}[Proof of Theorem \ref{t.realthm}]  
Let $\ga:=\tfrac{\beta\sigma}{8}\wedge \tfrac{1}{4}$, for $\beta$ as in Theorem \ref{t.algrate2} and $\sigma=\sigma(\la, \La, d)\in (0,1)$ the exponent from the Krylov-Safonov esimate. We will fix constants $\Sigma,\theta, K,A_1 \in [1,\infty)$ and $A_0 \in (0,1]$ to be determined below. 
Just to clarify that these undetermined constants will not be chosen in a circular way, we mention that~$A_1$ will be chosen first to be very large, depending only on~$(d,\lambda,\Lambda)$; we choose~$A_0$ very small and~$\theta$ very large, each depending on~$(A_1,d,\lambda,\Lambda)$; the constant~$K$ is chosen, to be very large depending on~$(A_0,A_1,d,\lambda,\Lambda)$; finally, $\Sigma$ is chosen last, and depends on the previous constants.

\smallskip

Define 
\begin{equation*}
c_{t}[v_{0}]:=\int_{\R^{d}} v(t,x)\, dx.
\end{equation*} 
The proof will be based on an induction argument in time. We fix a constant~$R$ as in the statement of the theorem, with~$R \geq \Sigma \vee \Y$. 

\smallskip

\emph{Step 1.} 
We make some crude estimates on~$v$ which will serve as the base case in our induction argument. By the exact same computation as in Lemma~\ref{le_lem_1}, for all $t \in (0,\infty)$ and~$x\in\Rd$, 
\begin{equation}
\label{e.crude.v}
|v(t,x)| 
\leq 
M
R^{-d} 
\exp\biggl( - \frac{|x|^2}{4\La(t+R^{2})} \biggr)
\,.
\end{equation}
In particular, \eqref{e.crude.v} implies that for all $s\in [R^{2}, KR^{2}]$, 
\begin{equation}\label{e.baseint}
|c_{s}[v_{0}]|\leq \int_{\R^{d}} MR^{-d}\exp\biggl( - \frac{|x|^2}{4\La (K+1)R^{2}} \biggr)\, dx =(4\pi \La (K+1))^{\frac{d}{2}}M.
\end{equation}
Moreover, for every~$s\in [R^{2}, KR^{2}]$,
\begin{align}\label{e.baseest}
\lefteqn{
|v(s,x)-c_{s} [v_{0}] \overline{P}(s,x)| 
} \qquad 
 \notag \\  & 
\leq |v(s,x)|+|c_{s}[v_{0}]|\overline{P}(s,x)
\notag \\  &
\leq MR^{-d}\exp\biggl( - \frac{|x|^2}{4\La(s+R^{2})} \biggr)+
\biggl(\frac{\La (K+1)}{\lambda} \biggr)^{\frac{d}{2}}M s^{-\frac{d}{2}}\exp\left(-\frac{|x|^{2}}{4\La s}\right)
\notag \\  &
\leq 
2\biggl( \frac{\La (K+1)}{\la}\biggl)^{\frac{d}{2}}Ms^{-\frac{d}{2}}\exp\left(-\frac{|x|^{2}}{A_{1}s}\right)
\,,
\end{align}
provided that~$A_{1}\geq 8\La$. 

\smallskip

\emph{Step 2.} 
The induction setup. 
Our induction hypothesis asserts that, for some time~$t_0 \geq KR^2$, we have, for every $s \in [K^{-1}t_0, t_{0}]$,
\begin{equation}\label{e.indc1}
|c_{s}[v_{0}]|\leq (4\pi \La (K+1))^{\frac{d}{2}+1} M \prod_{j=0}^{\infty}\left(1+K^{-\ga j}\right)
\end{equation}
and 
\begin{equation}\label{e.indc2}
\bigl| v(s,x)- c_{s}[v_{0}] \overline{P}(s, x) \bigr|
\leq 
2\biggl( 
\frac{\La(K+1)}{\la} \biggr)^{\!\!\frac d2+1}
M
\biggl( \frac{s}{R^2} \biggr)^{- \gamma}
s^{-\frac{d}{2}}\exp \biggl( -\frac{|x|^2}{A_1s} \biggr)\,.
\end{equation}
Observe that the induction hypothesis is valid for $t_{0}=KR^2$ in \eqref{e.baseint} and \eqref{e.baseest} respectively.  

\smallskip

We will show that the validity of these statements for $t_0$ implies their validity for~$Kt_0$, where $K\geq 10$ is a large constant to be determined.
To that end, select~$t \in [t_0,Kt_0]$ and set~$s:=K^{-1}t \in [K^{-1}t_0,t_0]$. 

\smallskip

We approximate~$v - c_{s}[v_{0}] \overline{P}$ by the solution~$Q(\cdot;s)$ of the initial-value problem
\begin{align*}
\left\{
\begin{aligned}
&
\partial_{t}Q(\cdot, \cdot;s) -\tr (\a D^{2}Q(\cdot, \cdot;s))=0
& \text{in} & \ (s,\infty) \times \Rd\,, 
\\ & 
Q(s,\cdot;s) = v(s,\cdot) - c_s[v_0] \overline{P}(s,\cdot) &\text{on} & \ \R^{d}\,.
\end{aligned}
\right.
\end{align*}
We also let~$\bar{Q}(\cdot,\cdot; s)$ be the solution of the homogenized problem 
\begin{align*}
\left\{
\begin{aligned}
&
\partial_{t}\bar{Q}(\cdot, \cdot;s) -\tr (\ahom D^{2}\bar{Q}(\cdot, \cdot;s))=0
& \mbox{in} & \ (s,\infty) \times \Rd\,, 
\\ & 
\bar{Q}(s,\cdot;s) = v(s,\cdot) - c_s[v_0]  \overline{P}(s,\cdot) & \text{on} & \ \R^{d}\,.
\end{aligned}
\right.
\end{align*}
By the triangle inequality,
\begin{align}
\label{e.split}
\lefteqn{
\bigl| v(t,x) - c_{s}[v_{0}] \overline{P}(t,x) \bigr| 
} \quad & 
\notag \\  &
\leq
\bigl| v(t,x) - Q(t,x;s) - c_{s}[v_{0}] \overline{P}(t,x) \bigr| 
+
|Q(t,x;s)  - \bar{Q}(t,x;s) | 
+
| \bar{Q}(t,x;s) | 
\,.
\end{align}

\emph{Step 3.} 
We estimate the last term on the right side of~\eqref{e.split}. The claim is  
\begin{equation}
\label{e.Q.bar.evolve}
| \bar{Q}(t,x;s) | 
\leq
CK^{\frac{d}{2}+1}MA_1^{\frac {d+1}{2}}
\left(\frac{t}{R^{2}}\right)^{\!\!-\gamma}\!
\biggl( \frac ts \biggr)^{\!\!-\frac12+\ga} \!
\biggl( \frac t{t-s} \biggr)^{\!\!\frac{d+1}2}\!
 t^{-\frac d2}\exp\left(-\frac{|x|^{2}}{ 16\La t}\right).
\end{equation}
Observe that, by definition of $c_{s}[v_{0}]$, the initial condition $v(s,x)-c_s[v_0] \overline{P}(s,x)$ in the PDE for~$\overline{Q}(\cdot\;;s)$ has zero mean. 
In view of~\eqref{e.indc2}, an application of Lemma~\ref{l.mean0} yields, for every~$t\geq (A_1+1) s$, 
\begin{equation*}
| \bar{Q}(t,x;s) | 
\leq
CK^{\frac{d}{2}+1}MA_1^{\frac {d+1}{2}}
\left(\frac{s}{R^{2}}\right)^{\!\!-\ga}s^{\frac{1}{2}}(t-s)^{-\frac{d+1}{2}}\exp\left(-\frac{|x|^{2}}{ 16\La (t-s)}\right). 
\end{equation*} 
We next compute
\begin{align*}
\left(\frac{s}{R^{2}}\right)^{-\ga}s^{\frac{1}{2}}(t-s)^{-\frac{d+1}{2}}
&
=
\left(\frac{t}{R^{2}}\right)^{\!\!-\ga}
\biggl( \frac ts \biggr)^{\!\!-\frac12+\ga}
\biggl( \frac t{t-s} \biggr)^{\!\!\frac{d+1}2} t^{-\frac{d}{2}}.
\end{align*}
Combining this with the above display yields \eqref{e.Q.bar.evolve}.

\smallskip

\emph{Step 4.} We estimate the second term on the right side~\eqref{e.split}. The precise claim is that  
\begin{align}
\label{e.step4display}
 |\bar{Q}(t,x; s)-Q(t,x; s)|
\leq 
CK^{\frac{d}{2}+1}M\left(\frac{t}{R^{2}}\right)^{\!\! - \gamma}\!\left(\frac{t}{s}\right)^{\!\!\ga+\frac{d}{2}} \!\biggl(\frac{\theta t}{s}\biggr)t^{-\ga}t^{-\frac{d}{2}}\exp\left(-\frac{|x|^{2}}{A_{1}t}\right).
\end{align}
Here we apply Proposition~\ref{prop_homogenize} with $r=\sqrt{\theta t\log t}$, which since $\theta\in[1,\infty)$ implies $t-s\leq t\leq r^2$.  Since $t\geq Ks \geq R^{2}\geq \Y^{2}$, this implies that $\sqrt{\theta t\log t}\geq \Y$. This, and the fact that $\tfrac{\beta\sigma}{2}\geq 4\ga$, yields
\begin{align}
\label{e.shomog}
\lefteqn{
\norm{\bar{Q}(\cdot, \cdot; s) - Q(\cdot,\cdot; s)}_{L^{\infty}([2^{-1}t, t]\times\R^{d})}
} \qquad & 
 \notag \\  &
\leq 
C\Big(\norm{v(s,\cdot)-c_{s}[v_{0}]\overline{P}(s,\cdot)}_{L^{\infty}(B_{\sqrt{\theta t\log t}})}
 \notag \\  & \qquad
+(\theta t\log t)^{\frac{\sigma}{2}}\left[v(s,\cdot)-c_{s}[v_{0}]\overline{P}(s,\cdot)\right]_{C^{0, \sigma}(B_{\sqrt{\theta t\log t}})}\Big)(\theta t\log t)^{-4\ga}
 \notag \\  &\qquad
+2\sup_{|x|\geq 2^{-1}\sqrt{\theta t\log t}}|v(s,x)-c_{s}[v_{0}]\overline{P}(s,x)|.
\end{align}
We next invoke \eqref{e.indc2} to argue that since $t\geq Ks\geq s$, 
\begin{equation*}
\begin{aligned}
\sup_{|x|\geq 2^{-1}\sqrt{\theta t\log t}}|v(s,x)-c_{s}[v_{0}]\overline{P}(s,x)|&\leq CK^{\frac{d}{2}+1}M  \left(\frac{s}{R^{2}}\right)^{\! - \gamma}s^{-\frac d2 }\exp \biggl(-\frac{\theta t \log t}{4A_1s} \biggr)\\
&\leq CK^{\frac{d}{2}+1}M  \left(\frac{s}{R^{2}}\right)^{\! - \gamma}s^{-\frac d2 }t^{-\frac{\theta}{4A_{1}}},
\end{aligned}
\end{equation*}
and we will choose $\theta:= 12 A_1\ga$, allowing us to conclude that
\begin{equation}\label{e.stails}
\sup_{|x|\geq 2^{-1}\sqrt{\theta t\log t}}|v(s,x)-c_{s}[v_{0}]\overline{P}(s,x)|\leq CK^{\frac{d}{2}+1} M  \left(\frac{s}{R^{2}}\right)^{ \!- \gamma}s^{-\frac d2 }t^{-3\ga}.
\end{equation}
This is the second main term in the right hand side of~\eqref{e.shomog}, and we turn now to first term. 
By~\eqref{e.indc2},
\begin{equation}\label{e.sinfdif}
\norm{v(s,\cdot)-c_{s}[v_{0}]\overline{P}(s,\cdot)}_{L^{\infty}(B_{\sqrt{\theta t\log t}})}
\leq 
CK^{\frac{d}{2}+1}M\left(\frac{s}{R^{2}}\right)^{\! - \ga}s^{-\frac{d}{2}},
\end{equation}
and moreover, by \eqref{e.indc1} and \eqref{e.indc2}, we obtain
\begin{equation}\label{e.sinf}
\norm{v}_{L^{\infty}([2^{-1}s,s]\times B_{\sqrt{\theta t\log t}})}
\leq CK^{\frac{d}{2}+1}Ms^{-\frac{d}{2}}
\,.
\end{equation}
We next apply Lemma~\ref{l.ksref} with $r_{1}=\sqrt{\theta t\log t}$ and $r=\sqrt{s}$ to obtain, in view of~\eqref{e.sinf}, 
\begin{equation*}
\begin{aligned}
(\theta t\log t)^{\frac{\sigma}{2}}\left[v(s,\cdot)\right]_{C^{0, \sigma}(B_{\sqrt{\theta t\log t}})}&\leq C\left(\frac{\theta t\log t}{s}\right)^{\frac{\sigma}{2}}\norm{v}_{L^{\infty}([2^{-1}s,s]\times B_{\sqrt{\theta t\log t}})}
\\ &
\leq 
CK^{\frac{d}{2}+1}M\left(\frac{\theta t\log t}{s}\right)^{\frac{\sigma}{2}}
s^{-\frac{d}{2}}.
\end{aligned}
\end{equation*}
Moreover, by \eqref{e.indc1}, 
\begin{equation}\label{e.ezks}
\begin{aligned}
c_{s}[v_{0}](\theta t\log t)^{\frac{\sigma}{2}}\left[\overline{P}(s,\cdot)\right]_{C^{0, \sigma}(B_{\sqrt{t\log t}})}&\leq c_{s}[v_{0}](\theta t\log t)^{\frac{1}{2}}\left[\overline{P}(s, \cdot)\right]_{C^{0,1}(B_{\sqrt{\theta t\log t}})}\\
&\leq Cc_{s}[v_{0}]\left(\theta t \log t\right) s^{-\frac{d}{2}-1}\\
&\leq CK^{\frac{d}{2}+1}M\left(\frac{\theta t \log t}{s}\right) s^{-\frac{d}{2}}.
\end{aligned}
\end{equation}
Combining the previous two displays yields
\begin{equation}\label{e.kss'}
(\theta t\log t)^{\frac{\sigma}{2}}\left[v(s,\cdot)-c_{s}[v_{0}]\overline{P}(s,\cdot)\right]_{C^{0, \sigma}(B_{\sqrt{\theta t\log t}})}\leq CK^{\frac{d}{2}+1}M\left(\frac{\theta t\log t}{s}\right)s^{-\frac{d}{2}}.
\end{equation}
Finally, inserting~\eqref{e.stails},~\eqref{e.sinfdif}, and~\eqref{e.kss'} into~\eqref{e.shomog} yields
\begin{equation*}
\begin{aligned}
&\norm{\bar{Q}(\cdot, \cdot; s) - Q(\cdot,\cdot; s)}_{L^{\infty}([2^{-1}t, t]\times\R^{d})}\\
& \quad 
\leq 
CK^{\frac{d}{2}+1}M
\left(\frac{\theta t\log t}{s}\right)
 s^{-\frac{d}{2}}(\theta t\log t)^{-4\ga}
+ CK^{\frac{d}{2}+1}
M  \left(\frac{s}{R^{2}}\right)^{ -  \gamma}s^{-\frac d2 }
t^{-3\ga}.
\end{aligned}
\end{equation*}
Now choosing the deterministic lower bound~$\Sigma$ on~$R$ to be sufficiently large, we have that for~$t\geq s\geq \Sigma,$
\begin{equation*}
(\log t)^{1-4\ga}\leq t^{\ga},
\end{equation*}
and this implies 
\begin{equation}
\norm{\bar{Q}(\cdot, \cdot; s) - Q(\cdot,\cdot; s)}_{L^{\infty}([t/2, t]\times\R^{d})}
\leq 
CK^{\frac{d}{2}+1}M\left[\left(\frac{\theta t}{s}\right)+ \left(\frac{s}{R^{2}}\right)^{ - \ga}\right]t^{-3\ga}s^{-\frac{d}{2}}.
\end{equation}
Now we use this to argue that for~$|x| \leq A_0 \sqrt{t \log t}$, since $R\geq \Sigma$, for $t\geq Ks\geq R^{-2}s$, 
\begin{align}\label{e.intfrac}
\lefteqn{
\bigl| \bar{Q}(t,x;s) - Q(t,x;s) \bigr| 
} \qquad & 
\\ \notag   &
\leq 
CK^{\frac{d}{2}+1}M\left[\left(\frac{\theta t}{s}\right)t^{-\ga}+ \left(\frac{s}{R^{2}}\right)^{ \!-\gamma}\right]t^{-2\ga}s^{-\frac{d}{2}}\exp\left(-\frac{|x|^{2}}{A_{1}t}\right)\exp\left(\frac{A_{0}^{2}\log t}{A_{1}}\right) \\ \notag 
&=CK^{\frac{d}{2}+1}M\left[\left(\frac{\theta t}{s}\right)t^{-\ga}+ \left(\frac{s}{R^{2}}\right)^{ - \gamma}\right]t^{-2\ga}s^{-\frac{d}{2}}\exp\left(-\frac{|x|^{2}}{A_{1}t}\right)t^{\frac{A_{0}^{2}}{A_{1}}}\\ \notag 
&\leq CK^{\frac{d}{2}+1}M\left(\frac{s}{R^{2}}\right)^{\! - \gamma}\biggl(\frac{\theta t}{s}\biggr)t^{-2\ga}s^{-\frac{d}{2}}\exp\left(-\frac{|x|^{2}}{A_{1}t}\right)t^{\frac{A_{0}^{2}}{A_{1}}}
\,. 
\end{align}
We now choose $A_{0}:= (A_1\ga)^{\frac12}$ so that $A_{1}^{-1}A_{0}^{2}= \ga$, which implies that, for $|x| \leq A_0 \sqrt{t \log t}$ and $t\geq Ks$,
\begin{align}\label{e.Q.bar.Q}
\bigl| \bar{Q}(t,x;s) - Q(t,x;s) \bigr|
\notag&\leq 
CK^{\frac{d}{2}+1}M\left(\frac{s}{R^{2}}\right)^{\! - \gamma}\biggl(\frac{\theta t}{s}\biggr)t^{-\ga}s^{-\frac{d}{2}}\exp\left(-\frac{|x|^{2}}{A_{1}t}\right)\\
&=CK^{\frac{d}{2}+1}M\left(\frac{t}{R^{2}}\right)^{\! - \gamma}\left(\frac{t}{s}\right)^{\!\ga+\frac{d}{2}}\biggl(\frac{\theta t}{s}\biggr)t^{\!-\ga}t^{-\frac{d}{2}}\exp\left(-\frac{|x|^{2}}{A_{1}t}\right)
\,.
\end{align}
This is the bulk estimate. For the tails, we estimate, for $|x| \geq A_0 \sqrt{t\log t}$, using \eqref{e.indc2}, uniform ellipticity, and the comparison principle, 
\begin{align*}
\lefteqn{
\bigl| Q (t,x;s)\bigr| 
} \qquad & 
\\
&\leq
2\biggl( 
\frac{\La(K+1)}{\la} \biggr)^{\frac d2+1} M\left(\frac{s}{R^{2}}\right)^{ \!-\gamma}s^{-\frac{d}{2}}
\int_{\Rd}
\exp\biggl( - \frac{|z|^2}{A_1s} - \frac {|x-z|^2}{4\La(t-s)} \biggr)\,dz\\
&=
CK^{\frac d2+1} M\left(\frac{s}{R^{2}}\right)^{ \!-\gamma}s^{-\frac{d}{2}}
\int_{\Rd}
\exp\biggl( - \frac{|z|^2}{A_1s} - \frac {|x-z|^2}{8\La(t-s)} - \frac {|x-z|^2}{8\La(t-s)} \biggr)\,dz
\,.
\end{align*}
We next observe that, 
for sufficiently large~$A_1$, depending on $(d, \la, \La)$, for all $t\geq Ks$, 
\begin{align*}
\frac{|z|^2}{A_1s} + \frac {|x-z|^2}{8\La(t-s)}
&\geq 
\frac{K|x|^2}{2A_1t}\,.
\end{align*}
Therefore we obtain, for~$|x| \geq A_0 \sqrt{t \log t}$,
\begin{align*}
\bigl| Q (t,x;s)\bigr|
& 
\leq
CK^{\frac d2+1} M\left(\frac{s}{R^{2}}\right)^{ \!-\gamma}s^{-\frac{d}{2}}
\exp\biggl( -\frac{K|x|^2}{2A_1t} \biggr) 
\int_{\Rd}
\exp\biggl( - \frac {|x-z|^2}{8\La(t-s)} \biggr)\,dz
\\ & 
\leq
CK^{\frac{d}{2}+1}M\left(\frac{s}{R^{2}}\right)^{ \!-\gamma}s^{-\frac{d}{2}}t^{\frac d2} 
\exp\biggl( -\frac{K|x|^2}{2A_1t} \biggr) 
\\ & 
\leq 
CK^{\frac{d}{2}+1}M\left(\frac{s}{R^{2}}\right)^{ \!-\gamma}s^{-\frac{d}{2}} t^{\frac d2} \exp\left(-
\frac{A_0^2(K-2)}{2A_{1}}
\log t \right)
\exp\biggl( -\frac{|x|^2}{A_1t} \biggr) 
\\ & 
=
CK^{\frac{d}{2}+1}M\left(\frac{s}{R^{2}}\right)^{ \!-\gamma}s^{-\frac{d}{2}}t^{\frac d2-\frac{A_0^2(K-2)}{2A_{1}}}
\exp\biggl( -\frac{|x|^2}{A_1t} \biggr),
\end{align*}
where in the penultimate line we used the condition that~$|x| \geq A_0 \sqrt{t\log t}$.  
If we require that~$K$ be sufficiently large that~$A_0^2(K-2)(2A_{1})^{-1} \geq d + 2\ga$, then we obtain
\begin{align*}
\bigl| Q (t,x;s)\bigr|
&\leq
CK^{\frac{d}{2}+1}M\biggl(\frac{s}{R^{2}}\biggr)^{ \!\!-\gamma}s^{-\frac{d}{2}} t^{-\frac d2-2\ga}
\exp\biggl( -\frac{|x|^2}{A_1t} \biggr)\\
&\leq CK^{\frac{d}{2}+1}M\left(\frac{t}{R^{2}}\right)^{ \!-\gamma}\biggl(\frac{t}{s}\biggr)^{\!\!\ga}t^{-\ga}t^{-\frac d2}
\exp\biggl( -\frac{|x|^2}{A_1t} \biggr)\,.
\end{align*}
An estimate for~$\bar{Q}$ can obtained in the same way. We therefore obtain, for every~$|x| \geq A_0 \sqrt{t \log t}$,  
\begin{align}
\label{e.fuck.off.tails}
\bigl| \bar{Q}(t,x;s) - Q(t,x;s) \bigr| 
&
\leq 
\bigl| \bar{Q}(t,x;s)\bigr| 
+ \bigl|Q(t,x;s) \bigr|\notag
\\  
&\leq CK^{\frac{d}{2}+1}M\left(\frac{t}{R^{2}}\right)^{ \!\!-\gamma}\left(\frac{t}{s}\right)^{\ga}t^{-\ga}t^{-\frac d2}
\exp\biggl( -\frac{|x|^2}{A_1t} \biggr), 
\end{align}
by choosing $K(d, \la, \La)$ sufficiently large. Combining~\eqref{e.Q.bar.Q} and~\eqref{e.fuck.off.tails}, we obtain \eqref{e.step4display}.

\smallskip

\noindent\emph{Step 5.}
We estimate the first term on the right side of~\eqref{e.split}. The claim is that 
\begin{multline}
\label{e.vQcs}
\bigl| v(t,x) - Q(t,x;s) - c_{s}[v_{0}] \overline{P}(t,x) \bigr|\\
 \leq CK^{\frac{d}{2}+1}M\left(\frac{t}{R^{2}}\right)^{\!\! - \ga}\left(\frac{t}{s}\right)^{\!\!\ga+\frac{d}{2}}\biggl(\frac{\theta t}{s}\biggr)t^{-\ga}t^{-\frac{d}{2}}\exp\left(-\frac{|x|^{2}}{A_{1}t}\right)
.
\end{multline}
The function~$v - Q(\cdot, \cdot;s)$ satisfies the heterogeneous equation with initial data (at time~$s$) equal to~$c_{s}[v_{0}]\overline{P}(s,x)$. Observe that, by choice of $A_{1}$, 
\begin{align}\label{e.ezbd}
|c_{s}[v_{0}]\overline{P}(s,x)|
\leq CK^{\frac{d}{2}+1}Ms^{-\frac{d}{2}}\exp\left(-\frac{|x|^{2}}{A_{1}s}\right). 
\end{align}
This implies that the initial condition is bounded above by (almost) the same right hand side as \eqref{e.indc2}, and we can obtain \eqref{e.vQcs} by an identical analysis as above. Indeed an application of Proposition \ref{prop_homogenize}, using identical choices of $A_{0}$, $\theta$ as before, implies that 
\begin{equation*}
\|v - Q(\cdot;s) - c_{s}[v_{0}] \overline{P}\|_{L^{\infty}([2^{-1}t, t]\times\R^{d})}
 \leq CK^{\frac{d}{2}+1}M\biggl(\frac{\theta t}{s}\biggr)t^{-3\ga}s^{-\frac{d}{2}}.
\end{equation*}
Now using that $t\geq Ks\geq sR^{-2}$, we have 
\begin{equation*}
\|v- Q(\cdot;s) - c_{s}[v_{0}] \overline{P}\|_{L^{\infty}([2^{-1}t, t]\times\R^{d})}\\
\leq CK^{\frac{d}{2}+1}
M \left(\frac{s}{R^{2}}\right)^{\!\!-\ga} \biggl(\frac{\theta t}{s}\biggr) t^{-2\ga}s^{-\frac d2 }.
\end{equation*}
An identical analysis as in Step 4, starting from \eqref{e.intfrac} allows us to conclude \eqref{e.vQcs} for $|x|\leq A_{0}\sqrt{t \log t}$. For the tail, we note that an identical analysis to that of Step 4 yields that by deterministic tail bounds, 
\begin{equation*}
|v(t,x) - Q(t,x;s)|+|c_{s}[v_{0}]\bar{P}(t,x)|\leq CK^{\frac{d}{2}}Ms^{-\frac{d}{2}}t^{-\frac{d}{2}-2\ga}\exp\left(-\frac{|x|^{2}}{A_{1}t}\right).
\end{equation*}
Using again that $t\geq sR^{-2}$, we obtain that for $|x|\geq A_{0}\sqrt{t\log t}$,   
\begin{equation*}
|v(t,x) - Q(t,x;s)|+|c_{s}[v_{0}]\bar{P}(t,x)|
\leq CK^{\frac{d}{2}+1}M\left(\frac{t}{R^{2}}\right)^{\!\!-\ga}\!\left(\frac{t}{s}\right)^{\!\!\ga}t^{-\ga}t^{-\frac{d}{2}}\exp\left(-\frac{|x|^{2}}{A_{1}t}\right),
\end{equation*}
and this implies \eqref{e.vQcs} holds everywhere. 

\smallskip

\emph{Step 6.} The verification of the induction.
Combining~\eqref{e.split}, \eqref{e.Q.bar.evolve},~\eqref{e.step4display} and~\eqref{e.vQcs} yields 
\begin{align*}
\lefteqn{
\bigl| v(t,x) - c_{s}[v_{0}] \overline{P}(t,x) \bigr|
} &
\\ &
\leq 
CK^{\frac{d}{2}+1}M
\biggl ( \!\frac{t}{R^{2}} \!\biggr )^{\!\!-\gamma} 
\Biggl[A_{1}^{\frac{d+1}{2}} \!
\biggl( \frac ts \biggr)^{ \! \!-\frac12+\ga} \!
\biggl( \frac t{t-s} \biggr)^{ \! \!\frac{d+1}2}  \! \!
+\biggl (\frac{t}{s}\biggr )^{ \! \!\ga+\frac{d}{2}} \!
\biggl (\frac{\theta t}{s}\biggr )t^{-\ga}\Biggr ]  
t^{-\frac d2}
\exp\left( \!-\frac{|x|^{2}}{A_{1} t}\right)
\,.
\end{align*}
Choosing now $t=Ks$ and simplifying, we obtain
\begin{align}\label{e.almostdone}
\lefteqn {
\bigl| v(Ks,x) - c_{s}[v_{0}] \overline{P}(Ks,x) \bigr|
}  & 
\notag \\ & 
\leq 
C K^{\frac{d}{2}+1}M\left(\frac{Ks}{R^{2}}\right)^{\!\!-\ga} \!\left(A_{1}^{\frac{d+1}{2}}K^{-\frac{1}{2}+\ga}+K^{\ga+\frac{d}{2}}K\theta(Ks)^{-\ga}\right) 
 (Ks)^{-\frac{d}{2}}\exp\left(\!-\frac{|x|^{2}}{A_{1}(Ks)}\right).
\end{align}
This implies, via integration, 
\begin{align}
\label{e.cbounds}
\bigl |c_{Ks}[v_{0}] - c_{s}[v_{0}] \bigr |
\leq 
C A_1^\frac{d}{2}K^{\frac{d}{2}+1}
M
\left(\frac{Ks}{R^{2}}\right)^{\!\!-\ga}
\left(A_{1}^{\frac{d+1}{2}}K^{-\frac{1}{2}+\ga}
+
K^{1+\frac{d}{2}}\theta s^{-\ga}
\right).
\end{align}
Combining \eqref{e.almostdone} and \eqref{e.cbounds}, we obtain 
\begin{align*}
\lefteqn{
\bigl| v(Ks,x) - c_{Ks}[v_{0}] \overline{P}(Ks,x) \bigr|
} \quad & 
\\ &
\leq \bigl| v(Ks,x) - c_{s}[v_{0}] \overline{P}(Ks,x) \bigr| +\bigl| c_{s}[v_{0}] -c_{Ks}[v_{0}] \bigr|\overline{P}(Ks,x) 
\\ &\leq CA_1^\frac{d}{2}K^{\frac{d}{2}+1}M\left(\frac{Ks}{R^{2}}\right)^{\!\!-\ga}\left(A_{1}^{\frac{d+1}{2}}K^{-\frac{1}{2}+\ga}
+
K^{1+\frac{d}{2}}\theta s^{-\ga}\right)
(Ks)^{-\frac{d}{2}}\exp\left(-\frac{|x|^{2}}{A_{1}(Ks)}\right).
\end{align*}
Since $\ga \leq \tfrac{1}{4}<\tfrac{1}{2}$, choosing $K$ sufficiently large in terms of $C, A_{1}$, and then choosing~$s\geq \Sigma$ deterministic, we obtain 
\begin{align*}
\bigl| v(Ks,x) - c_{Ks}[v_{0}] \overline{P}(Ks,x) \bigr|  
\leq 2\biggl( 
\frac{\La(K+1)}{\la} \biggr)^{\!\frac d2+1}M\left(\frac{Ks}{R^{2}}\right)^{\!\!-\ga}(Ks)^{-\frac{d}{2}}\exp\left(-\frac{|x|^{2}}{A_{1}Ks}\right), 
\end{align*}
and this is \eqref{e.indc2} evaluated at time $Ks$. We also clearly have by \eqref{e.cbounds} that, for $\Sigma$ and $K$ sufficiently large, 
\begin{equation}\label{e.cbounds'}
\left| c_{Ks}[v_{0}]-c_{s}[v_{0}]\right|\leq (4\pi \La(K+1))^{\frac{d}{2}+1}M\left(\frac{Ks}{R^{2}}\right)^{\!\!-\gamma}. 
\end{equation}
Since $s\in [R^{2}, KR^{2}]$, the bound \eqref{e.baseint} and the prior estimates imply 
\begin{equation*}
|c_{Ks}[v_{0}]|\leq (4\pi \La (K+1))^{\frac{d}{2}}M\biggl (1+\left(\frac{Ks}{R^{2}}\right)^{\!\!- \gamma}\,\biggr )\leq (4\pi \La (K+1))^{\frac{d}{2}}M\prod_{j=1}^{\infty}\left(1+K^{-j\ga}\right)\,,
\end{equation*}
and this verifies \eqref{e.indc1}. 

\smallskip
\emph{Step 7.} The conclusion. 
For every $t>0$, let $j_{0}=\lfloor \log_{K}(t/s)\rfloor$ (so $t\geq K^{j_{0}s}$) for some $s\in [R^{2}, KR^{2}]$. This implies that $j_{0}\geq \log_{K}(t/KR^{2})$, and hence \eqref{e.cbounds'} yields 
\begin{align}\label{e.cauchy}
|c_{\infty}[v_{0}]-c_{t}[v_{0}]|
&\leq 
\sum_{j=j_{0}}^{\infty} |c_{K^{j+1}s}[v_{0}]-c_{K^{j}s}[v_{0}]|\\ \notag 
&\leq (4\pi\La)^{\frac{d}{2}}(K+1)^{\frac{d}{2}+1}M\int_{\log_{K}(t/KR^{2})}^{\infty}K^{-\ga x}\, dx
\\ \notag 
&
\leq (4\pi\La)^{\frac{d}{2}}(K+1)^{\frac{d}{2}+1}M\frac{1}{\ga |\log \ga|}\left(\frac{t}{KR^{2}}\right)^{\!\!-\ga}
\leq CM\left(\frac{t}{KR^{2}}\right)^{\!\!-\ga}.
\end{align}
We next define $c[v_{0}]:=c_{\infty}[v_{0}]=\lim_{t\to\infty} c_{t}[v_{0}]<\infty$. This implies 
\begin{align*}
\lefteqn{
|v(t,x)-c[v_{0}]\bar{P}(t,x)|
} \quad 
 \notag \\  & 
\leq |v(t,x)-c_{t}[v_{0}]\bar{P}(t,x)|+|c_{t}[v_{0}]-c[v_{0}]|\bar{P}(t,x)\\
&\leq \frac{\La}{\la}(K+1)^{\frac d2+1}
M
\biggl( \frac{t}{R^2} \biggr)^{\!\!- \gamma}
t^{-\frac{d}{2}}\exp \biggl(\! -\frac{|x|^2}{A_1t} \biggr)+CM\left(\frac{t}{KR^{2}}\right)^{\!\!-\ga}t^{-\frac{d}{2}}\exp \biggl( \!-\frac{|x|^2}{A_1t} \biggr)\\
&\leq CM\biggl( \frac{t}{R^2} \biggr)^{\!\!- \gamma}t^{-\frac{d}{2}}\exp \biggl(\! -\frac{|x|^2}{A_1t} \biggr).
\end{align*}
This is~\eqref{e.CLT}. 

\smallskip

\emph{Step 8.} In order to prove \eqref{e.int}, we note that for every $t\geq 0$, 
\begin{align}\label{e.sand2}
\left|\int_{\R^{d}} v_{0}(x)\, dx-c[v_{0}]\right|&\leq \left|\int_{\R^{d}} \left(v_{0}(x)-v(t,x)\right)\, dx\right|+|c_{t}[v_{0}]-c[v_{0}]|. 
\end{align}
For $\theta\in (0,1)$ to be chosen, letting $t=R^{2+\theta}$ in \eqref{e.cauchy}, we have 
\begin{equation}\label{e.finc}
|c[v_{0}]-c_{R^{2+\theta}}[v_{0}]|\leq CMR^{-\theta\ga}.
\end{equation}
We next claim that 
\begin{equation}\label{e.1fint}
\max_{t\in [0, R^{2+\theta}]} \left|\int_{\R^{d}} \left(v_{0}(x)-v(t,x)\right)\, dx\right|\leq CM\left(1+M^{-1}R^{d+\sigma}[v_{0}]_{C^{0, \sigma}(B_{R})}\right)R^{-\ga}.
\end{equation}
Let $\bar{v}$ solve 
\begin{equation*}
\begin{cases}
\partial_{t}\bar{v}-\tr(\ahom D^{2}\bar{v})=0&\text{in $(0, \infty)\times \R^{d}$}, \\
\bar{v}(0, \cdot)=v_{0}(\cdot)&\text{on $\R^{d}$}.
\end{cases}
\end{equation*}
Without loss of generality, we redefine $\sigma$ to be the minimum of the exponent from the H\"older exponent of the initial data and $\sigma(\la, \La, d)$ from the Krylov Safonov estimate. Proposition \ref{prop_homogenize} yields that, for $R\geq \mathcal{Y}$, 
\begin{align*}
\lefteqn{
\norm{v-\bar{v}}_{L^{\infty}((0,R^{2+\theta}]\times \R^d)}
} \qquad &  
\\ &
\leq C\left(\norm{v_{0}}_{L^{\infty}(B_{R})}+R^{\sigma}[v_{0}]_{C^{0, \sigma}(B_{R})}\right)R^{-\frac{\beta\sigma}{2}}+2\sup_{|x|\geq 2^{-1}R}|v_{0}(x)|\\
&\leq C\left(\norm{v_{0}}_{L^{\infty}(B_{R})}+R^{\sigma}[v_{0}]_{C^{0, \sigma}(B_{R})}\right)R^{-\frac{\beta\sigma}{2}},
\end{align*}
where the term $2\sup_{|x|\geq 2^{-1}R}|v_{0}(x)|$ can be absorbed into $\norm{v_{0}}_{L^{\infty}(B_{R})}R^{-\frac{\beta\sigma}{2}}$ using the assumption of Gaussian decay and $d>\frac{\beta\sigma}{2}$.  In particular,
\begin{align*}
\norm{v-\bar{v}}_{L^{\infty}((0,R^{2+\theta}]\times B_{R^{1+\theta}})}&\leq CM\left(1+M^{-1}R^{d+\sigma}[v_{0}]_{C^{0, \sigma}(B_{R})}\right)R^{-d}R^{-\frac{\beta\sigma}{2}},
\end{align*}
using $\norm{v_0}_{L^\infty(\R^d)}\leq CMR^{-d}$.  Furthermore, by \eqref{e.crude.v}, with $R\geq \Sigma$ deterministic, 
\begin{align*}
 \max_{t\in[0,R^{2+\theta}]}\int_{|x|\geq R^{1+\theta}} \left(|v(t,x)|+|\bar{v}(t,x)|\right)\, dx
&\leq CM\int_{|x|\geq R^{1+\theta}} R^{-d} \exp\left(\!-\frac{|x|^{2}}{4\La R^{2+\theta}}\right)\\
&\leq CMR^{-\frac{\beta\sigma}{2}}.
\end{align*}
Combining the above two displays yields that 
\begin{align*}
\max_{t\in[0,R^{2+\theta}]}\left| \int_{\R^{d}}\left(v(t,x)-\bar{v}(t,x)\right)\, dx\right|& \leq \max_{t\in[0,R^{2+\theta}]}\int_{|x|\geq R^{1+\theta}} \left(|v(t,x)|+|\bar{v}(t,x)|\right)\, dx\\
&\qquad+\norm{v-\bar{v}}_{L^{\infty}((0,R^{2+\theta}]\times B_{R^{1+\theta}})}CR^{d(1+\theta)}\\
&\leq CM\left(1+M^{-1}R^{d+\sigma}[v_{0}]_{C^{0, \sigma}(B_{R})}\right)R^{d\theta}R^{-\frac{\beta\sigma}{2}}.
\end{align*}
Choosing $\theta:= \frac{\beta\sigma}{8d}$ yields 
\begin{equation*}
\max_{t\in[0,R^{2+\theta}]}\left| \int_{\R^{d}}\left(v(t,x)-\bar{v}(t,x)\right)\, dx\right|\leq CM\left(1+M^{-1}R^{d+\sigma}[v_{0}]_{C^{0, \sigma}(B_{R})}\right)R^{-\ga}.
\end{equation*} 
Since $\bar{v}$ solves a constant coefficient PDE and preserves mass, this verifies \eqref{e.1fint}, after redefining $\ga$. 
Combining \eqref{e.1fint}, \eqref{e.finc} into \eqref{e.sand2} yields \eqref{e.int}. 
\end{proof}

We next show that Theorem \ref{t.GF} follows as an easy consequence of Theorem \ref{t.realthm}. 

\begin{proof}[Proof of Theorem \ref{t.GF}]
We fix $p\in (0,d)$ and let $\Y=\Y_{p}$ from Theorem \ref{t.realthm}. 
Fix~$y \in \cu_0$. 

\smallskip

We begin by verifying the hypotheses of Theorem \ref{t.realthm}. By Proposition \ref{le_prop_2}, we recall that there exists $c,C,\overline{\ka}\leq \underline{\ka}, b$ and $B$ such that, for every~$t \geq 1$,
\begin{equation}\label{e.startbump}
ct^{-\underline{\ka}}\exp\left(-\frac{\abs{x}^2}{b t}\right)
\leq P(t,x,y)
\leq 
Ct^{-\overline{\ka}}\exp\left(-\frac{\abs{x}^2}{B t}\right).\end{equation}
Let $R:=\mathcal{Y}\vee \Sigma$, where $\Sigma=\Sigma(\la, \La, d, \al_{0}, K_{0})$ is deterministic. Define 
\begin{equation}\label{e.Mdef}
M:=R^{d}\norm{P(R^{2}, \cdot,y)}_{L^{\infty}(\R^{d})}. 
\end{equation}
Observe that by \eqref{e.startbump}, we must have 
\begin{equation}\label{e.Msandwich}
cR^{d-2\underline{\ka}}\leq M \leq CR^{d-2\overline{\ka}}.
\end{equation}
We claim that there exists $C=C(\la,\La, d, \al_{0}, K_{0})\in [1, \infty)$ so that 
\begin{align}\label{e.Mbd}
R^{d}P(3R^{2},x,y)&\leq CM\exp\left(-\frac{\abs{x}^2}{B R^{2} \log R}\right).
\end{align}
First, observe that the maximum principle yields that 
\begin{equation*}
R^{d}\norm{P(3R^{2}, \cdot,y)}_{L^{\infty}(\R^{d})}\leq M.
\end{equation*}
If $|x|\leq R\sqrt{6B\log(R)(\underline{\ka}-\overline{\ka})}$, if $C\geq e^{6(\underline{\ka}-\overline{\ka})},$ then \eqref{e.Mbd} holds. If, on the other hand,  $|x|\geq R\sqrt{6B\log(R)(\underline{\ka}-\overline{\ka})}$, then by \eqref{e.startbump} and \eqref{e.Msandwich},
\begin{align*}
R^{d}P(3R^{2},x,y)
\leq CR^{d-2\overline{\ka}}\exp\left(-\frac{|x|^{2}}{3BR^{2}}\right)
&
= CR^{d-2\underline{\ka}}R^{2(\underline{\ka}-\overline{\ka})}\exp\left(-\frac{|x|^{2}}{3BR^{2}}\right)\\
&\leq CM\exp\left(-\frac{|x|^{2}}{BR^{2}\log R}\right)\,.
\end{align*}
This proves \eqref{e.Mbd}. Theorem \ref{t.realthm}, starting at time $3R^{2}$, now yields the existence of $m(y)\in (0,\infty)$, with 
\begin{equation}\label{e.mdef}
m(y):=\lim_{t\to \infty}\int P(t,x,y)\, dx,
\end{equation}
such that, for every time~$t\geq 2R^{2}\log R\geq 3R^{2}+R^{2}\log R$,
\begin{equation}
|P(t,x,y)-m(y)\overline{P}(t,x)|\leq CM(\log R)^{\frac{d}{2}}\left(\frac{t}{R^{2} \log R}\right)^{\!\!-\ga}t^{-\frac{d}{2}}\exp\left(-\frac{|x|^{2}}{Ct}\right). 
\end{equation}
For $s\in (0, \infty)$, define $c_{s}:=\int_{\R^{d}} P(s,x,y)\, dx$.  We next claim that
\begin{equation}\label{e.Mlower}
M\leq Cc_{3R^{2}}.
\end{equation}
Observe that by \eqref{e.startbump}, since 
\begin{equation*}
\lim_{|x|\to \infty} R^{d}P(R^{2},x,y)\leq \lim_{|x|\to \infty} CR^{d}R^{-2\overline{\ka}}\exp\left(-\frac{|x|^{2}}{BR^{2}}\right)=0<M, 
\end{equation*}
there exists $\bar{x}$ so that 
\begin{equation*}
M=R^{d}P(R^{2}, \bar{x}, y).
\end{equation*}
We now apply the parabolic Harnack inequality in $Q_{R}(R^{2}, \bar{x})$, to conclude that 
\begin{align*}
M=R^{d}P\left(R^{2}, \bar{x}, y\right)&=\sup_{x\in B_{R}(\bar{x})}R^{d}P(R^{2}, x, y)\leq C\inf_{x\in B_{R}(\bar{x})} R^{d}P\left(3R^{2},x,y\right)
\,.
\end{align*}
Hence 
\begin{equation*}
c_{3R^{2}}\geq \int_{|x-\bar{x}|\leq R} P\left(3 R^{2}, x,y\right)\, dx\geq C\left(\inf_{x\in B_{R}}P\left(3 R^{2},x,y\right)\right)R^{d}\geq CM.
\end{equation*}
This implies \eqref{e.Mlower}. 

\smallskip

Observe that at time~$t=3R^{2}$, we can apply the Krylov-Safonov estimates (Lemma~\ref{l.ksref}) with $r_{1}=R \sqrt{\log R}$ and $t_{0}=3R^{2}$, $r=(\sqrt{2})^{-1}R$, and the definition of $M$ to obtain
\begin{align*}
\lefteqn{
(R \sqrt{\log R})^{\sigma}[P(3R^{2}, \cdot, y)]_{C^{0,\sigma}(B_{R \sqrt{\log R}})}
} \quad 
\\ & 
\leq C(\log R)^{\frac{\sigma}{2}}\norm{P(\cdot, \cdot, y)}_{L^{\infty}([\frac{3}{2}R^{2}, 3R^{2}]\times B_{R\sqrt{\log R}})}\\
&\leq C(\log R)^{\frac{\sigma}{2}}\norm{P(R^{2}, \cdot, y)}_{L^{\infty}(\R^{d})}
=
C(\log R)^{\frac{\sigma}{2}}MR^{-d}. 
\end{align*}
The prior display, \eqref{e.int}, and \eqref{e.Mlower} yield
\begin{equation}\label{e.1int}
|m(y)-c_{3R^{2}}|\leq  CM(\log R)^{\frac{d}{2}}(1+(\log R)^{\frac{\sigma}{2}})R^{-\ga} \leq Cc_{3R^{2}}(\log R)^{\frac{d+\sigma}{2}} R^{-\ga}
\,.
\end{equation}
Hence, by \eqref{e.1int} and \eqref{e.Mlower}, we get
\begin{equation}\label{e.mlower}
m(y)\geq c_{3R^{2}}(1-C(\log R)^{\frac{d+\sigma}{2}}R^{-\ga})\geq  \frac{1}{2}c_{3R^{2}}\geq cM. 
\end{equation}
This yields that for $t\geq 2R^{2}\log R$, 
\begin{equation*}
|P(t,x,y)-m(y)\overline{P}(t,x)|\leq Cm(y)(\log R)^{\frac{d}{2}}\left(\frac{t}{R^{2}\log R}\right)^{\!\!-\ga}t^{-\frac{d}{2}}\exp\left(-\frac{|x|^{2}}{Ct}\right). 
\end{equation*}
Defining $\tilde{\Y}=\mathcal{Y}(\log \Y)^{\frac{d}{4\ga}+\frac{1}{2}}$, we observe that $\tilde{\Y}=\O_{q}(C)$ for any $q<p<d$, and we may simplify this statement to the statement that
\begin{equation*}
|P(t,x,y)-m(y)\overline{P}(t,x)|\leq Cm(y)\left(\frac{t}{\tilde{\Y}^{2}}\right)^{\!\!-\ga}t^{-\frac{d}{2}}\exp\left(-\frac{|x|^{2}}{Ct}\right), 
\end{equation*}
and since $y\in \cu_{0}$, this proves \eqref{e.GF}. \end{proof}

\section{Estimates of the stationary invariant measure}\label{s.invmeas}

\subsection{Existence and uniqueness of the stationary invariant measure}
\label{ss.m.exis.uniq}

We show that the results of the previous section imply the existence and uniqueness of a stationary invariant measure, up to multiplication by a constant. 

\smallskip

For existence, we note that the mapping  
\begin{equation*}
v_0 \mapsto c[v_0] 
\end{equation*}
given by Theorem~\ref{t.realthm} is a linear functional on~$C_c^{0,\sigma}(\Rd)$. In view of Theorem~\ref{t.GF} and linearity, using the Green function formula, with~$m$ as in the statement of the theorem, we have for any $v_{0}$ satisfying \eqref{e.intgauss},
\begin{equation}\label{e.cformula}
c[v_0] 
=\lim_{t\to\infty}\int\int P(t,x,y)v_{0}(y)\, dy\, dx=
\int_{\Rd} m(y) v_0(y)\,dy
\,.
\end{equation}
By~\eqref{e.int}, $\int_{\R^{d}} m(y)v_{0}(y)\, dy$ is finite for every $v_{0}$ with $R^{\sigma}[v_{0}]_{C^{0, \sigma}(B_{R})}<\infty$ satisfying \eqref{e.intgauss}. The~$\Zd$--stationarity of~$m$ is immediate from that of the parabolic Green function.

\smallskip

We next check that~$m$ is the density of a measure satisfying~\eqref{e.invmeas.weak}.
Select~$\varphi \in C^\infty_c(\Rd)$ and let~$v$ be the solution of the initial value problem 
\begin{equation*}
\left\{
\begin{aligned}
& \partial_t v - \tr \bigl( \a D^2 v \bigr) = 0 
& \mbox{in} 
& \ (0,\infty) \times \Rd\,,
\\ &
v(0,\cdot) = \varphi 
& \mbox{on} 
& \ \Rd\,.
\end{aligned}
\right.
\end{equation*}
It is clear from the statement of Theorem~\ref{t.realthm} that the functional~$c[\cdot]$ is invariant under the parabolic flow; that is, 
\begin{equation*}
c[\varphi] = c [v(t,\cdot) ], \quad \forall t \in (0,\infty). 
\end{equation*}
Therefore, by \eqref{e.cformula}, and since $\vp\in C^{\infty}_{c}(\R^{d})$, for every~$t\in (0,\infty)$, 
\begin{equation*}
0 
=
\partial_t \int_\Rd m(y) v(t,y)\,dy 
=
\int_\Rd m(y) \partial_t v(t,y)\,dy
=
\int_\Rd m(y) \tr\bigl( \a D^2v(t,y) \bigr)\,dy\,.
\end{equation*}
Sending~$t\downarrow 0$ yields
\begin{equation*}
\int_{\R^{d}} m(y) \tr \bigl( \a D^2\varphi(y) \bigr) \, dy
= 0,
\end{equation*}
which completes the proof of~\eqref{e.invmeas.weak} for~$\mu$ with~$d\mu = m(x)\,dx$. 
It follows from deterministic regularity estimates---see Proposition~\ref{p.Holder.continuity} in the appendix---that~$m\in L^\infty_{\mathrm{loc}}(\Rd)$. In fact, we have by our assumption~\eqref{e.holder} that~$m\in C^{0,\alpha_0}_{\mathrm{loc}}(\Rd)$. 

\smallskip

Turning to the question of uniqueness, suppose that~$\mu$ is another invariant measure which is~$\Zd$--stationary and satisfies~$\E [ | \mu|(\cu_{0}) ] <\infty$. 
Then we have that the quantity 
\begin{equation*}
\int_\Rd P(t,x,y)  \,d\mu(x)
\end{equation*}
is~$\P$--a.s.~defined for all~$(t,y)\in(0,\infty)\times\Rd$.
Using the invariance of the measure, we have 
\begin{equation*}
\partial_t \int_\Rd P(t,x,y)  \,d\mu(x) =0.
\end{equation*}
By a corollary of the ergodic theorem, since~$\mu$ is stationary, then
\begin{equation*}
\lim_{t\to \infty} 
\int_\Rd P(t,x,y)  \,d\mu(x)
=
\overline{\mu} 
\lim_{t\to \infty} 
\int_\Rd P(t,x,y)\,dx
=
\overline{\mu} 
m(y)
\,.
\end{equation*}
The proof of the first equality follows essentially the same argument as the argument found at the end of Section~\ref{ss.mto1} (the proof of \eqref{e.weaklim.quant}), using the H\"older continuity of $P(t, \cdot, y)$, and we refer the reader to that argument for more details.
Thus
\begin{equation*}
d\mu(y) = \lim_{t\to 0} \int_\Rd P(t,x,y)\,d\mu(x) =
\lim_{t\to \infty} 
\int_\Rd P(t,x,y)  \,d\mu(x)
=\overline{\mu} 
m(y)\,.
\end{equation*}
This concludes the proof of uniqueness, and note that we will prove an even stronger uniqueness statement in Section \ref{s.C01}. 

\smallskip

We next prove the pointwise estimates on~$m$.
Consider a cutoff function~$\varphi\in C^\infty_c(\Rd)$ satisfying
\begin{equation*}
\indc_{B_R} 
\leq 
\varphi 
\leq 
\indc_{B_{2R}}
\quad \mbox{and}  \quad
\bigl\| D \varphi \bigr\|_{L^\infty} 
\leq 2R^{-1}\,.
\end{equation*}
Owing to Theorem~\ref{t.realthm}, we have that, for every~$R \geq \Y$, 
\begin{equation*}
\int_{B_{R}} m(y)\, dy \leq 
\int_{B_{2R}} \varphi(y) m(y)\, dy = c[\varphi] \leq CR^{d}+CR^{d-\ga}\,.
\end{equation*}
That is, for every $R\geq \Y$, 
\begin{equation}\label{e.mL1}
\| m \|_{\underline{L}^1(B_R)} \leq C. 
\end{equation}
We may next apply Proposition~\ref{p.Guido}, which we note is a scale-invariant estimate, to obtain the existence of an exponent~$\delta(d,\lambda,\Lambda)> 0$ such that, for every~$R \geq \mathcal{Y}$, 
\begin{equation}
\label{e.mqint}
\| m \|_{\underline{L}^{\frac{d}{d-1}+\delta} (B_{R/2}) }
\leq 
C\,.
\end{equation}
Note that the constant~$C$ in~\eqref{e.mqint} depends only on~$(d,\lambda,\Lambda)$ and not on~$(\alpha_0,K_0)$. 
By H\"older's inequality, this implies that, for some exponent~$\delta(d,\lambda,\Lambda)> 0$, which is not the same~$\delta$ as above, but can be computed in terms of it, 
\begin{equation*}
\sup_{x\in B_{R/4}} 
\| m \|_{L^1(x+\cu_0)} 
\leq 
CR^{d-1-\delta} \,.
\end{equation*}
Appealing now to Proposition~\ref{p.Holder.continuity} yields, for every~$R \geq \mathcal{Y}$,  
\begin{equation*}
\| m \|_{L^\infty(B_{R/4})}
\leq 
C R^{d-1-\delta}, 
\end{equation*}
with the constant~$C$ now depending additionally on~$(\alpha_0,K_0)$ in~\eqref{e.holder}, as well as on~$(d,\lambda,\Lambda)$. 
In particular, for every~$s \in (0,d)$, we have the estimate
\begin{equation*}
\P \bigl[ \| m \|_{L^\infty(B_{R/4})} > C R^{d-1-\delta} \bigr]
\leq 
\exp \bigl( -cR^s \bigr) \,.
\end{equation*}

\subsection{Proof of Theorem~\ref{t.mto1}}\label{ss.mto1}
We will now see that the assertions of Theorem~\ref{t.mto1} for~$m$ and $\A m$ are a simple consequence of Theorem~\ref{t.realthm} and Theorem~\ref{t.GF}. 

\begin{proof}[Proof of the estimate for the first term in~\eqref{e.m.conv}] 
Select~$\varphi\in C^\infty_c(\Rd)$ satisfying
\begin{equation*}
R^{-d}\indc_{R\cu_{0}} 
\leq 
\varphi 
\leq R^{-d}\indc_{R(1+R^{-\alpha})\cu_{0}}
\quad \mbox{and}  \quad
\bigl\| D \varphi \bigr\|_{L^\infty} 
\leq 2R^{-d-1+\alpha}\,,
\end{equation*}
for $\al\in (0,1]$ to be defined. 
Applying Theorem~\ref{t.realthm}, for $R \geq \Y$,  we deduce that 
\begin{equation*}
\bigl| c[\varphi] - 1\bigr|
\leq
\Bigl| c[\varphi] - \int \vp(x)\, dx \Bigr|+CR^{-\al}
\leq C(1+R^{\al})R^{-\ga}+CR^{-\al}\leq 
CR^{\alpha-\gamma}
+
CR^{-\alpha}
\,.
\end{equation*}
Also, using \eqref{e.mqint}, for $R\geq \Y_{p}$, for any $q\in [1, \frac{d}{d-1}]$,
\begin{equation*}
\int_{R(1+R^{-\alpha})\cu_{0} \setminus R\cu_{0}}
\varphi(x) m(x)\,dx 
\leq 
R^{-\frac{d}{p}}\left|R(1+R^{-\alpha})\cu_{0} \setminus R\cu_{0}\right|^{\frac{1}{p}}\norm{m}_{\underline{L}^{q}(2R\cu_{0})}
\leq CR^{-\frac{\al}{p}}.
\end{equation*}
From these, it follows from \eqref{e.cformula} that
\begin{equation*}
\biggl| 
\fint_{R\cu_{0}} m(x) \,dx 
-1
\biggr|
\leq |c[\varphi]-1|+\int_{R(1+R^{-\alpha})\cu_{0} \setminus R\cu_{0}}
\varphi(x) m(x)\,dx
\leq 
CR^{\alpha -\gamma}
+
CR^{-\frac{\al}{p}}
\,.
\end{equation*}
Taking~$\alpha\leq \frac{\ga}{2}$ sufficiently small, depending only on~$(d,\lambda,\Lambda)$, both of the exponents of~$R$ on the right side are negative. After relabeling $\ga$, we have shown that, for $R\geq \Y$, 
\begin{equation}\label{e.mavg}
\biggl| 
\fint_{R\cu_{0}} m(x) \,dx 
-1
\biggr|\leq CR^{-\ga},
\end{equation}
and this is the first half of \eqref{e.m.conv}.  
\end{proof}

We next give a quantitative rate of convergence for spatial averages of $m\A$ to $\ahom$. Observe that, as a consequence of the ergodic theorem, this proof in fact demonstrates the relation,
\begin{equation*}
\ahom_{ij}=\E\left[\fint_{\cu_{0}}\A_{ij} m\, dx\right].
\end{equation*}
We will use the stationary approximate correctors defined in~\eqref{e.dvd.def}. 
\begin{proof}[Proof of the estimate for the second term in~\eqref{e.m.conv}] 
Let $p\in(0,d)$, let $p<p'<2p \wedge d$, and let $\Y_{p'}$ denote the random variable defined simultaneously in Theorem~\ref{t.realthm} and in Proposition~\ref{p.delta.vee.delta}.  In the latter case, this says that, 
for every~$\delta \in (0,\Y_{p'}^{-1} ]$,  
\begin{equation}
\label{e.delta.vee.delta.app}
\sup_{x\in B_{\delta^{-100}}}|\ahom_{ij}-\delta^{2}w^{\delta}_{ij}(x)|\leq C\delta^{\beta}
\,.
\end{equation}
For every $\delta,R\in(0,\infty)$,  for $v^R_{i,j}$ the solution of
\begin{equation*}
\begin{cases}
\partial_t v^R_{ij}-\tr (\A D^{2}v^R_{ij})=0 &\text{in $(0,\infty)\times\R^d$}, \\
v^R_{ij}(0, \cdot)=\mathbf{1}_{B_{R}}\a_{ij}&\text{on $\R^d$, }
\end{cases}
\end{equation*}
we have that the function $w^{\delta,R}_{ij}$ defined by
\begin{equation}\label{am_7}\delta^2 w^{\delta,R}_{ij}(x) = \delta^2\int_0^\infty \exp(-s\delta^2) v^R_{ij}(s,x)\, ds,\end{equation}
is the unique (vanishing at infinity) solution of the equation
\[\delta^2w^{\delta,R}_{ij} - \tr(\a D^2w^{\delta,R}_{ij}) = \mathbf{1}_{B_R}\a_{ij}\;\;\textrm{in}\;\;\R^d.\]
Observe that the function $R^{-d}v^{R}_{ij}$ also solves the same PDE with the initial condition $R^{-d}\mathbf{1}_{B_{R}}\a_{ij}$, which satisfies the hypotheses of Theorem \ref{t.realthm}. Combined with \eqref{e.cformula}, this implies there exists a constant $C=C(\la, \La, d)\in [1, \infty)$ such that for $R\geq \Y_{p'}$, 
\begin{align}\label{e.cform}
&\Big|R^{-d}\int_{\R^{d}} \delta^2 w^{\delta,R}_{ij}\, dx- \fint_{B_R} \a_{ij} m\, dx\Big|\\
&=\left|\int_{\R^{d}}\delta^2\int_0^\infty \exp(-s\delta^2)\left(R^{-d}v^R_{ij}(s,x)-c[R^{-d}\mathbf{1}_{B_{R}}\a_{ij}]\overline{P}(s,x)\right)\, ds\, dx\right|\notag\\
&\leq C\delta^{2}R^{2+\frac{p'}{p}}+\delta^{2}\int_{R^{2+\frac{p'}{p}}}^{\infty} C\exp(-s\delta^2)\left(\frac{s}{R^{2}}\right)^{-\ga}\, ds\notag
\\ & \leq C\delta^{2}R^{2+\frac{p'}{p}}+C\exp(-\delta^{2}R^{2+\frac{p'}{p}})\leq C\delta^{2}R^{2+\frac{p'}{p}}, \notag
\end{align}
in the regime that $\delta^2R^{2+\frac{p'}{p}}\leq 1$.  We now estimate
\begin{equation}\label{am_1}
R^{-d}\int_{\R^{d}} \delta^2 w^{\delta,R}_{ij}\, dx= R^{-d}\int_{B_{R-R^{\theta}}}  \delta^2 w^{\delta,R}_{ij} \, dx + R^{-d}\int_{\R^d\setminus B_{R-R^{\theta}}}\delta^2 w^{\delta,R}_{ij} \, dx.
\end{equation}
Recall that by Proposition~\ref{le_prop_2} and \eqref{am_7}, for some $C\in(0,\infty)$, for every $\theta'\in(0,1)$,
\[ \sup_{x\in B_{R-R^{\theta}}}\abs{w^{\delta,R}_{ij}(x)-w^\delta_{ij}(x)}  \leq C\biggl(\exp(-\delta^{-2\theta'})+\exp\biggl(-\frac{\delta^{2+2\theta'}R^{2\theta}}{2\Lambda}\biggr) \biggr).\]
Fix $\delta^{-1} = R^{\frac{\theta}{1+\theta'}-\frac{p'}{2p}}$ with $0<\theta'<\theta<1$ so that $\frac{\theta}{1+\theta'}>\frac{p'}{2p}$. 
To treat the first term on the righthand side of \eqref{am_1}, by \eqref{e.delta.vee.delta.app}, and the choice of constants, there exists $C\in(0,\infty)$,  $\gamma=\ga(\la, \La, d, p)\in(0,1)$, such that, for all $R\geq \Y_{p'}$,
\begin{equation}\label{am_4}\abs{R^{-d}\int_{B_{R-R^{\theta}}}  \delta^2 w^{\delta,R}_{ij}\, dx- \ahom_{ij}} \leq CR^{-\gamma}.\end{equation}
Moreover, it follows from Proposition~\ref{le_prop_2} that, for every $\abs{x}\geq R+R^{\theta}$ and $\ve\in(0,1)$,
\[\abs{\delta^2w_{ij}^\delta(x)}\leq C\biggl(\exp(-\delta^2(\abs{x}-R)^{2-\ve})+\exp\biggl(-\frac{(\abs{x}-R)^{\ve}}{2\Lambda}\biggr)\biggr).\]
We observe that $\frac{\theta}{1+\theta'}-\frac{p'}{2p}>0$ guarantees that for all $\ve\in(0,1)$ sufficiently small
\[\big(1-\frac{\ve}{2}\big)\theta>\frac{\theta}{1+\theta'}-\frac{p'}{2p},\]
and therefore that there exists $\gamma'\in(0,1)$ and $C,C'\in(0,\infty)$ such that, for every $\abs{x}\geq R+R^\theta$,
\[\abs{\delta^2w_{ij}^\delta(x)}\leq C\exp(-C'(\abs{x}-R)^{\gamma'}).\]
We then have from the choice of constants that, for every $R\geq \Y_{p'}$, the second term on the righthand side of \eqref{am_1} satisfies,  for some $C\in(0,\infty)$,
\begin{equation}\label{am_3}\abs{R^{-d}\int_{\R^d\setminus B_{R-R^{\theta}}}\delta^2 w^{\delta,R}_{ij}\, dx}\leq C R^{\theta-1}.\end{equation}
After returning to \eqref{am_1}, the claim follows from the choice of constants and \eqref{e.cform}, \eqref{am_4}, and \eqref{am_3}, that for $R \geq \Y_{p'}$,
\begin{equation*}
\left|\fint_{B_{R}} \a_{ij}m\, dx-\ahom_{ij}\right|\leq CR^{-\ga}.
\end{equation*}
This completes the proof. 
\end{proof}

To prove the bounds for~$m^\ep$ in~\eqref{e.weaklim.quant}, we need the following general multiscale lemma. We put a slash in our sums to denote averaging; that is, if~$S$ is any finite set, then 
\begin{equation*}
\avsum_{x\in S} f(x) : = \frac{1}{|S|} \sum_{x\in S} f(x) \,.
\end{equation*}

\begin{lemma}\label{l.multiscale}
For every~$\alpha \in (0,1]$, there exists a constant~$C(\alpha,d)<\infty$ such that, for every~$f\in L^{1}(\cu_{0})$, 
\begin{equation*}
\| f-(f)_{\cu_{0}}\|_{W^{-\al,1}(\cu_0)} 
\leq 
C \sum_{n=-\infty}^0
3^{n\al}
\avsum_{z\in 3^n\Zd\cap \cu_0} 
\Bigl| 
\fint_{z+\cu_n} \left(f(x)-(f)_{\cu_{0}}\right)\, dx
\Bigr|
\,,
\end{equation*}
where we recall that $(f)_{\cu_{0}}=\fint_{\cu_{0}}f$.
\end{lemma}
\begin{proof}
Let~$f\in L^1(\cu_0)$ with~$(f)_{\cu_{0}}=0$. Consider a test function~$\phi \in C^{0,\alpha}(\cu_0)$ with $\norm{\phi}_{C^{0, \al}(\cu_{0})}\leq 1$. Let $\mathcal{I}_{n}:=3^{n}\Z^{d}\cap \cu_{0}$. Then for every integer $n\leq -1$, for every $z\in \mathcal{I}_{n+1}$, 
\begin{align*}
\int_{z+\cu_{n+1}}f(x)(\phi(x)-\phi(z))\, dx&=\sum_{y\in \mathcal{I}_{n}\cap (z+\cu_{n+1})}\int_{y+\cu_{n}} f(x)(\phi(x)-\phi(y))\, dx\\
&\quad +\sum_{y\in \mathcal{I}_{n}\cap (z+\cu_{n+1})} (\phi(y)-\phi(z))\int_{y+\cu_{n}}f(x)\, dx. 
\end{align*}
Since $|y-z|\leq \sqrt{d}3^{n+1}$, we can bound the second term using the regularity of $\phi$ to conclude that 
\begin{align*}
\lefteqn{\sum_{\substack{y\in \mathcal{I}_{n}\cap (z+\cu_{n+1})}} (\phi(y)-\phi(z))\int_{y+\cu_{n}}f(x)\, dx}\\
&\qquad\qquad\qquad\qquad\qquad
\leq (\sqrt{d}3^{n+1})^{\al}|\cu_{n}| \sum_{y\in \mathcal{I}_{n}\cap (z+\cu_{n+1})} \biggl| \fint_{y+\cu_{n}}f(x)\, dx \biggr| \\
&\qquad\qquad\qquad\qquad\qquad= C3^{n\al}|\cu_{n}|  \sum_{y\in \mathcal{I}_{n}\cap (z+\cu_{n+1})} \biggl|   \fint_{y+\cu_{n}}f(x)\, dx \biggr|.
\end{align*}
Combining the above and summing over $z\in \mathcal{I}_{n+1}$, we have
\begin{align*}
&
\sum_{z\in \mathcal{I}_{n+1}} \int_{z+\cu_{n+1}}f(x)(\phi(x)-\phi(z))\, dx
\\ & \qquad 
\leq 
\sum_{y\in \mathcal{I}_{n}}\int_{y+\cu_{n}} f(x)(\phi(x)-\phi(y))\, dx
+ C3^{n\al}|\cu_{n}| \sum_{y\in \mathcal{I}_{n}}  \biggl|  \fint_{y+\cu_{n}}f(x)\, dx \biggr|
\\ & \qquad 
=\sum_{y\in \mathcal{I}_{n}}\int_{y+\cu_{n}} f(x)(\phi(x)-\phi(y))\, dx
+C3^{n\al}\avsum_{y\in \mathcal{I}_{n}}  \biggl|  \fint_{y+\cu_{n}}f(x)\, dx \biggr|
\,.
\end{align*}
Observe that by the triangle inequality,
\begin{equation*}
\lim_{n\to-\infty}3^{n\al}\avsum_{y\in \mathcal{I}_{n}}  \biggl|  \fint_{y+\cu_{n}}f(x)\, dx \biggr| \leq \lim_{n\to-\infty}3^{n\al}\int_{\cu_{0}}\abs{f(x)}\, dx=0,
\end{equation*}
since $f\in L^1(\cu_{0})$.  We now iterate this in $n$, so for $z\in \cu_{0}$, 
\begin{align*}
\int_{\cu_{0}}f(x)&(\phi(x)-\phi(z))\, dx\leq \sum_{n=-\infty}^{0}C3^{n\al}\avsum_{y\in \mathcal{I}_{n}}  \Bigl|  \fint_{y+\cu_{n}}f(x)\, dx \Bigr|. 
\end{align*}
Finally, since $(f)_{\cu_0}= 0$, this is equivalent to the claim. 
\end{proof}

\begin{proof}[{Proof of the estimate~\eqref{e.weaklim.quant}}]
We prove the first half of \eqref{e.weaklim.quant}. For $\ga\in (0,1)$ to be chosen, we express
\begin{equation*}
\|m^{\ve}-1\|_{W^{-\frac{\ga}{2},1}(\cu_{0})}\leq \|m^{\ve}-(m^{\ve})_{\cu_{0}}\|_{W^{-\frac{\ga}{2},1}(\cu_{0})}+\left|(m^{\ve})_{\cu_{0}}-1\right|, 
\end{equation*}
and by an application of Lemma~\ref{l.multiscale}, and for any $N<0$,
\begin{align*}
\|m^{\ve}-1\|_{W^{-\frac{\ga}{2},1}(\cu_{0})}
&
\leq \sum_{n=-\infty}^{N}C3^{n\frac{\ga}{2}}\avsum_{y\in \mathcal{I}_{n}}  \Bigl|  \fint_{y+\cu_{n}}\left(m^{\ve}(x)-(m^{\ve})_{\cu_{0}}\right)\, dx \Bigr|
\\ & \quad 
+ \sum_{n=N}^{0}C3^{n\frac{\ga}{2}}\avsum_{y\in \mathcal{I}_{n}}  \Bigl|  
\fint_{y+\cu_{n}} \!\!\! 
\left(m^{\ve}(x)-(m^{\ve})_{\cu_{0}}\right)\, dx \Bigr|
+\left|(m^{\ve})_{\cu_{0}}-1\right|.
\end{align*}
Fix $q\in (0,d)$ and let $p=q+\frac{d-q}{2}>q$. Let $\bar{\Y}:=\tilde{\Y}_{2,p}^{\frac{p}{q}}$ for $\tilde{\Y}_{2,p}$ as in Remark \ref{r.Ystep1}, with $\ve^{-1}\geq \bar{\Y}$. Observe that $\bar{\Y}=\O_{q}(C)$. Let $N<0$ so that $\ve\leq 3^{\frac{p}{p-q}N}$, and this entails that $3^{N}\ve^{-1}\geq \ve^{-\frac{q}{p}}\geq \tilde{\Y}_{2,p}$. Thus, by Remark \ref{r.Ystep1}, for each $n\in [N,0]$, 
\begin{equation*}
3^{n}\ve^{-1}\geq 3^{N}\ve^{-1}\geq \sup_{\substack{{n\in \Z}\\{n<0}}} \sup_{z\in 3^{2n}\Z^{d}\cap \cu_{0}} \tau_{z} \mathcal{Y}_{p} \geq \sup_{n\in [N,0]} \sup_{z\in \mathcal{I}_{n}} \tau_{z}\mathcal{Y}_{p},
\end{equation*}
 so there exists $\ga=\ga(\la, \La, d, p)$ such that by \eqref{e.mavg}, 
\begin{align*}
\avsum_{y\in \mathcal{I}_{n}}  \Bigl|  \fint_{y+\cu_{n}}\left(m^{\ve}(x)-(m^{\ve})_{\cu_{0}}\right)\, dx \Bigr|\leq C(3^{-n}\ve)^{\ga}\leq C3^{\frac{q}{p-q}N\ga}.
\end{align*}
Furthermore, by \eqref{e.mqint}, for $n\leq N$, 
\begin{equation*}
\avsum_{y\in \mathcal{I}_{n}}  \Bigl|  \fint_{y+\cu_{n}}\left(m^{\ve}(x)-(m^{\ve})_{\cu_{0}}\right)\, dx \Bigr|\leq \norm{m^{\ve}}_{L^{1}(\cu_{0})}+\left|(m^{\ve})_{\cu_{0}}\right|\leq C
\,.
\end{equation*}
Therefore, 
\begin{equation*}
\|m^{\ve}-1\|_{W^{-\frac{\ga}{2},1}(\cu_{0})}
\leq 
\sum_{n=-\infty}^{N}C3^{n\frac{\ga}{2}}+C|N|3^{\frac{q}{p-q}N\ga}
\leq C3^{\frac{q}{p-q}N\frac{\ga}{2}}
=
C\ve^{\frac{q}{p}\frac{\ga}{2}}
\,. 
\end{equation*}
The second half of \eqref{e.weaklim.quant} follows from an identical analysis, replaying $m^{\ve}$ by $\a_{\ve}m^{\ve}$.
\end{proof}

\subsection{Large-scale regularity for invariant measures}\label{s.C01}

In Section~\ref{ss.m.exis.uniq}, we proved the uniqueness of~$m$ with respect to the class of \emph{stationary} weak solutions of the equation~\eqref{e.invmeas}. This is a fairly weak uniqueness criterion. In particular, we have not yet proven that the invariant measure is unique in a \emph{quenched} sense; i.e., if~$m$ is unique for any \emph{particular} realization of the coefficients. We will now show that this is indeed the case, for~$\P$--almost every realization of the coefficients.

\smallskip

In fact, we will show more, namely a quantitative version of this uniqueness statement which says that any weak solution of~\eqref{e.invmeas} in a (large) domain must be close, on smaller subsets of this domain, in a quantitative sense, to a multiple of~$m$. This takes the form of a ``large-scale~$C^{0,1}$ estimate'' for solutions of~\eqref{e.invmeas}.

\smallskip

\begin{lemma}[Harmonic approximation]
\label{l.harmapprox}
Let~$\delta\in (0,\tfrac12]$ and $p\in (0,d)$. Let~$R\geq 1$ and~$\nu$ be a measure with support in~$B_{2R}$ which satisfies~\eqref{e.invmeas.weak} with~$U=B_{2R}$. 
Let~$\{ \eta_r \}_{r>0}$ denote the standard mollifier. Define, for every $r \in [0,R]$, 
\begin{equation}
\label{e.Nr}
N_r(x)
:= 
(\eta_r \ast \nu)(x)
:= 
\int_{\Rd} \eta_r(x-y) \,d\nu(y)\,,
\quad
x\in \Rd\,.
\end{equation}
Then there exist constants~$C=C(\delta,d,\lambda,\Lambda)\in [1,\infty)$, and $\al=\al(\la, \La, d, p)\in (0,1)$ and a random variable $\X$ with $\X=\O_{p}(C)$ such that, for every~$r \in [ \X ,R^{1-\delta}]$, we have the estimate
\begin{equation}
\label{e.harmapprox}
\bigl\| \tr \bigl( \ahom D^2 N_r \bigr) \bigr\|_{L^\infty(B_R)} 
\leq 
Cr^{-d-2-\alpha}| \nu |(B_{2R})
\,.
\end{equation}
\end{lemma}
\begin{proof}
Define, for every $r,t \in (0,\infty)$ and~$x\in\Rd$, 
\begin{equation}
\label{e.vr.def}
v_r(t,x):= 
\int_{\Rd} 
\int_{\Rd} 
\eta_r (x-z) P(t,y,z)\,dz
\,d \nu(y)=\int_{\R^{d}} \left( \eta_{r}\ast P(t,y, \cdot)\right)(x)\,d\nu(y),
\end{equation}
and
\begin{equation}
\label{e.wr.def}
w_r(t,x) 
:=
\int_{\Rd} \int_{\Rd} 
\eta_r (x-z) 
\overline{P}(t,y-z) 
\,dz \,d\nu(y)
=
\bigl( N_r \ast \overline{P}(t,\cdot) \bigr) (x)
\,.
\end{equation}
It is clear that~$w_r$ is a solution of the homogenized equation:
\begin{equation}
\label{e.wr.eq}
\partial_t w_r  - \tr (\ahom D^2w_r) = 0 \quad \mbox{in} \ (0,\infty) \times \Rd. 
\end{equation}
Observe that~$w_r(0,\cdot) = N_r$. Our proof strategy is broken into the following three steps:
\begin{enumerate}
\item[1.] Show that~$\partial_t v_r$ is small, using that~$\nu$ is an invariant measure;
\item[2.] Show that~$v_r - w_r$ is small, using homogenization error estimates;
\item[3.] Deduce that~$\tr(\ahom D^2 w_r)$ is small, using interpolation and crude control on~$\partial^2_t (v_r-w_r)$. 
\end{enumerate}

\smallskip

\emph{Step 1.}
We use the assumption that~$\nu$ is an invariant measure to obtain a bound on~$\partial_t v_r$, at least for~$x\in B_R$ and~$t$ not too large. The precise claim is that, for every~$\delta \in (0,\tfrac12]$, there exist~$C<\infty$ and~$c>0$, depending on~$(\delta,d,\lambda,\Lambda)$ such that, for every~$r \leq R^{1-\delta}$, we have the estimate
\begin{equation}
\label{e.snuff}
\sup_{k\in \{1,2\}} 
\| \partial_t^k v_r \|_{L^\infty([0,r^2]\times B_R)}
\leq
C \exp \bigl( -c R^{2\delta} \bigr)| \nu |(B_{2R})
\,.
\end{equation}
To see why this should be so, observe that~$v_r(t,x)$ is the integral of the product of~$\nu$ and a solution of the heterogeneous equation and thus  
\begin{align*}
\partial_t v_r(t,x)
&
:= 
\int_{\Rd} 
\partial_t \biggl[
\int_{\Rd} 
\eta_r (x-z) P(t,y,z)\,dz
\biggr]
\,d \nu(y)
\\ & \;
=
\int_{\Rd} 
\tr\biggl( \a  D^2_y \biggl[
\int_{\Rd} 
\eta_r (x-z) P(t,y,z)\,dz
\biggr] \biggr)
\,d \nu(y)
\,.
\end{align*}
If the term under the~$D^2_y$ had compact support in~$B_{2R}$, then this expression would be identically zero, since~$\nu$ is an invariant measure. 
While this is unfortunately not the case, it is still \emph{nearly} true for~$x\in B_R$,~$r\ll R$ and $t=r^2$, due to the decay of the parabolic Green function. Indeed, if~$x\in B_R$, then the integrand vanishes unless~$z \in B_{r}(x) \subseteq B_{R+r}$ and thus,  for~$y$ near~$\partial B_{2R}$, the factor~$P(t,y,z)$ must be very small if $t^2 \ll R$. To obtain a precise estimate from this idea, we select a cutoff function~$\chi \in C^\infty_c(B_{2R})$ satisfying 
\begin{equation}
\label{e.cutoff.chi}
\indc_{B_{4R/3}} \leq \chi \leq \indc_{B_{3R/2}}\,,
\qquad 
\|D^k \chi\|_{L^\infty(B_{2R})}
\leq 
C_kR^{-k}\,, \quad \forall k \in\{1,2\},
\end{equation}
and compute
\begin{align}
\label{e.partialt.vr}
\partial_t v_r(t,x) 
 & 
=
\int_{\Rd} 
\tr
\biggl(
\a
D^2_y
\biggl[
(1-\chi(y))
\int_{\Rd} 
\eta_r (x-z) 
P(t,y,z)\,dz
\biggr]
\biggr)
\,d \nu(y)
\,,
\end{align}
where we used that, since~$\nu$ is an invariant measure in~$B_{2R}$ and~$\chi$ has compact support in~$B_{2R}$,
\begin{equation*}
\int_{\Rd} 
\tr
\biggl(
\a  D^2_y
\biggl[
\chi(y)
\int_{\Rd} 
\eta_r (x-z) 
P(t,y,z)\,dz
\biggr]
\biggr)
\,d \nu(y)
=
0
\,.
\end{equation*}
Using~\eqref{e.cutoff.chi}, we have that 
\begin{align*}
\lefteqn{
\sup_{y\in \Rd}
\biggl| 
D^2_y
\biggl[
(1-\chi(y))
\int_{\Rd} 
\eta_r (x-z) 
P(t,y,z)\,dz
\biggr]
\biggr| 
} \qquad & 
\\ &
\leq
C
\sup_{k\in\{0,1,2\}}
\frac{1}{R^{2-k}}
\sup_{|y|>4R/3}
\biggl| 
D^k_y
\biggl[
\int_{\Rd} 
\eta_r (x-z) 
P(t,y,z)\,dz
\biggr]
\biggr| 
\,.
\end{align*}
We next estimate the term on the right side using the quenched bounds for the Green function of Proposition~\ref{le_prop_2} and then~$C^2$-Schauder estimates on the unit scale to get similar bounds for the derivatives. The latter are used very crudely in this context---we lose powers of the scale separation---but we can easily afford this because we have an exponential term in our favor which overwhelms everything. 
We estimate, for every~$r \leq R^{1-\delta}$, ~$x\in B_R$, and~$t<r^2$,
\begin{equation*}
\sup_{|y|>4R/3}
\biggl| 
D^k_y
\biggl[
\int_{\Rd} 
\eta_r (x-z) 
P(t,y,z)\,dz
\biggr]
\biggr| 
\leq
Cr^{-k} 
\exp \left( 
-\frac{R^2}{r^2} 
\right) 
\leq
Cr^{-k} \exp \bigl( -c R^{2\delta} \bigr)\,,
\end{equation*}
where the constants~$c$ and~$C$ depend on~$\delta$ in addition to~$(d,\lambda,\Lambda)$.
This completes the proof of~\eqref{e.snuff} for~$k=1$. For~$k=2$, we differentiate~\eqref{e.partialt.vr} again in~$t$ and use that~$\partial_tP$ is another solution of the parabolic equation, which therefore has~$C^2$ estimates on the unit scale. This kicks up more powers of~$r$ which can be absorbed by the exponential. 

\smallskip

\emph{Step 2.}
We apply the homogenization error estimate to obtain a bound on~$v_r - w_r$. Let $\X=\Y_{p}$ as defined in Remark \ref{r.Ystep1}. The precise claim is that, for every {$r\geq \X$},
\begin{equation}
\label{e.smug}
\| v_r - w_r \|_{L^\infty([0,r^2]\times B_R)} 
\leq 
C r^{-d-\alpha}
| \nu |(B_{2R})
\,.
\end{equation}
We begin by subtracting~\eqref{e.wr.def} from~\eqref{e.vr.def} to obtain
\begin{equation*}
(v_r - w_r) (t,x)
=
\int_{\Rd} 
\biggl[ 
\int_{\Rd}
\eta_r (x-z)
\bigl( P(t,y,z) - \overline{P}(t,y-z) \bigr)
\,dz
\biggr]
\,d\nu(y)
\,.
\end{equation*}
The term in brackets is the difference of the Cauchy problems for the heterogenous and homogenized equations, starting with initial data~$\eta_r(\cdot-x)$, evaluated at the point~$(t,y)$. Applying Proposition~\ref{prop_homogenize} yields, for $x\in B_R$ and $r\geq \X$, 
\begin{equation*}
\sup_{(t,y) \in [0, r^{2})\times B_{2R}}
\biggl|
\int_{\Rd}
\eta_r (x-z)
\bigl( P(t,y,z) - \overline{P}(t,y-z) \bigr)
\,dz
\biggr|
\leq
C r^{-d-\alpha}
\,.
\end{equation*}
The combination of the previous displays yields~\eqref{e.smug}. 

\smallskip

\emph{Step 3.} Interpolation and conclusion.
Using~\eqref{e.snuff} and~\eqref{e.smug}, we have, for every~$x\in B_R$, 
\begin{align*}
\lefteqn{
\| \partial_t w_r  \|_{L^\infty ( [0,r^2)\times B_R)}
} \quad &
\\ & 
\leq
\| \partial_t (w_r-v_r)   \|_{L^\infty ( [0,r^2)\times B_R)}
+
C \exp \bigl( -c R^{2\delta} \bigr)| \nu |(B_{2R})
\\ & 
\leq 
\| w_r-v_r  \|_{L^\infty ( [0,r^2)\times B_R)}^{\frac12} 
\| \partial_t^2( w_r-v_r ) \|_{L^\infty ( [0,r^2)\times B_R)}^{\frac12} 
+
C \exp \bigl( -c R^{2\delta} \bigr)| \nu |(B_{2R})
\\ & 
\leq 
Cr^{-\frac{d+\alpha}{2}}\left(| \nu |(B_{2R})\right)^{\frac12} \big(
\| \partial^2_t w_r  \|_{L^\infty ( [0,r^2)\times B_R)}^{\frac12} +\| \partial^2_t v_r  \|_{L^\infty ( [0,r^2)\times B_R)}^{\frac12} \big)
\\ & \quad +
C \exp \bigl( -c R^{2\delta } \bigr)| \nu |(B_{2R})
\,.
\end{align*}
Recalling \eqref{e.wr.eq} and \eqref{e.snuff}, we use estimates on caloric functions to substitute 
\begin{equation*}
\| \partial^2_t w_r  \|_{L^\infty ( [0,r^2)\times B_R)}
\leq 
Cr^{-4-d} | \nu |(B_{2R}) ,
\end{equation*}
and we obtain
\begin{equation*}
\| \partial_t w_r  \|_{L^\infty ( [0,r^2)\times B_R)}
\leq
Cr^{-d-2-\alpha/2}
| \nu |(B_{2R})
\,.
\end{equation*}
In view of~\eqref{e.wr.eq} and the fact that~$w_r(0,\cdot) = N_r$, the proof of \eqref{e.harmapprox} is now complete, upon redefining $\al$.
\end{proof}

We next show that, up to a small error, we can recover the small-scale behavior of the invariant measure from its large-scale spatial averages. 
Note that this estimate is already enough to imply a quenched quantitative uniqueness of the invariant measure.

\begin{lemma}[Weak-to-strong estimate]
\label{l.weak.to.strong}
Let~$\delta\in (0,\tfrac12]$ and $p\in (0,d)$. Let~$R\geq 1$ and~$\nu$ be a measure with support in~$B_{2R}$ which satisfies~\eqref{e.invmeas.weak} with~$U=B_{2R}$ and has density~$N\in L^1(B_{2R})$. 
Let~$\{ \eta_r \}_{r>0}$ denote the standard mollifier. Let~$N_r$ be defined as in~\eqref{e.Nr}. 
Then there exist constants~$C(\delta,d,\lambda,\Lambda)\in [1,\infty)$ and $\al=\al(\la, \La, d,p)\in (0,1)$, and a random variable $\X$ with $\X=\O_{p}(C)$ such that, for every~$r \in [ \X ,R^{1-\delta}]$ and~$x\in B_R$, 
\begin{equation}
\label{e.pointwise.N}
\bigl| 
N(x) 
-
m(x) N_r(x)
\bigr|
\leq
C(m(x)+1)r^{-d-\alpha}| \nu|(B_{2R})
\end{equation}
and, consequently, 
\begin{equation}
\label{e.reach}
R^{-d} | \nu |(B_R)
\leq 
C
\| \nu \ast \eta_r \|_{\underline{L}^1(B_{5R/2})}
+
Cr^{-d-\alpha} | \nu|(B_{2R})
\,.
\end{equation}
\end{lemma}
\begin{proof}
By the hypotheses, we may write~$d\nu(x) = N(x) \,dx$ for some~$N\in L^1(B_{2R})$.
Define 
\begin{equation*}
v(t,x) :=
\int_{\Rd}
P(t,y,x) \,d\nu(y)
=
\int_{\Rd}
P(t,y,x) N(y) \,dy
\end{equation*}
and
\begin{equation*}
v_r(t,x) := \int_{\Rd} P(t,y,x) N_r(y) \,dy\,.
\end{equation*}
By an argument similar to Step~1 of the proof of Lemma~\ref{l.harmapprox} (specifically \eqref{e.snuff}), we have that, for every~$t \in[0, R^{2(1-\delta)}]$,
\begin{equation*}
\| \partial_t v(t,\cdot) \|_{L^\infty(B_R)}
\leq
C \exp(-cR^{3\delta/2}) 
| \nu|(B_{2R})
\,.
\end{equation*}
Therefore, for every~$t \in[0, R^{2(1-\delta)}]$,
\begin{align*}
|v(t,x) - N(x) |
=
|v(t,x) - v(0,x)| 
&\leq
C t \exp (-cR^{3\delta/2})
| \nu|(B_{2R})\\
&
\leq 
C\exp(-cR^\delta)
| \nu|(B_{2R})
\,.
\end{align*}
We observe next using the Krylov-Safonov H\"older estimate that, for $\sigma=\sigma(\la, \La, d)\in (0,1)$, for~$t\in [r^2,R^{2(1-\delta)}]$, by Corollary \ref{c.NA}, 
\begin{align*}
| v(t,x) - v_r(t,x) |
&=
\bigg|\int_{\Rd} \bigl( P (t,\cdot,x) - P(t,\cdot,x) \ast \eta_r \bigr)(y) \,d\nu(y)\bigg|\\
&\leq
Cm(x)r^{-\sigma} t^{-\frac d2}
| \nu|(B_{2R})
\,.
\end{align*}
We deduce that
\begin{equation*}
\bigl| N(x) - v_r(r^2,x) \bigr| 
\leq 
Cm(x)r^{-d-\sigma}| \nu|(B_{2R})
+
C\exp(-cR^\delta)
| \nu|(B_{2R})
\,.
\end{equation*}
In particular,
\begin{equation}
\label{e.pointwise.N.0}
\biggl| 
N(x) 
-
\int_{\Rd} P(r^2,y,x) N_r(y) \,dy
\biggr|
\leq
C(m(x)+1)r^{-d-\alpha}| \nu|(B_{2R})
\,.
\end{equation}
We next use~\eqref{e.pointwise.N.0}, Corollary \ref{c.NA}, the first half of \eqref{e.weaklim.quant}, and \eqref{e.mL1}, to get, for every~$r \in [\X,R^{1-\delta}]$,  
\begin{align*}
|\nu|(B_R) 
&
=
\| N \|_{L^1(B_R)}
\\ & 
\leq
\! \int_{B_{R}}
\left| 
\int_{\Rd} P(r^2,y,x) N_r(y) \,dy
\right| \,dx
+
Cr^{-d-\sigma} | \nu|(B_{2R})
\! \int_{B_R} (1+m(x))\,dx 
\\ & 
\leq 
Cr^{-d}
\int_{B_R} \! \int_{\Rd} 
m(x)
\exp\left(-\frac{|x-y|^{2}}{Cr^2}\right)
|N_r(y)| \,dy \,dx 
+
CR^d r^{-d-\sigma} | \nu|(B_{2R})
\\ & 
\leq 
C  \| N_r \|_{L^1(B_{5R/2})}
+
CR^d r^{-d-\sigma} | \nu|(B_{2R})
\,.
\end{align*}
This yields~\eqref{e.reach} for $\al=\sigma$.

\smallskip

To obtain~\eqref{e.pointwise.N}, it suffices by~\eqref{e.pointwise.N.0} to show that, for every~$x\in B_R$, 
\begin{equation}
\label{e.Nrm.to.NP}
\Bigl| N_r(x) m(x)  - \int_{\Rd} P(r^2,y,x) N_r(y) \,dy
\Bigr| 
\leq
C(m(x) +1)
r^{-d-\alpha}| \nu |(B_{2R})\,.
\end{equation}
Fix another small exponent~$\theta>0$ to be selected below,~$x\in B_R$, and let~$h_r \in C^\infty(B_{r^{1+\theta}})$ be the solution of
\begin{equation*}
\left\{
\begin{aligned}
& - \tr \bigl( \ahom D^2 h_r \bigr) = 0 
& \mbox{in} & \ B_{r^{1+\theta}}(x) \,, \\
& h_r = N_r & \mbox{on} & \ \partial B_{r^{1+\theta}}(x) \,.
\end{aligned}
\right.
\end{equation*}
By~\eqref{e.harmapprox}, if~$r^{1+\theta} \leq R$, then 
\begin{equation}
\label{e.approximate.harm}
\| h_r - N_r \|_{L^\infty(B_{r^{1+\theta}})} 
\leq 
Cr^{-d-2-\alpha+2(1+\theta)}| \nu |(B_{2R})
\leq
Cr^{-d-\alpha/2}| \nu |(B_{2R})
\,,
\end{equation}
if we choose~$\theta < \alpha/4$. 
We then compute
\begin{align*}
\lefteqn{
N_r(x) m(x)  - \int_{\Rd} P(r^2,y,x) N_r(y) \,dy
} \qquad & 
\\ &
=
\bigl( N_r(x) - h_r(x) \bigr) m(x) 
+
\int_{\Rd} P(r^2,y,x) \bigl( h_r(y) - N_r(y)  \bigr) \,dy
\\ & \qquad 
-
\int_{\Rd} \bigl( P(r^2,y,x) - m(x) \overline{P}(r^2,y-x) \bigr) h_r(y)  \,dy
\\ & \qquad 
+
m(x) \Bigl( h_r(x) - 
\int_{\Rd}  \overline{P}(r^2,y-x) h_r(y)  \,dy \Bigr)
\,. 
\end{align*}
Each of the four terms on the right side is bounded from above by at most $m(x)$ times the right side of~\eqref{e.approximate.harm}. The first two terms are controlled by~\eqref{e.approximate.harm}, with the second term also using~\eqref{e.PGF.bound}, the third by~\eqref{e.GF} and the bound that $\|N_{r}\|_{L^{\infty}(\R^{d})}\leq r^{-d}\|N\|_{L^{1}(\R^{d})}$, and the last term using the mean value property for~$h_r$ and by chopping the tails of the heat kernel~ $\overline{P}(r^2,y-x)$ for $|y-x|> r^{1+\theta}$ (which is necessary because~$h_r$ is only harmonic in~$B_{r^{1+\theta}}(x)$, not all of~$\Rd$).
As a result, we get~\eqref{e.Nrm.to.NP} and the proof is now complete. 
\end{proof}

We have shown in Lemma~\ref{l.harmapprox} that, on sufficiently large scales, the invariant measure $\nu$ is close (in a weak sense) to a harmonic function. We will next use this property to deduce large-scale regularity for our invariant measures.

\begin{proof}[Proof of Theorem~\ref{t.largescaleC01}]
We perform an iteration down the scales, following~\cite{AS}. Here we adapt the argument presented in~\cite[Section 3.2]{AKMbook}. 

\smallskip

Let $d\mu(x) = m(x)\,dx$ be the unique stationary invariant measure constructed in Section~\ref{ss.m.exis.uniq}. Throughout this argument, we use the fact that by Theorem \ref{t.mto1} (specifically the first part of \eqref{e.m.conv}), $|\mu|(B_{r})\sim r^{d}$, and $\norm{\mu\ast \eta_{r}}_{\underline{L}^{1}(B_{r})}\approx 1$ for $r\geq \X$, with an error that is algebraic in $r$. We select another invariant measure~$d\nu(x) = N(x) \,dx$ with density~$N\in L^1(B_{2R})$. 
We let~$\al(\la, \La, d, p)>0$ denote the minimum of the two exponents in the statements of Lemmas~\ref{l.harmapprox} and~\ref{l.weak.to.strong} and Theorem \ref{t.mto1}.

\smallskip

Fix parameters~$\delta,\theta \in (0,\frac12]$ and~$r_0 \in [\X^{1/(1-\delta)} ,\tfrac{1}{10}R]$ to be selected below. Define
\begin{equation*}
M:=
\sup_{r \in [ r_0, 10^{-1}R]}
\frac1r 
\| (\nu - k_r \mu) \ast\eta_{r^{1-\delta}} 
\|_{\underline{L}^1(B_{5r})}
\,,
\end{equation*}
and, for $p\in \R^{d}$, let $\ell_{p}(x):=p\cdot x$. For $\nu$ as above, let 
\begin{equation*}
E(r):=\frac{1}{r}\inf_{p\in \R^{d}, k\in \R} \|(\nu\ast \eta_{r^{1-\delta}})-(\ell_{p}+k\mu\ast\eta_{r^{1-\delta}})\|_{\underline{L}^{1}(B_{5r})}. 
\end{equation*}
We denote $p_{r}$ and $k_{r}$ to be the minimizers, i.e. 
\begin{equation*}
E(r)=\frac{1}{r}\|(\nu-k_{r}\mu)\ast \eta_{r^{1-\delta}}-\ell_{p_{r}}\|_{\underline{L}^{1}(B_{5r})},
\end{equation*}
where the existence of minimizers follows from the fact that $E(r)$ is bounded from below by zero, and the fact that any minimizing sequence $\{(p_n,k_n)\}_{n\in\N}$ must remain in a compact subset of $\R^d\times\R$ due to the local $L^1$-integrability of $\nu$ and $\mu$.  Comparing to $p=0$ and $k=0$, by the triangle inequality and Young's inequality for convolutions, we have 
\begin{equation*}
\frac{1}{r}\|\ell_{p_{r}}+k_{r}(\mu\ast\eta_{r^{1-\delta}})\|_{\underline{L}^{1}(B_{5r})}\leq \frac{2}{r}\|\nu\ast \eta_{r^{1-\delta}}\|_{\underline{L}^{1}(B_{5r})}\leq \frac{C|\nu|(B_{6r})}{r^{d+1}}.
\end{equation*}
By equivalence of norms on $\R^{d+1}$ and since $\frac{1}{2}\leq \mu\ast \eta_{r^{1-\delta}}\leq 2$ for $r\geq \X$, this implies that 
\begin{equation}\label{e.krbd}
|p_{r}|+\frac{1}{r}|k_{r}|\leq  \frac{C}{r}\|\nu\ast \eta_{r^{1-\delta}}\|_{\underline{L}^{1}(B_{5r})}\leq \frac{C|\nu|(B_{6r})}{r^{d+1}}.
\end{equation}
This in particular yields that $|\nu-k_{r}\mu|(B_{6r})\leq C|\nu|(B_{6r})$. By comparing to $p=0$, we have $E(r)\leq M$. 

We now prove a general fact we will use throughout the argument. For any invariant measure $\tilde{\nu}$, for any $\theta s\leq r$ and $\delta\leq \al\wedge \tfrac{1}{2}$, 
\begin{equation}\label{e.dscale}
\|\tilde{\nu}\ast (\eta_{s^{1-\delta}}-\eta_{r^{1-\delta}})\|_{\underline{L}^{1}(B_{\theta s})}\leq Cs^{-d-\delta}|\tilde{\nu}|(B_{3r}), 
\end{equation}
where $C=C(\la, \La, d, \theta)$. By the triangle inequality, \eqref{e.pointwise.N}, Theorem \ref{t.mto1}, Young's inequality, and the Lipschitz bound on $\eta_{r}$, we have
\begin{align*}
\lefteqn{ 
\bigl\| \tilde{\nu} \ast (\eta_{s^{1-\delta}} - \eta_{r^{1-\delta}}) \bigr\|_{\underline{L}^1(B_{\theta s})} 
} \quad
 & 
\\ &
\leq
\fint_{B_{\theta s}}
\biggl |\int_{\Rd}
\big(\tilde{\nu} \ast \eta_{r}\big)(y) m(y)  (\eta_{s^{1-\delta}} - \eta_{r^{1-\delta}})(x-y)
\,dy
\biggr |
\,dx
\\ &
\qquad 
+\fint_{B_{\theta s}}
\int_{\Rd}
\big|\tilde{\nu} (y)-(\tilde{\nu} \ast \eta_{r})(y) m(y) \big| (\eta_{s^{1-\delta}} + \eta_{r^{1-\delta}})(x-y)
\,dy
\,dx
\\ &
\leq \fint_{B_{\theta s}}
\biggl |\int_{\Rd}
\big(\tilde{\nu} \ast \eta_{r}\big)(y) m(y)  (\eta_{s^{1-\delta}} - \eta_{r^{1-\delta}})(x-y)
\,dy
\biggr |
\,dx
\\ &
\qquad 
+ Cr^{-d-\alpha}|\tilde{\nu}|(B_{2r})\fint_{B_{\theta s}}\int_{\R^{d}} [m(y)+1](\eta_{s^{1-\delta}} + \eta_{r^{1-\delta}})(x-y)
\,dy\, dx\\
&\leq 
\fint_{B_{\theta s}}
\biggl |\int_{\Rd}
\big(\tilde{\nu} \ast \eta_{r}\big)(x) m(y)  (\eta_{s^{1-\delta}} - \eta_{r^{1-\delta}})(x-y)
\,dy
\biggr |
\,dx\\
&\qquad+
\fint_{B_{\theta s}}\!
\int_{\Rd}
\bigl |\big(\tilde{\nu}{\ast }\eta_{r}\big)(x) {-} \big(\tilde{\nu}{\ast} \eta_{r}\big)(y)\bigr | 
m(y) 
(\eta_{s^{1-\delta}} {+} \eta_{r^{1-\delta}})(x{-}y)
\,dy
\,dx
+
Cr^{-d-\alpha}|\tilde{\nu}|(B_{2r})
\\ &
\leq
\fint_{B_{\theta s}}\big|\tilde{\nu} \ast \eta_{r}\big|(x)
\left[\biggl |\int_{\Rd}
 m(y)\eta_{s^{1-\delta}}(x-y)\, dy -1\biggr | +\biggl | 1- \int_{\Rd}
 m(y) \eta_{r^{1-\delta}}(x-y)\, dy\biggr|\right]\, dx\\ &
 \qquad+\fint_{B_{\theta s}} |\tilde{\nu}|(B_{3r})\int_{\R^{d}}\|\nabla \eta_{r}\|_{\infty}|x-y|m(y) 
(\eta_{s^{1-\delta}} {+} \eta_{r^{1-\delta}})(x{-}y)
\,dy\, dx+Cr^{-d-\alpha}|\tilde{\nu}|(B_{2r})
\\&
\leq C s^{-d-\alpha(1-\delta)}|\tilde{\nu}|(B_{3r})
+C r^{-d} r^{-\delta}|\tilde{\nu}|(B_{3r})
+
Cr^{-d-\alpha}|\tilde{\nu}|(B_{3r})
\,,
\end{align*}

We can also prove a continuity estimate on $|k_{r}-k_{s}|$. Indeed, by the triangle inequality, \eqref{e.krbd}, and Young's inequality for convolutions, for any $s\leq r$, and assuming $\delta\leq \frac{\al}{\al+1}$,
\begin{align*}\label{e.krks}
|k_{r}-k_{s}|&\leq \|\ell_{p_{r}}+k_{r}-(\ell_{p_{s}}+k_{s})\|_{\underline{L}^{1}(B_{s})}\notag\\
&\leq \|\nu\ast\eta_{s^{1-\delta}}-(\ell_{p_{s}}+k_{s})\|_{\underline{L}^{1}(B_{s})}+\|\nu\ast\eta_{s^{1-\delta}}-(\ell_{p_{r}}+k_{r})\|_{\underline{L}^{1}(B_{s})}\notag\\
&\leq C\|\nu\ast\eta_{s^{1-\delta}}-(\ell_{p_{s}}+k_{s}(\mu\ast \eta_{s^{1-\delta}}))\|_{\underline{L}^{1}(B_{5s})}+|k_{s}|\|\mu\ast \eta_{s^{1-\delta}}-1\|_{\underline{L}^{1}(B_{s})}\\
&\quad+C\|\nu\ast\eta_{r^{1-\delta}}-(\ell_{p_{r}}+k_{r}(\mu\ast \eta_{r^{1-\delta}}))\|_{\underline{L}^{1}(B_{5s})}+|k_{r}|\|\mu\ast \eta_{r^{1-\delta}}-1\|_{\underline{L}^{1}(B_{s})}\\
&\quad+\|\nu\ast(\eta_{s^{1-\delta}}-\eta_{r^{1-\delta}})\|_{\underline{L}^1(B_{s})}\notag\\
&\leq CsE(s)+C\left(\frac{r}{s}\right)^{d}rE(r)+C\left(\|\nu\ast \eta_{s^{1-\delta}}\|_{\underline{L}^{1}(B_{5s})}s^{-\delta}+\left(\frac{r}{s}\right)^{d}\|\nu\ast \eta_{r^{1-\delta}}\|_{\underline{L}^{1}(B_{5r})}r^{-\delta}\right)\\
&\quad+Cs^{-d}|\nu|(B_{2r})\notag\\
&\leq C\left(\frac{r}{s}\right)^{d}\left[Mr+s^{-d}|\nu|(B_{2r})+s^{-\delta}\left(\|\nu\ast \eta_{s^{1-\delta}}\|_{\underline{L}^{1}(B_{5s})}+\|\nu\ast \eta_{r^{1-\delta}}\|_{\underline{L}^{1}(B_{5r})}\right)\right].
\end{align*}
This could also be improved by doing a shift. Letting $\tilde{\nu}:=\nu-k_{r}\mu$, it is clear that $|k_{r}-k_{s}|=|\tilde{k}_{r}-\tilde{k}_{s}|,$ and $\tilde{M}=M$. Thus, repeating the above calculation with $\tilde{\nu}$, and using the definition of $M$ and the triangle inequality, we have
\begin{align*}
|k_{r}-k_{s}|&\leq C\left(\frac{r}{s}\right)^{d}\left[Mr+s^{-d}|\nu-k_{r}\mu|(B_{2r})+s^{-\delta}\left(Ms+Mr+|k_{r}-k_{s}|\|\mu\ast \eta_{s^{1-\delta}}\|_{\underline{L}^{1}(B_{5s})}\right)\right]\\
&\leq C\left(\frac{r}{s}\right)^{d}\left[Mr+s^{-d}|\nu-k_{r}\mu|(B_{2r})+2s^{-\delta}|k_{r}-k_{s}|\right],
\end{align*}
which implies that that for $s\geq r_{0}$, 
\begin{equation}\label{e.krks}
|k_{r}-k_{s}|\leq C\left(\frac{r}{s}\right)^{d}\left[Mr+s^{-d}|\nu-k_{r}\mu|(B_{2r})\right].
\end{equation}

\emph{Step 1.} Application of the weak-to-strong estimate. 
We claim that there exists an exponent~$\delta_0(\la, \La, d, p)\in (0, \frac 12)$ and exponent $\beta(\al)$ such that for every $\delta\in (0, \delta_{0}]$ and $r\in [r_0,\frac14R]$, 
\begin{equation}
\label{e.weaktostrong.app1}
| \nu -k_{r}\mu| (B_{2r})
\leq 
CMr^{d+1}
+
CR^{-\beta \log \tfrac{R}{r}}
| \nu | (B_{R})
\,.
\end{equation}

By applying Lemma~\ref{l.weak.to.strong} (specifically \eqref{e.reach}), for every~$r\in [r_0,\frac14R]$, provided that~$\delta\leq \frac{\al}{2(d+\al)}\wedge \frac{1}{2}$,
\begin{align*}
| \nu -k_{r}\mu| (B_{2r})&\leq Cr^{d}\|(\nu-k_{r}\mu)\ast \eta_{r^{1-\delta}}\|_{\underline{L}^{1}(B_{5r})}+Cr^{-(1-\delta)(d+\al)+d}|\nu-k_{r}\mu|(B_{4r})\\
&\leq CMr^{d+1}+Cr^{-\al/2}|\nu-k_{r}\mu|(B_{4r}). 
\end{align*}
Now by \eqref{e.krks} and using that $|\mu|(B_{6r})\leq Cr^{d}$, we have
\begin{equation*}
| \nu -k_{r}\mu| (B_{2r})\leq CMr^{d+1}+Cr^{-\al/2}|\nu-k_{2r}\mu|(B_{4r}).
\end{equation*}
We now iterate this to get that for all $r\in [r_{0}, \frac{1}{6}R]$, with $r_{0}$ (deterministically) sufficiently large so that $Cr^{-\frac{\al}{2}}<1$, 
\begin{align*}
| \nu -k_{r}\mu| (B_{2r})&\leq CMr^{d+1}\sum_{j=0}^{\log_{2}\tfrac{R}{6r}}\left(Cr^{-\tfrac{\al}{2}}\right)^{j}+\left(Cr^{-\tfrac{\al}{2}}\right)^{(1+\log_{2}\tfrac{R}{6r})}\prod_{j=0}^{\log_{2}\tfrac{R}{6r}} 2^{-j\tfrac{\al}{2}}|\nu-k_{R/6}\mu|(B_{R/3})\\
&\leq CMr^{d+1}+\left(Cr^{-\tfrac{\al}{2}}\right)^{(1+\log_{2}\tfrac{R}{6r})}2^{-\tfrac{\al}{2}((\log_{2}\tfrac{R}{6r})^{2}+\log_{2}\tfrac{R}{6r})}|\nu|(B_{R})\\
&\leq CMr^{d+1}+CR^{-\beta \log \tfrac{R}{r}}|\nu|(B_{R}).
\end{align*}
Notice that upon doing the iteration, by \eqref{e.krks}, we also have the statement that for every $r\in [r_{0}, \tfrac{1}{10}R]$, 
\begin{align}\label{e.imp}
|\nu-k_{r}\mu|(B_{10r})&\leq |\nu-k_{5r}\mu|(B_{10r})+C|k_{r}-k_{5r}||\mu|(B_{10r})\notag\\
&\leq |\nu-k_{5r}\mu|(B_{10r})+C\left[Mr^{d+1}+|\nu-k_{5r}\mu|(B_{10r})\right]\notag\\
&\leq CMr^{d+1}+C|\nu-k_{5r}\mu|(B_{10r})\notag\\
&\leq CMr^{d+1}+CR^{-\beta \log \tfrac{R}{r}}
| \nu | (B_{R}),
\end{align}
and by an identical argument, 
\begin{equation}\label{e.3est}
|\nu-k_{r}\mu|(B_{3r})\leq CMr^{d+1}+CR^{-\beta \log \tfrac{R}{r}}.
\end{equation}

\emph{Step 2.} The excess decay iteration. We claim that there exists~$\delta_0(\la, \La, d, p)\in (0, \frac 12]$ and,~for each~$\delta \in (0,\delta_0]$, an exponent~$\beta(\delta,\la, \La, d, p)\in (0, \frac 12)$ such that, for every~$r \in [ r_0 ,\tfrac{1}{10}R]$,
\begin{equation}
\label{e.decay.me}
E(r) 
\leq 
C\Bigl( \frac {r}{R} \Bigr)^{\!\beta} 
\frac1R
\| \nu \ast \eta_{R^{1-\delta}} \|_{\underline{L}^1(B_{5R})}
+
C Mr^{-\beta} 
+
CR^{-1-d-\beta}
| \nu | (B_{R})
\,.
\end{equation}
We proceed by harmonic approximation, borrowing the regularity from the homogenized equation. 
According to Lemma~\ref{l.harmapprox} and a similar argument as in \eqref{e.approximate.harm}, for every~$r\in [\X^{1/(1-\delta)},R]$ there exists an~$\ahom$-harmonic function~$h_r$ which satisfies, by \eqref{e.imp},
\begin{align}\label{e.ha}
\| (\nu - k_r \mu)\ast \eta_{r^{1-\delta}} - h_r \|_{\underline{L}^1(B_{5r})} &\leq \norm{ (\nu - k_r \mu)\ast \eta_{r^{1-\delta}} - h_r}_{L^\infty(B_{5r})}\notag\\
&\leq 
Cr^{-(1-\delta)(d+2+\alpha)}r^{2}
| \nu - k_r \mu |(B_{10r})\notag\\
&\leq Cr^{-d-\frac{\al}{2}}| \nu - k_r \mu |(B_{10r}),\notag\\
&\leq C M r^{1-\frac{\al}{2}}
+
Cr^{-d-\frac{\al}{2}}R^{-\beta \log \tfrac{R}{r}}
| \nu | (B_{R})\,,
\end{align}
where the third inequality is valid after we impose the condition~$\delta\leq \frac{1}{2}\wedge\frac{\al}{4d+8}$, and the last inequality is a consequence of \eqref{e.imp}.

We finally proceed with the excess iteration. By using minimality, the triangle inequality, \eqref{e.ha}, \eqref{e.dscale}, and \eqref{e.3est}, and assuming that $5s\leq r$, 
\begin{align*}
E(s)&=\frac{1}{s}\|(\nu-k_{s}\mu)\ast\eta_{s^{1-\delta}}-p_{s}\cdot x\|_{\underline{L}^{1}(B_{5s})}\\
&\leq \frac{1}{s}\inf_{p\in\Rd} \|(\nu-k_{r}\mu)\ast\eta_{s^{1-\delta}}-p\cdot x\|_{\underline{L}^{1}(B_{5s})}\\
&\leq  \frac{1}{s}\inf_{p\in\Rd} \|(\nu-k_{r}\mu)\ast\eta_{r^{1-\delta}}-p\cdot x\|_{\underline{L}^{1}(B_{5s})}+\frac{1}{s}\|(\nu-k_{r}\mu)\ast (\eta_{r^{1-\delta}}-\eta_{s^{1-\delta}})\|_{\underline{L}^{1}(B_{5s})}\\
&\leq \frac{1}{s}\inf_{p\in\Rd} \|h_{r}-p\cdot x\|_{\underline{L}^{1}(B_{5s})}+\frac{1}{s}\|(\nu-k_{r}\mu)\ast\eta_{r^{1-\delta}}-h_{r}\|_{\underline{L}^{1}(B_{5s})}+\frac{1}{s}Cs^{-d-\delta}|\nu-k_{r}\mu|(B_{3r})\\
&\leq \frac{1}{s}\inf_{p\in\Rd} \|h_{r}-p\cdot x\|_{\underline{L}^{1}(B_{5s})}+\frac{C}{s}\left(\frac{r}{s}\right)^{d}(M r^{1-\frac{\al}{2}}
+
r^{-d-\frac{\al}{2}}R^{-\beta \log \tfrac{R}{r}}
| \nu | (B_{R}))\\
&\quad+\frac{C}{s}s^{-d-\delta}(CMr^{d+1}
+ CR^{-\beta \log \tfrac{R}{r}}
| \nu | (B_{R})).
\end{align*}
Using the $C^{1,1}$ regularity of~$\ahom$-harmonic functions, for~$s<\frac15r$, we have
\begin{equation*}
\frac1s
\inf_{p\in\Rd} 
\big\| h_r - \ell_p \big\|_{\underline{L}^1(B_{5s})}
\leq 
C \Bigl( \frac sr \Bigr) 
\frac1r 
\inf_{p\in\Rd} 
\big\| h_r - \ell_p \big\|_{\underline{L}^1(B_{5r})}
\,,
\end{equation*}
and thus by \eqref{e.ha}, 
\begin{align*}
\frac1s
\inf_{p\in\Rd} &
\big\| h_r - \ell_p \big\|_{\underline{L}^1(B_{5s})}\\
&\leq 
C \Bigl( \frac sr \Bigr) 
\left[ \inf_{p\in\Rd}\frac1r \big\|  (\nu-k_{r}\mu)\ast \eta_{r^{1-\delta}} - \ell_p \big\|_{\underline{L}^1(B_{5r})}+\frac1r \big\| (\nu-k_{r}\mu)\ast \eta_{r^{1-\delta}}-h_{r}\big\|_{\underline{L}^{1}(B_{5r})}\right]\\
&\leq C\left(\frac{s}{r}\right)\left[E(r)+ C M r^{-\frac{\al}{2}}
+
Cr^{-d-1-\frac{\al}{2}}R^{-\beta \log \tfrac{R}{r}} 
| \nu | (B_{R})\right].
\end{align*}
Combining the above, this implies
\begin{align*}
E(s)&\leq \left(\frac{s}{r}\right)E(r)+CM\left(\frac{r}{s}\right)^{d+1}s^{-\delta}+\frac{C}{s^{d+1}}s^{-\delta}R^{-\beta \log \tfrac{R}{r}}|\nu|(B_{R}).
\end{align*}

Taking $s:=\theta r$ for $\theta \in (0,\tfrac15]$ sufficiently small (this requires~$r \geq C$ since we have also imposed~$s\geq r^{1-\delta^3}$ above, and we can ensure it is satisfied by requiring~$r_0 \geq C$) depending only on~$(d,\lambda,\Lambda)$, we obtain
\begin{equation*}
E(\theta r) 
\leq 
\frac12 E(r) 
+
C M r^{-\beta} 
+
C
r^{-1-d-\beta} R^{-\beta \log \frac{R}{r}}
| \nu | (B_{R})
\,.
\end{equation*}
Iterating this inequality gives, after possibly shrinking~$\beta$,
\begin{equation*}
E(r) 
\leq 
C\Bigl( \frac {r}{R} \Bigr)^{\!\beta} 
E(R)+
C Mr^{-\beta} 
+
C
r^{-1-d-\beta} R^{-\beta \log \frac{R}{r}}
| \nu | (B_{R})
\,, \ \ \forall r\in \Bigl[\X^{1/(1-\delta)},\frac{1}{10}R\Bigr]\,. 
\end{equation*}
For the final term, we note that there exists $\Theta=\Theta(d, \la, \La)$ such that for every $R\geq \Theta r$, $\beta \log \frac{R}{r}\geq \beta \log \Theta \geq 1+d+\beta$. It then follows that by increasing the constant $C$ that we obtain~\eqref{e.decay.me}.

\smallskip
\emph{Step 3.} Control of the slopes. We claim that 
\begin{equation}
\label{e.slopebound}
\sup_{r\in [r_0,\tfrac{1}{10}R]}
|p_r| 
\leq 
\frac{C}R
\| \nu \ast \eta_{R^{1-\delta}} \|_{L^1(B_{5R})}
+
CMr_0^{-\beta}
+
CR^{-1-d- \beta}
| \nu | (B_{R})
\,.
\end{equation}
By~\eqref{e.krks} and \eqref{e.weaktostrong.app1}, 
\begin{align}
\label{e.soft.kr.diff.theta}
|k_r - k_{\theta r} |
&\leq 
C\left[Mr+r^{-d}|\nu-k_{r}\mu|(B_{2r})\right]
\,, \quad \forall r \in [r_0, \tfrac{1}{10}R ]
\notag\\
&\leq C\left[Mr
+
Cr^{-d} R^{-\beta \log \frac{R}{r}}
| \nu | (B_{R})\right].
\end{align}
By \eqref{e.dscale}, this implies
\begin{equation}
\label{e.change.k.softly}
\frac{1}{r}
\bigl\|
(k_{r} - k_{\theta r}) \mu \ast 
(\eta_{r^{1-\delta}}
-\eta_{(\theta r)^{1-\delta}} )
\bigr\|_{\underline{L}^1(B_{5r})}
\leq 
C(Mr^{- \beta}+Cr^{-1-d-\beta} 
R^{-\beta \log\tfrac{R}{r}}
| \nu | (B_{R}))
\,.
\end{equation}
Using the triangle inequality again with the help of~\eqref{e.dscale},~\eqref{e.decay.me}, and~\eqref{e.change.k.softly}, we get
\begin{align*}
\bigl| p_{r}-p_{\theta r}\bigr|
&
=
\frac{C}{r}
\| \ell_{p_r} - \ell_{p_{\theta r}} \|_{\underline{L}^1(B_{5r})}
\\ & 
\leq 
C
(E(r) {+} E(\theta r)) 
+
\frac{1}{r}
\bigl\|
(\nu-k_{r}\mu) \ast 
(\eta_{r^{1-\delta}}
-\eta_{(\theta r)^{1-\delta}} )
\bigr\|_{\underline{L}^1(B_{5r})}
\\ & \quad 
+ 
\frac{1}{r}
\bigl\|
(k_{r} - k_{\theta r}) \mu \ast 
(\eta_{r^{1-\delta}}
-\eta_{(\theta r)^{1-\delta}} )
\bigr\|_{\underline{L}^1(B_{5r})}
\\ & 
\leq 
C\Bigl( \frac {r}{R} \Bigr)^{\!\beta} 
\frac1R
\| \nu \ast \eta_{R^{1-\delta}} \|_{\underline{L}^1(B_{5R})}
+
CMr^{-\beta}
+
Cr^{-1-d-\beta}R^{-\beta \log\tfrac{R}{r}}
| \nu | (B_{R})
\,.
\end{align*}
Summing over the scales yields, for every~$j$ such that~$\theta^jR \geq r_0$, 
\begin{align}\label{535_1}
\bigl| p_{\theta^k R}\bigr|
&
\leq 
|p_R| +
\sum_{j=0}^k
\bigl| p_{\theta^j R}-p_{\theta^{j+1} R}\bigr|
\notag\\ & 
\leq
|p_R| + 
\frac{C}R
\| \nu \ast \eta_{R^{1-\delta}} \|_{\underline{L}^1(B_R)}
+
CMr_{0}^{-\beta}
+
C(\theta^{k}R)^{-1-d-\beta}R^{-\beta\log (\theta^{-k})}
| \nu | (B_{R})\,.
\end{align}
Since $\theta\in (0, \tfrac{1}{5}]$, there exists $k_{0}$ sufficiently large such that for all $k\geq k_{0}$, $\log(\theta^{-k}) \geq \frac{1+d+\beta}{\beta}$. It follows from \eqref{535_1} that, for any $k$, after increasing the final constant $C$,
\[ \bigl| p_{\theta^k R}\bigr|\leq |p_R| + 
\frac{C}R
\| \nu \ast \eta_{R^{1-\delta}} \|_{\underline{L}^1(B_R)}
+
CMr_{0}^{-\beta}
+CR^{-1-d-\beta}| \nu | (B_{R})
\,. \]
Using finally that by \eqref{e.krbd}, $|p_R| \leq 
\frac{C}R
\| \nu \ast \eta_{R^{1-\delta}} \|_{L^1(B_{5R})}$, 
 we obtain~\eqref{e.slopebound}.

\smallskip

\emph{Step 4.} 
The conclusion.
By the triangle inequality, 
\begin{align*}
\frac1r\big\| (\nu - k_r \mu)\ast \eta_{r^{1-\delta}} \big\|_{\underline{L}^1(B_{5r})}
&
\leq 
\frac1r\| \ell_{p_r} \|_{\underline{L}^1(B_{5r})} 
+
\frac1r 
\big\| (\nu - k_r \mu)\ast \eta_{r^{1-\delta}} - \ell_{p_r} \big\|_{\underline{L}^1(B_{5r})}
\\ \notag & 
=
C|p_r| + E(r). 
\end{align*}
Using~\eqref{e.decay.me},~\eqref{e.slopebound} and the prior display, we obtain, for every $r_0 \in[ \X^{1/(1-\delta)},\frac{1}{10}R]$, 
\begin{align}\label{e.triangle}
M 
&
= 
\sup_{r\in [r_0,10^{-1}R]}
\frac{1}{r}\big\| (\nu - k_r \mu)\ast \eta_{r^{1-\delta}} \big\|_{\underline{L}^1(B_{5r})}
\notag \\ &   
\leq 
\frac{C}{R}
\| \nu \ast \eta_{R^{1-\delta}} \|_{\underline{L}^1(B_{R})}
+
CMr_0^{-\beta}
+
CR^{-1-d-\beta}
| \nu | (B_{5R})
\,.
\end{align}
Taking~$r_0: = \max\{ \X^{1/(1-\delta)}, C\}$ for a sufficiently large constant~$C$, we can absorb the second term on the right side, to obtain
\begin{equation}\label{e.Mlower2}
M \leq 
\frac CR
\| \nu \ast \eta_{R^{1-\delta}} \|_{\underline{L}^1(B_{5R})}
+
CR^{-1-d-\beta}
| \nu | (B_{R})
\,.
\end{equation}
Using~\eqref{e.soft.kr.diff.theta}, if the supremum in $M$ occurs at $\bar{r}$, by replacing $k_{\bar{r}}$ by $k:=k_{r_{0}}$, we conclude that
\begin{align*}
\frac{1}{\bar{r}}\big\| (\nu - k \mu)\ast \eta_{\bar{r}^{1-\delta}} \big\|_{\underline{L}^1(B_{\bar{r}})}&\leq M+C\frac{1}{\bar{r}}|k-k_{\bar{r}}|\\
&\leq M+C\frac{1}{\bar{r}}\sum_{j=0}^{\log_{\theta}\frac{\bar{r}}{R}}|k_{\theta^{j+1}R}-k_{\theta^{j}R}|\\
&\leq M+C\frac{1}{r_{0}}\sum_{j=0}^{\log_{\theta^{-1}}\frac{R}{r_{0}}}C\left[M(\theta^{j}R)
+
C(\theta^{j}R)^{-d-\beta} R^{-\beta \log \frac{R}{r}}
| \nu | (B_{R})\right]\\
&\leq CM+CR^{-1-d-\beta}|\nu|(B_{R})\,,
\end{align*}
where the final inequality follows from a repetition of the argument after \eqref{535_1} above.  Therefore, 
\begin{equation*}
\sup_{r\in [r_0,R]}
\frac1r\big\| (\nu - k \mu)\ast \eta_{r^{1-\delta}} \big\|_{\underline{L}^1(B_{5r})}
\leq 
\frac CR
\| \nu \ast \eta_{R^{1-\delta}} \|_{\underline{L}^1(B_{5R})}
+
CR^{-1-d-\beta}
| \nu | (B_{R})
\,.
\end{equation*}
Finally, in view of~\eqref{e.weaktostrong.app1} and \eqref{e.Mlower2}, we deduce that 
\begin{equation*}
r^{-d} 
|\nu - k \mu |(B_{2r})
\leq
\frac{Cr}{R} 
\| \nu \ast \eta_{R^{1-\delta}} \|_{\underline{L}^1(B_{5R})}
+
\frac{Cr}{R}
R^{-d-\beta}
| \nu | (B_{R})
\,.
\end{equation*}
The proof is now complete.

\end{proof}

\section{Quantitative ergodicity for the environmental process}
\label{s.EPVP}

In this section we give the proof of Theorem~\ref{t.environmental}.

\begin{proof}[Proof of Theorem~\ref{t.environmental}]
Let $\xi\in(0,1)$, let $p\in(0,\xi d)$, and let $i, j \in \{ 1, \ldots , d \}$.  We define the process
\begin{equation*}
Z_t  
:=  \int_0^t 
\bigl( \a_{ij} (X_s) - \ahom_{ij} \bigr) \, ds,
\end{equation*}
and let $\mathcal{G}_{t}$ denote the standard filtration associated to $\left\{X_{t}\right\}$.  For a fixed $T\in(0,\infty)$, we will decompose $Z_T$ into a martingale part and an error part. Precisely, for $K\in\N$ to be fixed later, let $\tau = \nicefrac{T}{K}$ and let $\{M_k\}_{k=0}^K$ be the discrete martingale defined by $M_0=0$ and, for every $k\in\{0,\ldots,K-1\}$,
\begin{align*}
M_{k+1} & = M_k+\int_{k\tau}^{(k+1)\tau}\a_{ij}(X_s)\, ds - \mathbf{E}\left[\int_{k\tau}^{(k+1)\tau} \a_{ij}(X_s)\, ds \bigg| \mathcal{G}_{k \tau} \right]
\\ & = \int_0^{(k+1)\tau} \a_{ij}(X_s)\, ds - \sum_{n=0}^k \mathbf{E}\left[\int_{n\tau}^{(n+1)\tau} \a_{ij}(X_s)\, ds \bigg| \mathcal{G}_{n \tau}\right],
\end{align*}
and let $\{\mathcal{E}_k\}_{k=0}^K$ be the error terms defined by
\[ \mathcal{E}_{k+1} =  \sum_{n=0}^{k}\left(\mathbf{E}\left[\int_{n\tau}^{(n+1)\tau} \a_{ij}(X_s)\, ds \bigg| \mathcal{G}_{n \tau}\right]-\tau\ahom_{ij}\right),\]
for which we have that
\begin{equation}\label{ep_000}Z_T = M_K+\mathcal{E}_K.\end{equation}
We will treat the martingale term using Azuma's inequality.  Precisely, since for every $k\in\{0,\ldots,K-1\}$,
\[\abs{M_{k+1}-M_k}\leq 2\Lambda \tau,\]
it follows from Azuma's inequality and the fact that $M_0=0$ that, for every $\eta \in(0,\infty)$, by definition of $\tau$,
\begin{equation}\label{ep_00}\mathbf{P}[\abs{M_K}\geq \eta T] \leq \exp\left(-\frac{\eta^2 T^2}{8\Lambda^2K\tau^2}\right) = \exp\left(-\frac{\eta^2 T}{8\Lambda^2\tau}\right).\end{equation}
This completes the analysis of the martingale term.

To treat the error term, we will now fix the choice of constants.  Let $\beta = \nicefrac{(p+d\xi)}{2d}\in(\nicefrac{p}{d},\xi)$, let $\beta' = \xi$, let $p'=\nicefrac{p}{\beta}$, and let $K=\lceil T^{1-\beta'}\rceil$ so that $3^{-1}T^{\beta'}\leq \tau\leq 3T^{\beta'}$. For $\gamma\in(0,\infty)$ and $C\in[1,\infty)$, let $\Y_{p'} = \O_{p'}(C)$ be the random variable chosen to be sufficiently large (as in Remark \ref{r.Ystep1}), which guarantees that for $T^{\beta}\geq \Y_{p'}^{2}$, the estimates from Theorem~\ref{t.GF} and Theorem \ref{t.mto1} hold everywhere in the cube with side length $\Lambda T^2$. We henceforth restrict to the event that $T^\beta\geq \Y^2_{p'}$, where we observe by definition of $p$, $p'$, and $\beta$ that $\Y^{\beta^{-1}}_{p'} = \O_p(C)$.   We then first observe the identity, for every $n\in\{0,\ldots,K-1\}$,
\[\mathbf{E}\left[\int_{n\tau}^{(n+1)\tau} \a_{ij}(X_s)\, ds \bigg| \mathcal{G}_{n \tau}\right] = \int_{n\tau}^{(n+1)\tau} \int_{\R^d}P(s-n\tau,X_{n\tau},y)\a_{ij}(y)\, dy\, ds,\]
and hence $\mathcal{E}_{k+1}$ can be re-expressed as 
\begin{align*}
\mathcal{E}_{k+1}&=\sum_{n=0}^{k} \bigg[ \int_{n\tau}^{n\tau+T^{\beta}}\left( \int_{\R^d}P(s-n\tau,X_{n\tau},y)\a_{ij}(y)\, dy-\ahom_{ij}\right)\, ds\\
\notag&\quad +\int_{n\tau+T^{\beta}}^{(n+1)\tau}\left( \int_{\R^d}P(s-n\tau,X_{n\tau},y)\a_{ij}(y)\, dy-\ahom_{ij}\right)\, ds\bigg].
\end{align*}
We observe using the boundedness of $\a$ and $\ahom$ that the first term is controlled by
\begin{equation*}\label{ep_03}
\sum_{n=0}^{k}\int_{n\tau}^{n\tau+T^\beta}\left(\int_{\R^d}P(s-n\tau,X_{n\tau},y)\a_{ij}(y)\, dy - \ahom_{ij}\right) \, ds
\leq 2k\Lambda T^\beta,\end{equation*}
which implies, by definition of $\tau$, 
\begin{align}\label{e.1decomp}
T^{-1}\mathcal{E}_{K} & \leq \sum_{n=0}^{K-1} \left(\tau^{-1} \int_{n\tau+T^{\beta}}^{(n+1)\tau}\left( \int_{\R^d}P(s-n\tau,X_{n\tau},y)\a_{ij}(y)\, dy-\ahom_{ij}\right)\, ds\right)
\\ \nonumber & \qquad +2K\Lambda T^{\beta-1}.
\end{align}

To estimate, the second term, we will form the decomposition, for $X^*_t := \max_{0\leq t\leq T}\abs{X_t}$,
\begin{align}\label{e.bigdecomp}
&\int_{n\tau+T^{\beta}}^{(n+1)\tau}\left( \int_{\R^d}P(s-n\tau,X_{n\tau},y)\a_{ij}(y)\, dy-\ahom_{ij}\right)\, ds\\
\notag &=\int_{n\tau+T^{\beta}}^{(n+1)\tau}\left( \int_{\R^d}P(s-j\tau,X_{n\tau},y)\mathbf{1}_{\{X^*_T> \frac{1}{2}T^2\}}\a_{ij}(y)\, dy\right)\, ds\\
\notag &\quad+\int_{n\tau+T^{\beta}}^{(n+1)\tau}\left( \int_{B_{T^{2}}}P(s-n\tau,X_{n\tau},y)\mathbf{1}_{\{X^*_T\leq \frac{1}{2}T^2\}}\a_{ij}(y)\, dy-\ahom_{ij}\right)\, ds\\
\notag &\quad +\int_{n\tau+T^{\beta}}^{(n+1)\tau}\left( \int_{\R^{d}\setminus B_{T^{2}}}P(s-n\tau,X_{n\tau},y)\mathbf{1}_{\{X^*_T\leq \frac{1}{2}T^2\}}\a_{ij}(y)\, dy\right)\, ds.
\end{align}
It follows from the boundedness of $\A$ and Proposition~\ref{le_prop_2} that the first term on the righthand side of \eqref{e.bigdecomp} satisfies
\begin{multline}\label{ep_20}
\mathbf{P}\bigg[\sum_{n=0}^{K-1}\int_{n\tau+T^{\beta}}^{(n+1)\tau}\abs{\int_{\R^d}P(s-n\tau,X_{n\tau},y)\mathbf{1}_{\{X^*_T> \frac{1}{2}T^2\}}\a_{ij}(y)\, dy}\, ds\neq 0\bigg]
\\   
\leq \La \mathbf{P}\left[X^*_T\geq \frac{1}{2}T^2\right]\leq \Lambda\exp\left(-\frac{T^3}{8\Lambda}\right).
\end{multline}
For the third term of \eqref{e.bigdecomp}, the boundedness of $\a$, Proposition~\ref{le_prop_2}, and $s-n\tau\geq T^\beta$ yield that there exists $C=C(\la, \La, d,\al_{0}, K_{0})\in[1,\infty)$ such that
\begin{align}\label{ep_11}
\int_{n\tau+T^{\beta}}^{(n+1)\tau} \biggl |\int_{\R^{d}\setminus B_{T^{2}}}
\!\!P(s-n\tau,X_{n\tau},y)\mathbf{1}_{\{X^*_T\leq \frac{1}{2}T^{2}\}}\a_{ij}(y)\, dy \biggr|\, ds
 \leq C\Lambda (\tau-T^{\beta}) \exp\left(\!-\frac{T^3}{C}\right)
\,.
\end{align}
To control the second term on the righthand side of \eqref{e.bigdecomp}, we form the decomposition
\begin{align*}
\lefteqn{
\int_{B_{T^{2}}}P(s-n\tau,X_{j\tau},y)\mathbf{1}_{\{X^*_T\leq \frac{1}{2}T^2\}}\a_{ij}(y)\, dy
} \qquad &
\\ &
= \int_{B_{T^2}}m(y)\overline{P}(s-n\tau,X_{j\tau},y)\mathbf{1}_{\{X^*_T\leq \frac{1}{2}T^2\}}\a_{ij}(y)\, dy
\\ & \qquad + \int_{B_{T^2}}(P(s-n\tau,X_{n\tau},y)-m(y)\overline{P}(s-n\tau,X_{n\tau},y))\mathbf{1}_{\{X^*_T\leq \frac{1}{2}T^2\}}\a_{ij}(y)\, dy.
\end{align*}
Theorem~\ref{t.GF}, Proposition \ref{p.Holder.continuity}, and the boundedness of $\a$ yield that, for $\Y^2_{p'}\leq T^{\beta}\leq (s-n\tau)$,
\begin{align*}
 & \biggl |\int_{B_{T^2}}\left(P(s-n\tau,X_{n\tau},y)-m(y)\overline{P}(s-n\tau,X_{n\tau},y)\right)\mathbf{1}_{\{X^*_T\leq \frac{1}{2}T^2\}}\a_{ij}(y)\, dy \biggr |
\\ \nonumber &\qquad  \leq C\int_{\R^d}(s-n\tau)^{-\frac{d}{2}}\left(\frac{s-n\tau}{\Y^2_{p'}}\right)^{-\gamma}m(y)\exp\left(-\frac{\abs{X_{n\tau}-y}^2}{C(s-n\tau)}\right)\, dy
\leq C\Y^{2\gamma}_{p'}(s-n\tau)^{-\gamma}
\end{align*}
and hence 
\begin{align*}
&\int_{n\tau+T^{\beta}}^{(n+1)\tau}\biggl ( \int_{B_{T^{2}}}\!\! P(s-n\tau,X_{n\tau},y)\mathbf{1}_{\{X^*_T\leq \frac{1}{2}T^2\}}\a_{ij}(y)\, dy-\ahom_{ij}\biggr )\, ds\\
\notag &\quad\quad\quad\leq C\int_{n\tau+T^{\beta}}^{(n+1)\tau} \!\int_{B_{T^2}}\!\!
\overline{P}(s-n\tau,X_{n\tau},y)\mathbf{1}_{\{X^*_T\leq \frac{1}{2}T^2\}}\left(m(y)\a_{ij}(y) -\ahom_{ij}\right)\, dy\, ds\\
&\quad\quad\quad
+\frac{C}{1-\ga}\Y^{2\ga}_{p'}\tau^{1-\ga}.
\end{align*}
An application of Theorem~\ref{t.mto1}, a dyadic approximation of the heat kernel, and Proposition \ref{p.Holder.continuity} yield that, for some $\theta\in (0,1)$, 
\begin{align*}
\bigg|\int_{n\tau+T^{\beta}}^{(n+1)\tau}&\int_{B_{T^2}}\overline{P}(s-n\tau,X_{n\tau},y)\mathbf{1}_{\{X^*_T\leq \frac{1}{2}T^2\}}(m(y)\a_{ij}(y) -\ahom_{ij})\, dy \, ds\bigg|
\\& \leq C\int_{n\tau+T^\beta}^{(n+1)\tau}
\biggl ((s-n\tau)^{-\frac{\gamma}{2}}+ C\Lambda \exp\biggl (\!-\frac{(s-n\tau)^\theta}{C}\biggr )\biggr )\, ds
\\& \leq \frac{C}{1-\nicefrac{\gamma}{2}} \tau^{1-\frac{\gamma}{2}}+ C\Lambda\tau \exp\biggl (\!-\frac{T^{\beta\theta}}{C}\biggr ).
\end{align*}
Combining the last two displays, we have 
\begin{multline}\label{e.2decomp}
\int_{n\tau+T^{\beta}}^{(n+1)\tau}
\bigg( \int_{B_{T^{2}}}
P(s-n\tau,X_{n\tau},y)\mathbf{1}_{\{X^*_T\leq \frac{1}{2}T^2\}}\a_{ij}(y)\, dy-\ahom_{ij}\bigg)
\, ds\\
\leq C \tau^{1-\frac{\gamma}{2}}+ C\Lambda\tau \exp\left(-\frac{T^{\beta\theta}}{C}\right)+\frac{C}{1-\ga}\Y^{2\ga}_{p'}\tau^{1-\ga}.
\end{multline}
Returning to \eqref{e.bigdecomp}, we have from \eqref{ep_11} and \eqref{e.2decomp} that
\begin{align*}
&\int_{n\tau+T^{\beta}}^{(n+1)\tau}\left( \int_{\R^d}P(s-n\tau,X_{n\tau},y)\a_{ij}(y)\, dy-\ahom_{ij}\right)\, ds \\
\notag &\leq \int_{n\tau+T^{\beta}}^{(n+1)\tau}\left( \int_{\R^d}P(s-n\tau,X_{n\tau},y)\mathbf{1}_{\{X^*_T> \frac{1}{2}T^2\}}\a_{ij}(y)\, dy\right)\, ds\\
\notag&\qquad+C\Lambda (\tau-T^{\beta}) \exp\left(-\frac{T^3}{C}\right)+\frac{C}{1-\nicefrac{\gamma}{2}} \tau^{1-\frac{\gamma}{2}}+ C\Lambda\tau \exp\left(-\frac{T^{\beta\theta}}{C}\right)+\frac{C}{1-\ga}\Y^{2\ga}_{p'}\tau^{1-\ga}.
\end{align*}
Returning to \eqref{e.1decomp}, we have 
\begin{align*}
T^{-1}\mathcal{E}_{K}&\leq \sum_{n=0}^{K-1} \bigg[\tau^{-1} \int_{n\tau+T^{\beta}}^{(n+1)\tau}\left( \int_{\R^d}P(s-n\tau,X_{n\tau},y)\mathbf{1}_{\{X^*_T> \frac{1}{2}T^2\}}\a_{ij}(y)\, dy\right)\, ds\bigg]\\
& \qquad +K\left(\frac{C}{1-\nicefrac{\gamma}{2}} \tau^{-\frac{\gamma}{2}}+ C\Lambda \exp\left(-\frac{T^{\beta\theta}}{C}\right)+\frac{C}{1-\ga}\Y^{2\ga}_{k, p'}\tau^{-\ga}+2\La T^{\beta-1}\right).
\end{align*}
from which, on the event $\Y_{p'}^2\leq T^\beta$, since $\tau\geq 3^{-1}T^{\beta'}$, for a deterministic $C\in(0,\infty)$,
\begin{align}\label{ep_16}
T^{-1}\mathcal{E}_{K}&\leq \sum_{n=0}^{K-1} \bigg[\tau^{-1} \int_{n\tau+T^{\beta}}^{(n+1)\tau}\left( \int_{\R^d}P(s-n\tau,X_{n\tau},y)\mathbf{1}_{\{X^*_T> \frac{1}{2}T^2\}}\a_{ij}(y)\, dy\right)\, ds\bigg]\\
\notag&\qquad+CKT^{-(\beta'-\beta)\ga}.
\end{align}
It then follows from \eqref{ep_000}, \eqref{ep_00}, \eqref{ep_20}, \eqref{ep_16}, and $|K-T^{1-\beta'}|\leq 1$ that, for every $\eta\in (2CT^{(\beta-\beta')(1-\ga)}, \infty)$,
\begin{align}\label{ep_25}
&\mathbf{P}[T^{-1}\abs{Z_T}\geq \eta]\notag \\
& \leq \mathbf{P}[\abs{M_K}\geq 2^{-1}\eta T] \notag \\
&\quad + \mathbf{P}\left[\sum_{n=0}^{K-1}\tau^{-1}\int_{n\tau}^{(n+1)\tau}\abs{\int_{\R^d}P(s-n\tau,X_{n\tau},y)\mathbf{1}_{\{X^*_T> \frac{1}{2}T^2\}}\a_{ij}(y)\, dy}\, ds\neq 0\right]
\\ \nonumber & \leq \Lambda\exp\left(-\frac{T^{3}}{8\Lambda}\right)+\exp\left(-\frac{\eta^2 T^{1-\xi}}{32\Lambda^2}\right).
\end{align}
We then conclude that, for every for every $\eta \geq 2CT^{\frac{\gamma\delta}{2}}$,
\[ \mathbf{P}[T^{-1}\abs{Z_T}\geq \eta] \leq (1+\Lambda)\exp\left(-\frac{\eta^2 T^{1-\xi}}{32\Lambda^2}\right),\]
where for large $\eta$ we are using the boundedness of the random variable $T^{-1}Z_T$ to account for the first term on the righthand side of \eqref{ep_25}.  This completes the proof.  \end{proof}

\appendix

\section{Deterministic estimates on the adjoint equation}
\label{a.adjoint.estimates}

In this appendix, we prove deterministic estimates on the adjoint equation
\begin{equation}
\label{e.adjoint.m}
L^* m := 
-\sum_{i,j=1}^d\partial_i \partial_j  \bigl( \a_{ij} m \bigr) = 0 \,.
\end{equation}
The estimates we consider here are purely deterministic. That is, we do not assume that here~$\a$ is a random field, or consider statistical properties of solutions. We are interested only in the PDE theory for weak solutions of~\eqref{e.weaksol.adjoint}. 
Most of the results we present here have not, to our knowledge, appeared in the literature before. Some of them are folklore-type results (perhaps well-known to a few experts); others are not very difficult to extrapolate from known results but have nonetheless not to our knowledge been written down. 

\smallskip

Writing~\eqref{e.adjoint.m} in the way we have above makes it appear that~$m$ is a scalar function, but we actually first formulate weak solutions for general Radon measures. If~$U \subseteq\Rd$ is a domain and~$\mu$ is a Radon measure on~$U$, then we say that~$\mu$ is a \emph{weak solution of~\eqref{e.adjoint.m} in~$U$} if 
\begin{equation}
\label{e.weaksol.adjoint}
-\int_{U} \a_{ij}(x) \partial_i \partial_j \varphi(x) \,d\mu(x) = 0 \,, \quad \forall \varphi \in C^\infty_c(U)\,.
\end{equation}
Of course,~\eqref{e.weaksol.adjoint} is equivalent to the statement obtained by replacing the condition~$\varphi \in C^\infty_c(U)$ with~$\varphi \in C^2_c(U)$. 
If~$\mu$ happens to have a locally integrable density~$m$, then we also say that~$m$ is a weak solution of~\eqref{e.adjoint.m} if~$\mu$ is. 

\smallskip

\smallskip

In the following proposition, we show that a nonnegative Radon measures satisfying~\eqref{e.weaksol.adjoint} must have a density in~$L^{d/(d-1)}$ with a corresponding estimate which does not depend on the modulus of continuity of~$\a$. 
This kind of~$L^{d/(d-1)}$ estimate for the ``doubly divergence form'' adjoint equation goes back to the work of Bauman~\cite{B}. The exponent~$d/(d-1)$ is expected to appear since it is the H\"older dual exponent to~$d$, which arises in the ABP inequality for nondivergence form equations. 
However, in that work and in all the subsequent works of which we are aware, the constants appearing in localized versions of this estimate depended on the modulus of continuity of the coefficients and the estimate is therefore not scale-invariant. This is because a cutoff function creates a term that requires gradient estimates to control, and gradient estimates are typically valid only if the coefficients are continuous. 
Note that estimates which depend on the modulus of continuity are essentially useless at large scales, even if the coefficients are smooth. 

\smallskip

We present here a scale-invariant estimate with constants depending only on~$(d,\lambda,\Lambda)$, which was shown to us by Guido De Philippis, who has graciously agreed to allow us to include it here. The argument is a variant of the original argument of Bauman~\cite{B}, but dualizes the estimates for the Monge-Amp\`ere equation (the ABP inequality) rather than the Calder\'on-Zygmund estimates. 

\smallskip

Since the estimate we get is scale-invariant, we actually get a slightly better integrability exponent than~$d/(d-1)$, thanks to the Gehring lemma. This cannot be improved further: see Remark~\ref{r.nonono} below. 
\begin{proposition}
\label{p.Guido}
Let~$\a: B_{r} \to \mathcal{S}^{d}$ be continuous and satisfy~\eqref{e.unifellip} and consider a nonnegative Radon measure~$\mu$ on~$B_r$ satisfying
\begin{equation}
\label{e.invmeas.assump}
-\int_{B_r}\a_{ij} (x) \partial_{i}\partial_j \varphi(x) d\mu(x) 
= 0 \,, 
\quad \forall \varphi\in C^2_c(B_r)\,.
\end{equation}
Then there exist~$\delta(d,\lambda,\Lambda)>0$ and~$C(d,\lambda,\Lambda)<\infty$ and a function~$m\in L^{\frac{d}{d-1} + \delta}_{\mathrm{loc}}(B_r)$ satisfying
\begin{equation*}
d\mu(x) = m(x)\,dx
\end{equation*}
and
\begin{equation}
\label{e.Guido.plus}
\| m \|_{\underline{L}^{\frac d{d-1} + \delta} (B_{\sfrac r2})}
\leq 
\frac{C \mu(B_r)}{|B_r|}
\,.
\end{equation}
\end{proposition}
\begin{proof}
By the Gehring lemma (see~\cite[Lemma C.4]{AKMbook}), it suffices to prove the statement of the proposition for~$\delta = 0$. That is, we want to show that~$\mu$ has a density~$m$ which satisfies, for a constant~$C(d,\lambda,\Lambda)<\infty$, the estimate
\begin{equation}
\label{e.Guido.delta=0}
\| m \|_{\underline{L}^{\frac d{d-1}} (B_{\sfrac r2})}
\leq 
\frac{C \mu(B_r)}{|B_r|}
=
C \| m \|_{\underline{L}^1(B_r)} 
\,.
\end{equation}
By scaling, we may assume~$r=1$. 

\smallskip

Fix~$f \in C^\infty_c(B_{1/2})$ with~$f \geq 0$ and let~$v \in C^\infty(B_4)$ satisfy the following Dirichlet problem for the Monge-Amp\`ere equation:
\begin{equation*}
\left\{
\begin{aligned}
& \det D^2 v = f^{d}  & \mbox{in} & \ B_4\,, \\ 
& v = 0 & \mbox{on} & \ \partial B_4\,.
\end{aligned}
\right.
\end{equation*}
Then~$v\in C^2(B_4)$ is convex, and we have, for a constant~$C(d)<\infty$, the estimate
\begin{equation*}
\| v \|_{L^\infty(B_4)} + \| D v \|_{L^\infty(B_4)} 
\leq 
C \| f^d \|_{L^1(B_4)}^{\sfrac 1d}
=
C \| f \|_{L^d(B_{1/2})} \,.
\end{equation*}
Select a cutoff $\zeta \in C^\infty_c(B_{1})$ satisfying~$\indc_{B_{\sfrac12}} \leq \zeta \leq \indc_{B_{\sfrac34}}$ and~$\| D^j\zeta \|_{L^\infty(B_1)} \leq 10$ for~$j\in\{1,2\}$. Using the arithmetic-geometric mean inequality, and the fact that~$v$ is convex, we have 
\begin{align*}
\int_{B_1} f(x) \,d\mu(x) 
&
=
\int_{B_1} \zeta(x) f(x) \,d\mu(x) 
\\ & 
=
\int_{B_1} \zeta(x) \bigl( \det(D^2v(x)) \bigr)^{\sfrac 1d} \,d\mu(x) 
\\ & 
\leq 
\frac{1}{d} \int_{B_1} \zeta(x) \Delta v(x)  \,d\mu(x) 
\leq 
\frac {1}{\lambda d}\int_{B_1} 
\zeta(x)\a_{ij}(x) \partial_{i}\partial_j v(x)\,d\mu(x)\,.
\end{align*}
We next use the identity  
\begin{equation}
\label{e.easyiden}
\zeta\a_{ij}\partial_{i}\partial_j v
=
\a_{ij} \partial_{i}\partial_j (\zeta v) 
-
v \a_{ij}\partial_{i}\partial_j \zeta
-2 \a_{ij} \partial_i v\partial_j \zeta 
\end{equation}
and the assumption that~$\mu$ is an invariant measure, i.e.~\eqref{e.invmeas.assump}, to compute
\begin{align*}
\lefteqn{
\int_{B_1} 
\zeta(x)\a_{ij}(x) \partial_{i}\partial_j v(x)\,d\mu(x)
} \qquad & 
\\ &
=
- \int_{B_1} 
\bigl( 
v(x) \a_{ij}(x)\partial_{i}\partial_j \zeta(x)
-2 \a_{ij}(x) \partial_i v(x)\partial_j \zeta(x)
\bigr) \,d\mu(x)
\\ & 
\leq 
\bigl( 
C \| v\|_{L^\infty(B_1) } \| D^2 \zeta \|_{L^\infty(B_1)}
+
C \| Dv\|_{L^\infty(B_1) } \| D \zeta \|_{L^\infty(B_1)}
\bigr) \mu(B_1)
\\ & 
\leq 
C \mu(B_1) \| f \|_{L^d(B_{\sfrac12})}\,.
\end{align*}
We have shown that, for every~$f \in C^\infty_c(B_{1/2})$ with~$f \geq 0$,
\begin{equation*}
\int_{B_1} f(x)\,d\mu(x)
\leq 
C \mu(B_1) \| f \|_{L^d(B_{\sfrac12})} \,.
\end{equation*}
Therefore, by density,~$f\mapsto \int_{B_1} f(x)\,d\mu(x)$ can be extended to a bounded linear functional on~$L^d(B_{\sfrac12})$ and it therefore follows by the representation theorem that~$\mu\vert_{B_{\sfrac12}}$ is given by a density belonging to~$L^{d/(d-1)}(B_{\sfrac12})$, which, if we denote it by~$m$, satisfies~\eqref{e.Guido.delta=0}. 
The result follows. 
\end{proof}

\begin{remark}
\label{r.nonono}
We claim that Proposition~\ref{p.Guido} cannot be improved, in the sense that, for any~$q>\frac{d}{d-1}$ and~$K<\infty$, there exists~$\a: B_{1} \to \mathcal{S}^{d}$ which is continuous and uniformly elliptic and a function~$m\in L^{\frac{d}{d-1}}(B_1)$ satisfying
\begin{equation*}
L^* m = 0 \quad \mbox{in} \ B_1,
\end{equation*}
such that 
\begin{equation}
\label{e.mKm}
\| m \|_{L^q(B_1)} > K \| m \|_{L^1(B_1)}\,.
\end{equation}
For instance, let us consider
\begin{equation}
\label{e.a.example}
\a(x) = 
\lambda \Bigl( I_d - \frac{x\otimes x}{|x|^2} \Bigr)
+ \Lambda \frac{x\otimes x}{|x|^2}\,.
\end{equation}
Then, by a direct computation, we find that the function
\begin{equation*}
m(x):= |x|^{-\gamma} \,,  \quad \gamma:= (d-1) \frac{\Lambda-\lambda}{\Lambda}\,,
\end{equation*}
satisfies, for every~$j\in\{1,\ldots,d\}$, 
\begin{equation*}
\sum_{i=1}^d \partial_i \bigl(  \a_{ij} m(x)  \bigr) = 0
\,.
\end{equation*}
In particular,~$L^*m =0$. Note that, for any~$q >\frac{d}{d-1}$, 
\begin{equation*}
m \not \in L^q(B_1) \quad \iff \quad  \frac{\lambda}{\Lambda} < 1 - \frac{d}{q(d-1)} \,.
\end{equation*}
Now, this coefficient field~$\a(x)$ that we have defined in~\eqref{e.a.example} is not continuous at the origin. However, this can be fixed by mollifying~$\a$ very slightly, to make it qualitatively continuous. This will only change~$m$ slightly, and by slowly removing the mollification we can get an example satisfying~\eqref{e.mKm} for any choice of~$K$. 
(The positive homogeneity just makes it easier to compute.) 
\end{remark}

If we allow ourselves to use the continuity of the coefficients~$\a(x)$ in a quantitative way, then we can improve the integrability of the invariant measure to~$L^q$ for every~$q \in (1,\infty)$. 
This is obtained by an argument following Bauman~\cite{B}, the main difference is that rather than proving~\eqref{e.reverse.Holder} for~$q=1$, we prove it for all~$q\in [1,d)$ allowing us to bootstrap the integrability exponent. 

\begin{proposition}
\label{p.Patty.plus}
Let~$\a: B_{r} \to \mathcal{S}^{d}$ be continuous and satisfy~\eqref{e.unifellip} and~$m \in L^1(B_{2r})$ be a weak solution of the equation
\begin{equation}
\label{e.invmeas.m.eq.2}
L^* m :=-\sum_{i,j=1}^d \partial_{i} \partial_j \bigl( \a_{ij} m \bigr) 
= 0 \quad \mbox{in} \ B_{2r}\,.
\end{equation}
Then
\begin{equation*}
m \in \bigcap_{q\in [1,\infty)} L^{q} (B_r)
\end{equation*}
and, for each~$q\in [1,\infty)$, there exists a constant~$C<\infty$ which depends on~$q$ and the modulus of continuity of~$\a$, in addition to $(d,\lambda,\Lambda)$, such that
\begin{equation}
\label{e.Patty}
\| m \|_{\underline{L}^{q}(B_r)}
\leq
C \| m \|_{\underline{L}^{1}(B_{2r})}
\,.
\end{equation}
\end{proposition}
\begin{proof}
Fix $q\in (1,d)$ and let~$q^* := \frac{dq}{d-q}$. 
We will prove the estimate 
\begin{equation}
\label{e.reverse.Holder}
\| m \|_{\underline{L}^{q^*}(B_r)}
\leq 
C 
\| m \|_{\underline{L}^{q}(B_{2r})}\,,
\end{equation}
for a constant~$C<\infty$ depending on~$(q,d,\lambda,\Lambda)$ in addition to the modulus of continuity of~$\a$.
Arguing by duality, we will demonstrate the bound~\eqref{e.reverse.Holder} by showing that, for every $f\in L^{p_*}(B_{r})$ with $p:= \frac{q}{q-1}$ and~$p_*:= \frac{pd}{d+p} = \frac{q^*}{q^*-1}$,  we have 
\begin{equation}
\label{e.dualize.dat.shit}
\biggl| \fint_{B_{r}} mf  \, dx \biggr|
\leq 
C \| m \|_{\underline{L}^q(B_{2r})}
\| f \|_{\underline{L}^{p_*}(B_{r})} \,.
\end{equation}
To that end, we select~$f\in C^\infty_c(B_r)$ and consider the solution~$u$ of the problem
\begin{equation*}
\left\{
\begin{aligned}
&Lu:= -\tr( \a D^2u) = f  & \mbox{in} & \ B_{2r} \,, \\
& u = 0 & \mbox{on} & \ \partial B_{2r}\,.
\end{aligned}
\right.
\end{equation*}
Since~$p_*\in (1,\infty)$, the classical Calder\'on-Zygmund~$W^{2,p_*}$ estimates (see for instance~\cite[Theorem 9.11]{GT}) imply that~$u\in W^{2,p_*}(B_{2r})$
and we have the estimate
\begin{equation*}
\| D^2 u \|_{\underline{L}^{p_*}(B_{2r})} 
\leq 
C \| f \|_{\underline{L}^{p_*}(B_{2r})}
\leq 
C \| f \|_{\underline{L}^{p_*}(B_{r})}
 \,.
\end{equation*}
By the Sobolev inequality, we deduce that $u\in W^{1,p}(B_{r})$ and, moreover, we have the estimate
\begin{equation*}
\| u \|_{\underline{L}^{p}(B_r)} 
+
r\| D u \|_{\underline{L}^{p}(B_r)} 
\leq
Cr^2 \| f \|_{\underline{L}^{p_*}(B_{r})}
\,.
\end{equation*}
Select a cutoff function~$\zeta \in C^\infty_c(B_{2r})$ which satisfies
\begin{equation*}
\indc_{B_r} \leq \zeta \leq \indc_{B_{3r/2}} 
\quad \mbox{and} \quad 
\| D^k \zeta\|_{L^\infty(B_{2r})}
\leq 
C r^{-k} , \quad \forall k \in\{1,2\}\,.
\end{equation*}
Using~\eqref{e.invmeas.m.eq.2}, the identity~\eqref{e.easyiden} above, and H\"older's inequality, 
we find 
\begin{align*}
\fint_{B_{r}} mf \, dx
&
=
\fint_{B_{2r}}
m \zeta \a_{ij} \partial_i \partial_j u \,dx
\\ & 
= 
- \fint_{B_{2r}} 
\bigl( 
u(x) \a_{ij}(x)\partial_{i}\partial_j \zeta(x)
-2 \a_{ij}(x) \partial_i u(x)\partial_j \zeta(x)
\bigr) m \,dx
\\ & 
\leq 
\bigl( C \| u\|_{\underline{L}^p(B_{2r}) } \| D^2 \zeta \|_{L^\infty(B_{2r})}
+
C \| Du\|_{\underline{L}^p(B_{2r}) } \| D \zeta \|_{L^\infty(B_{2r})}
\bigr) 
\| m \|_{\underline{L}^q(B_{2r})}
\\ & 
\leq 
C\| m \|_{\underline{L}^q(B_{2r})} \| f \|_{\underline{L}^{p_*}(B_{r})}\,.
\end{align*}
By density, we obtain~\eqref{e.dualize.dat.shit} and this completes the proof of~\eqref{e.reverse.Holder}. 

\smallskip

The statement of the proposition now follows from~\eqref{e.reverse.Holder}, Proposition~\ref{p.Guido}, and a simple covering argument. 
\end{proof}

To go beyond~$m\in \cap_{q\in [1,\infty)} L^q$ and obtain local pointwise boundedness of an adjoint solution~$m$, we need to assume more than just the continuity of~$\a$. We need H\"older continuity (or at least a Dini-type modulus) so as to implement a harmonic approximation scheme, and sum over the scales. This is in complete analogy to the classical Schauder estimates. 

\begin{proposition}
\label{p.Holder.continuity}
Let~$\alpha \in (0,1]$ and~$\a: B_{r} \to \mathcal{S}^{d}$ belong to~$C^{0,\alpha}(B_r;\mathcal{S}^d)$. 
Let~$m \in L^1(B_{2r})$ be a weak solution of the equation
\begin{equation}
\label{e.invmeas.m.eq.3}
L^* m :=-\sum_{i,j=1}^d \partial_{i} \partial_j \bigl( \a_{ij} m \bigr) 
= 0 \quad \mbox{in} \ B_{2r}\,.
\end{equation}
Then~$m \in C^{0,\alpha}(B_r)$ and there exists~$C(\alpha,d,\lambda,\Lambda)\in [1, \infty)$ such that
\begin{equation}
\label{e.Patty.2}
\| m \|_{L^\infty(B_r)} 
+
r^{\alpha} \bigl[ m \big]_{C^{0,\alpha} (B_r)}
\leq
Cr^{\al} [\a]_{C^{0,\alpha}(B_r)}
\| m \|_{\underline{L}^{1}(B_{2r})}
\,.
\end{equation}
\end{proposition}
\begin{proof}
By a change of variables, we may suppose without loss of generality that~$\a(0) = I_d$, and also that~$r=1$. 
Let~$\{ \eta_\delta\}_{\delta>0}$ be the standard mollifier; we will be using that it is radial. 
We will argue that
\begin{equation}
\label{e.adjoint.Schauder.wts.0}
\bigl| m\ast \eta_{2\delta}(0) - m\ast \eta_\delta (0) \bigr| 
\leq
C \delta^{\alpha} \bigl[ \a \bigr]_{C^{0,\alpha}(B_{1})} 
\| m \|_{\underline{L}^1(B_{2\delta})}
\,,\quad \forall \delta\in (0,\sfrac18)\,.
\end{equation}
Consider the unique bounded solution~$w_\delta\in C^\infty(\Rd)$ of the problem 
\begin{equation*}
-\Delta w_\delta = \eta_{2\delta} - \eta_\delta \quad \mbox{in} \ \Rd. 
\end{equation*}
By symmetry,~$w_\delta$ is radial. In fact, it vanishes outside of~$B_{2\delta}$ because the elliptic Green function is harmonic away from its pole, and is therefore unchanged by mollification (by a radial mollifier) by the mean value property. Classical Schauder estimates for the Poisson equation yield,\footnote{Alternatively, just write down the Green function formula and read the estimates off.} for every~$\beta \in (0,1]$, 
\begin{equation}
\label{e.wr.estimates}
\| D^2 w_\delta\|_{L^\infty(\Rd)} 
\leq 
\| \eta_{2\delta} - \eta_\delta \|_{L^\infty(\Rd)} 
+
\delta^{\beta} \bigl[ \eta_{2\delta} - \eta_\delta \bigr]_{C^{0,\beta}(\Rd)}
=
C\delta^{-d}\,.
\end{equation}
For each~$\delta < \sfrac12$, we may test~$w_\delta \in C^\infty_c(B_{2\delta}) \subseteq C^\infty_c(B_{1})$ in the weak formulation of the adjoint equation,~\eqref{e.wr.estimates}, to obtain
\begin{align*}
\bigl| m\ast \eta_{2\delta}(0) - m\ast \eta_\delta (0) \bigr| 
&
=
\biggl| \int_{B_{2\delta}}
\bigl( \eta_{2\delta}(x) - \eta_\delta(x) \bigr) m(x) \,dx \biggr|
\\ & 
=
\bigg| 
\int_{B_{2\delta}}
\Delta w_\delta(x) m(x) \,dx \biggr|
\\ & 
= 
\biggl|
\int_{B_{2\delta}}
\bigl( \Delta w_\delta(x) - \tr(\a(x) D^2w_\delta(x) ) \bigr) m(x) \,dx\biggr|
\\ & 
\leq
C 
\| m \|_{L^1(B_{2\delta})}
\| \a - I_d \|_{L^\infty(B_{2\delta})}
\| D^2 w_\delta \|_{L^\infty(B_{2\delta})}
\\ & 
\leq 
C\bigl[ \a \bigr]_{C^{0,\alpha}(B_{1})} 
\delta^{\alpha}
C\delta^{-d} \| m \|_{L^1(B_{2\delta})}
\,.
\end{align*}
This yields~\eqref{e.adjoint.Schauder.wts.0}.

\smallskip

The estimate~\eqref{e.adjoint.Schauder.wts.0} implies the proposition. To see this, note that it implies that~$\{ m \ast\eta_{2^{-n}\delta}(0)\}_{n\in\N}$  is Cauchy and so if $0$ is a Lebesgue point for~$m$, then $m(0) = \lim_{n\to \infty} m \ast\eta_{2^{-n}\delta}(0)$. Otherwise, we may redefine~$m(0)$. Summing~\eqref{e.adjoint.Schauder.wts.0} over the scales, we obtain
\begin{equation*}
\| m - m(0) \|_{\underline{L}^1(B_\delta)} 
\leq C \delta^{\alpha} \bigl[ \a \bigr]_{C^{0,\alpha}(B_{1})} 
\| m \|_{\underline{L}^1(B_{1})}\,, \forall \delta\in (0,\sfrac18)\,.
\end{equation*}
By translation, we obtain 
\begin{equation*}
\sup_{x\in B_1} 
\| m - m(x) \|_{\underline{L}^1(B_\delta(x))} 
\leq C \delta^{\alpha} \bigl[ \a \bigr]_{C^{0,\alpha}(B_{2})} 
\| m \|_{\underline{L}^1(B_{2})}\,, \forall \delta\in (0,\sfrac18)\,.
\end{equation*}
The result now follows from Campanato's characterization of H\"older spaces. 
\end{proof}

Note that higher regularity ($C^1$ and beyond) for the adjoint equation can be obtained in the case that the coefficients~$\a$ are also at least~$C^{1,\alpha}$. But this is classical, because we can write the equation in divergence form and apply Schauder estimates for divergence form equations.

{
\footnotesize
\bibliographystyle{abbrv}
\bibliography{nondivboot}
}

\end{document}